\documentclass[a4paper, 12pt]{amsart}

\usepackage[utf8]{inputenc}
\usepackage{amsmath,amsfonts,amssymb,amsopn,amscd,amsthm}
\usepackage{tikz-cd}
\usepackage{comment}
\usepackage{todonotes}
\usepackage{dsfont}
\usepackage{graphicx}
\usepackage{color}
\usepackage[colorlinks]{hyperref}
\usepackage{epigraph}
\usepackage[left=3.5cm,top=2.5cm,bottom=2cm,right=3cm]{geometry}

\title[Quantum $SL_2$, Infinite curvature and Pitman's 2M-X theorem]
      {Quantum $SL_2$, Infinite curvature and Pitman's 2M-X theorem}

\author{Fran\c cois \textsc{Chapon}}
\address{Institut de math\'ematiques de Toulouse, UMR5219, Universit\'e de Toulouse,  118 route de Narbonne, F-31062 Toulouse Cedex 9, France}
\email{francois.chapon@math.univ-toulouse.fr}

\author{Reda   \textsc{Chhaibi}}
\address{Institut de math\'ematiques de Toulouse, UMR5219, Universit\'e de Toulouse,  118 route de Narbonne, F-31062 Toulouse Cedex 9, France}
\email{reda.chhaibi@math.univ-toulouse.fr}

\date{\today}

\allowdisplaybreaks[4]

\DeclareMathOperator{\End}{End}

\DeclareMathOperator{\hw}{hw}

\DeclareMathOperator{\tr}{tr}
\DeclareMathOperator{\Tr}{Tr}

\DeclareMathOperator{\wt}{wt}

\DeclareMathOperator{\eqlaw}{\stackrel{\Lc}{=}}
\DeclareMathOperator{\Id}{Id}

\DeclareMathOperator{\modulo}{mod}
\DeclareMathOperator{\sgn}{sgn}

\DeclareMathOperator{\Span}{Span}

\DeclareMathOperator{\id}{id}

\DeclareMathOperator{\Argcosh}{Argcosh}

\def\half{\frac{1}{2}}

\def\1{{\mathbf 1}}

\def\N{{\mathbb N}}
\def\H{{\mathbb H}}
\def\Z{{\mathbb Z}}

\def\R{{\mathbb R}}
\def\C{{\mathbb C}}

\def\T{{\mathbb T}}

\def\P{{\mathbb P}}
\def\E{{\mathbb E}}

\def\Ac{{\mathcal A}}
\def\Bc{{\mathcal B}}
\def\Cc{{\mathcal C}}

\def\Fc{{\mathcal F}}
\def\Gc{{\mathcal G}}

\def\Lc{{\mathcal L}}

\def\Nc{{\mathcal N}}
\def\Oc{{\mathcal O}}
\def\Pc{{\mathcal P}}

\def\Rc{{\mathcal R}}

\def\Uc{{\mathcal U}}

\def\afrak{{\mathfrak a}}

\def\bfrak{{\mathfrak b}}
\def\gfrak{{\mathfrak g}}
\def\hfrak{{\mathfrak h}}

\def\nfrak{{\mathfrak n}}

\def\slfrak{{\mathfrak sl}}
\def\glfrak{{\mathfrak gl}}
\def\sufrak{{\mathfrak{su} }}

\def\kbar{{r}}

\newtheorem{thm}{Theorem}[section]
\newtheorem{proposition}[thm]{Proposition}

\newtheorem{question}[thm]{Question}

\newtheorem{definition}[thm]{Definition}
\newtheorem{example}[thm]{Example}
\newtheorem{lemma}[thm]{Lemma}

\newtheorem{rmk}[thm]{Remark}

\numberwithin{equation}{section}
\numberwithin{figure}{section}

\newcommand{\ba}{\begin{array}}
\newcommand{\ea}{\end{array}}
\newcommand{\be}{\begin{equation}}
\newcommand{\ee}{\end{equation}}
\newcommand{\bea}{\begin{eqnarray}}
\newcommand{\eea}{\end{eqnarray}}
\newcommand{\beaa}{\begin{eqnarray*}}
\newcommand{\eeaa}{\end{eqnarray*}}
\newcommand{\ignore}[1]{}
\newcommand{\vertiii}[1]{{\left\vert\kern-0.25ex\left\vert\kern-0.25ex\left\vert #1 
    \right\vert\kern-0.25ex\right\vert\kern-0.25ex\right\vert}}
\begin{document}

\begin{abstract}
The classical theorem by Pitman states that a Brownian motion minus twice its running infimum enjoys the Markov property.

On the one hand, Biane understood that Pitman's theorem is intimately related to the representation theory of the quantum group $\Uc_q\left( \slfrak_2 \right)$, in the so-called crystal regime $q \rightarrow 0$. On the other hand, Bougerol and Jeulin showed the appearance of exactly the same Pitman transform in the infinite curvature limit $r \rightarrow \infty$ of a Brownian motion on the hyperbolic space $\H^3 = SL_2(\C)/SU_2$. This paper aims at understanding this phenomenon by giving a unifying point of view.

In order to do so, we exhibit a presentation $\Uc_q^\hbar\left( \slfrak_2 \right)$ of the Jimbo-Drinfeld quantum group which isolates the role of curvature $r$ and that of the Planck constant $\hbar$. The simple relationship between parameters is $q=e^{-r}$. The semi-classical limits $\hbar \rightarrow 0$ are the Poisson-Lie groups dual to $SL_2(\C)$ with varying curvatures $r \in \R_+$. We also construct classical and quantum random walks, drawing a full picture which includes Biane's quantum walks and the construction of Bougerol-Jeulin. Taking the curvature parameter $r$ to infinity leads indeed to the crystal regime at the level of representation theory ($\hbar>0$) and to the Bougerol-Jeulin construction in the classical world ($\hbar=0$).

All these results are neatly in accordance with the philosophy of Kirillov's orbit method.
\end{abstract}
\renewcommand{\subjclassname}{%
  \textup{2010} Mathematics Subject Classification}

\maketitle
\begin{flushright}
{\it To our teachers, \\ Philippe Biane and Philippe Bougerol. }
\end{flushright}

\medskip
\medskip
\medskip
\hrule
\setcounter{tocdepth}{2}
\tableofcontents
\hrule

\section*{Notation}
$\Lc(X)$ is the probability measure which is the law of a random variable $X$. Equality in law between two random variables $X$ and $Y$ is written $X \eqlaw Y$. If $X = \left( X_t \right)_{t \in \T}$ is a process indexed by $\T$, then its natural filtration is denoted $\Fc^X$. We will only consider $\T = \N$ for discrete time, and $\T = \R_+$ for continuous time.

The Vinogradov symbol $\ll$ is equivalent to the $\Oc$ notation: $f \ll g \Leftrightarrow f = \Oc(g)$. Moreover, if the implicit constant depends on other quantities, they will be indicated by subscripts.

If $V$ is a finite dimensional vector space, then $\Tr: \End(V) \rightarrow \C$ is the usual trace, while $\tr := \frac{\Tr}{\dim V}$ is the normalized trace. $V^*$ is the dual vector space, and the duality pairing is denoted by $\langle \cdot, \cdot \rangle$. Moreover, if $G$ is a group, then its Lie algebra i.e. the tangent space at the identity is written $T_e G = \gfrak$.

Throughout the paper, $\hbar>0$ will denote a positive real number, which will play the role of Planck constant. For any $\Lambda \in \R$, we write $\Lambda^\hbar := \hbar \lfloor \Lambda/\hbar \rfloor \in \hbar \Z$.

\section{Statement of the problem}

In order to state the problem at the end of this section, we start by presenting the related body of work while distilling the necessary geometric and representation-theoretic notions as we progress. To that endeavor, we adopt the informal style of a survey, which will not be faithful nor all-encompassing. Also, certain prior results will be slightly reformulated, in order to reflect a personal point of view and lay the groundwork for this paper.

While we are focused on relating the representation theory of quantum groups and (possibly non-commutative) geometry, our starting point is Pitman's theorem from probability theory \cite{P75}. It will play the role of Ariadne's thread while navigating through the maze created from the interaction of these various fields.

\begin{thm}[Pitman's 2M-X Theorem, Discrete version]
\label{thm:PitmanDiscrete}
Let $\left( X_n ; n \in \N \right)$ be a simple random walk in $\Z$, i.e. increments are independent and
$$ \forall n \in \Z_+, \ 
   \P\left( X_{n+1} - X_n = 1 \right)
 = 1 - \P\left( X_{n+1} - X_n = -1 \right)
 = \half \ .
$$
Then the process $\left( \Lambda_n^{\infty} ; n \in \N \right)$ defined as
$$ \Lambda_n^{\infty} := X_n - 2 \inf_{0 \leq k \leq n} X_k \ $$
is a Markov chain on $\N$ with transition kernel given by $Q$:
\begin{align}
\label{eq:LR_transitions}
Q\left( \lambda, \lambda + 1 \right) = \frac{\lambda+2}{2(\lambda+1)} \ ,
& \quad 
Q\left( \lambda, \lambda - 1 \right) = \frac{\lambda}{2(\lambda+1)} \ .
\end{align}
Moreover, the missing information is stationary and equidistributed in law in the sense that for all $n\in \N$:
$$ \Lc\left( X_n \ | \ \Fc^{\Lambda^\infty}_n, \ \Lambda_n^\infty = \lambda \right) \ 
   = \ 
   \frac{1}{\lambda+1} \sum_{\substack{ -\lambda \leq k \leq \lambda \\ \lambda-k \text{ even }}} \delta_k
   \ . $$
\end{thm}

Pitman's original proof uses the combinatorics of random walks and is formulated in terms of the running maximum instead of the running infimum. Both are equivalent upon replacing $X$ by $-X$, hence the common name of ``Pitman's $2M-X$ Theorem", where the capital letter $M$ stands for ``Maximum". 

From the discrete version, one obtains a Brownian version thanks to a simple application of Donsker's invariance principle and by computing the diffusive rescaling of the Markov kernel $Q$.

\begin{thm}[Pitman's 2M-X Theorem, Continuous version]
\label{thm:PitmanContinuous}
Let $\left( X_t \ ; \ t \in \R_+ \right)$ be a standard Brownian motion. Then the process $\left( \Lambda_t^{\infty} \ ; \ t \in \R_+ \right)$ defined as
$$ \Lambda_t^{\infty} := X_t - 2 \inf_{0 \leq s \leq t} X_s \ $$
is a Bessel 3 process, that is to say it has the same distribution as
$$ \Lambda_t^{0} := \sqrt{X_t^2 + Y_t^2 + Z_t^2} \ , $$
where $\left( X, Y, Z \right)$ is a Euclidean Brownian motion on $\R^3$.

Moreover, the missing information is stationary and equidistributed in law in the sense that:
$$ \Lc\left( X_t \ | \ \Fc^{\Lambda^\infty}_t, \ \Lambda_t^\infty = \lambda \right) \ 
   = \ 
   \frac{1}{2\lambda} \mathds{1}_{[-\lambda, \lambda]}(x) dx \ . $$
\end{thm}

\begin{rmk}
The reader trained in probability theory knows that the Markov property is very fragile and can be easily broken, while $\left( -\inf_{0 \leq s \leq t} X_s \ ; \ t \in \R_+ \right)$ is the archetype of non-Markovian behavior. As such, Pitman's theorem is rather peculiar. It is also very rigid, as
$$ \left( X_t - k \inf_{0 \leq s \leq t} X_s \ ; \ t \in \R_+\right)$$
enjoys the Markov property only for $k=0$, $1$, and $2$ \cite{MO04} ; the latter case being by far the most interesting in our opinion.

\end{rmk}

In fact, direct proofs of Theorem \ref{thm:PitmanContinuous} at the level of continuous-time stochastic processes are available. Jeulin \cite{RY13} has an approach that uses filtration enlargement techniques and \cite{RP81} makes use of intertwinings of Markov kernels. These two proofs led to a flurry of very interesting probabilistic developments. For example, see the essay \cite{N06} for filtration enlargement, and \cite{DM09} for intertwining.

If other proofs and generalizations abound, we want to focus on two specific approaches where the complex group $SL_2(\C)$ plays an important role. The approach by Bougerol and Jeulin \cite{BJ02} is based on a geometric construction while the approach by Biane \cite{B_suq, B_crystals} is based on the representation theory of that group. The goal of this paper is to exhibit a direct relationship between the two, via semi-classical limits. Let us mention the recent paper \cite{AAS19} which is similar in spirit to ours as it analyzes the radial part of Brownian motion on $\H^3$ while stressing the role of Poisson geometry.

The Lie algebra of $SU_2$ is
\begin{align}
\label{def:su2}
\sufrak_2 & := T_e SU_2 = \Span_\R\left( X_\gfrak, Y_\gfrak, Z_\gfrak \right)
\end{align}
where $\left( X_\gfrak, Y_\gfrak, Z_\gfrak \right)$ is the basis of anti-Hermitian matrices:

       $$ X_\gfrak =  \begin{pmatrix} i & 0 \\  0 & -i \end{pmatrix} \ ;
	  Y_\gfrak =  \begin{pmatrix} 0 & i \\  i &  0 \end{pmatrix} \ ;
	  Z_\gfrak =  \begin{pmatrix} 0 & 1 \\ -1 &  0 \end{pmatrix} \ ,
       $$
       which are $i$ times the so-called Pauli matrices.
Our use of the subscript $\cdot_\gfrak$ is due to the fact that the symbols $X$, $Y$ and $Z$ will often be used to refer to other objects. The complexification of $\sufrak_2$ is the Lie algebra of $SL_2(\C)$:
\begin{align}
\label{def:sl2}
\slfrak_2 & := T_e SL_2(\C) = \sufrak_2 \otimes \C = \Span_\C\left( E, F, H \right) \ ,
\end{align}
where:
\begin{align}
\label{def:HEF}
&
   H = \begin{pmatrix} 1 & 0 \\  0 & -1 \end{pmatrix} \ ;
   E = \begin{pmatrix} 0 & 1 \\  0 &  0 \end{pmatrix} \ ;
   F = \begin{pmatrix} 0 & 0 \\  1 &  0 \end{pmatrix} \ .
\end{align}

\subsection{Bougerol and Jeulin's approach via curvature deformation} In the paper \cite{BJ02}, Bougerol and Jeulin take a parameter $r>0$ and consider a left-invariant process $g^r = \left( g_t^r \ ; \ t \geq 0 \right)$ on the symmetric space $\H^3 = SL_2(\C) / SU_2$. Because of the Gram-Schmidt decomposition, we make the identification $\H^3 = SL_2(\C) / SU_2 \approx NA$, where $NA$ is the subgroup of lower triangular matrices with positive diagonals. More precisely:
$$ A := \left\{ 
   \begin{pmatrix}
   a & 0 \\
   0 & a^{-1}
   \end{pmatrix}
   \ | \ a \in \R_+^*
   \right\} \ ,
\quad 
\text{and}
\quad 
   N := \left\{ 
   \begin{pmatrix}
   1 & 0 \\
   b & 1
   \end{pmatrix}
   \ | \ b \in \C
   \right\} \ .
$$
The  corresponding Lie algebras are denoted by $\afrak := T_e A = \R H$ and $\nfrak := T_e N = \R F \oplus \R iF$.

In that identification, the process $g^r$ satisfies the left-invariant stochastic differential equation (SDE for short)
       $$ \forall t \geq 0, \ 
           g_t^r = \begin{pmatrix}
                   \half r dX_t    & 0 \\
                  r( dY_t + i Z_t) & -\half r dX_t
                   \end{pmatrix}
	  \circ g_t^r \ ,$$
where $(X, Y, Z)$ is a standard Euclidean Brownian motion on $\R^3$. Here, the symbol $\circ$ refers to the Stratonovich integration convention. Solving explicitly the SDE yields for all $t \geq 0$:
\begin{align}
\label{eq:BJ_process}
 g_t^r 
 = \begin{pmatrix}
	    e^{\half r X_t} & 0 \\
	    r e^{\half r X_t} \int_0^t e^{-r X_s} d(Y_s+iZ_s) & e^{-\half r X_t}
	\end{pmatrix} \ .
\end{align}

The reader unfamiliar with stochastic integration should see the above equation as a definition for the process $g^r$. More importantly, the parameter $r>0$ should be seen as a curvature parameter. The definitive explanation will be given in section \ref{section:commutativeDiagram} where we will see that we are considering the hyperbolic space $\H^3$ as the space with constant sectional curvature $-\half r^2$. As such, there is no harm in loosely referring to $r$ as curvature. At this stage, let us only mention the following. We have, as $r \rightarrow 0$:
$$  g_t^r 
 = \Id +
   r
   \begin{pmatrix}
	    \half X_t  & 0 \\
	    Y_t + iZ_t & -\half X_t
	\end{pmatrix}
   + o(r)
 =: \Id + r x_t^0 + o(r) \ ,
$$
and thus appears a three dimensional Brownian motion $\left( x_t^0 = \frac{\partial g_t^r}{\partial r} \rvert_{r=0} \ ; \ t \geq 0 \right)$ on $\afrak \oplus \nfrak \approx \R^3$, which is a flat space. Because of Brownian motion's time-scaling properties, rescaling $r$ amounts to speeding up the Brownian motion and hence the associated vector fields. As the process $g^r$ moves more erratically as $r>0$ grows larger, the non-commutativity of the underlying space $NA$ becomes more apparent. One could say that the space increases in curvature, which is a key element in the following result by Bougerol and Jeulin.

Their result holds for all complex semi-simple groups $G$, but in the context of $G=SL_2(\C)$, we have:
\begin{thm}[Bougerol-Jeulin, \cite{BJ02}]
\label{thm:BJ}
Let $\half r \Lambda_t^r$ be the radial part of $g_t^r$, i.e. $\exp\left( r \Lambda_t^r \right)$ is the largest singular value of $g_t^r$ or equivalently that $\Lambda^r \geq 0$ and there exists $\left( k_1, k_2 \right)\colon \R_+ \rightarrow SU_2 \times SU_2$ such that
$$
   g_t^r 
   =
   k_1(t)
   \begin{pmatrix} e^{\half r \Lambda_t^r} & 0 \\ 0 & e^{-\half r \Lambda_t^r} \end{pmatrix}
   k_2(t) \ .
$$
Then, $\Lambda^r$ is a process whose distribution does not depend on $r > 0$. It is explicitly given by:
\begin{align}
\label{eq:BJdynamic}
   \Lambda_t^r = & \frac{1}{r}
   \Argcosh\left[ \half r^2 \left|e^{\half r X_t} \int_0^t e^{-r X_s} (dY_s + i dZ_s) \right|^{2} + \cosh(r X_t) \right] \ ,
\end{align}
where $\Argcosh(x) = \log\left( x + \sqrt{x^2-1} \right)$ is the inverse of $\cosh: \R_+ \rightarrow [1, \infty)$. Moreover, for all $t>0$, we have the limits in probability:
\begin{align}
\label{eq:BJlimits}
\left\{
\begin{array}{ccccc}
\Lambda_t^{r=0}      & := & \P - \lim_{r \rightarrow 0     } \Lambda_t^r & = & \sqrt{X_t^2 + Y_t^2 + Z_t^2} \ ,\\
\Lambda_t^{r=\infty} & := & \P - \lim_{r \rightarrow \infty} \Lambda_t^r & = & X_t - 2 \inf_{0 \leq s \leq t} X_s \ .
\end{array}
\right.
\end{align}
In particular, these processes are both Bessel processes of dimension $3$.
\end{thm}
Because Bougerol and Jeulin treat the general case, for a general complex semi-simple Lie group $G$, extracting the above statement is not a trivial task. As part of our unifying picture, we shall provide a complete proof in the case of $SL_2(\C)$, where the key arguments are simplified while giving a few illuminating computations. The only novelty in our treatment of Theorem \ref{thm:BJ} is in the proof that the law $\Lc(\Lambda^r)$ does not depend on $r>0$. This fact is rather subtle and so is the argument of Bougerol and Jeulin. We give a short argument based on the rigidity of quantum groups and the results developed in this paper. We also reinterpret the argument of Bougerol and Jeulin through the lens of spherical harmonic analysis, thereby showing where the curvature and the rigidity of quantum groups are hidden.

The important remark is that the Pitman transform shows up in infinite (negative) curvature, while the norm process in $\R^3$ appears in flat curvature. The interpretation of the parameter $r$ as curvature is mainly absent from the literature except in the very astute remark in the final paragraphs of \cite[Section 1]{BJ02}.

\subsection{Kirillov's orbit method}
\label{subsection:orbitMethod}

The correct general framework to understand our group-theoretic story is Kirillov's orbit method. As explained in \cite{Kir99}, the orbit method is more of a philosophy, with merits and demerits. For a Lie group $G$ with Lie algebra $\gfrak = T_e G$, it is well-known that the study of representation theory for $G$ is equivalent to its local version, that is the study of the representations of the Lie algebra $\gfrak$. Equivalently, one prefers to work with the universal enveloping algebra which is defined as
$$ \Uc^\hbar\left( \gfrak \right)
 = T\left( \gfrak \right)
   /
   \left\{ x \otimes y - y \otimes x - \hbar[x, y] \right\} \ ,$$
where $T(\gfrak)$ is the tensor algebra. The fundamental idea behind the orbit method is that the representation theory $G$ should be seen as the quantization of a certain Poisson manifold. Already, one sees that the algebra $\Uc^\hbar\left( \gfrak \right)$ degenerates as $\hbar \rightarrow 0$ to the algebra $S\left( \gfrak \right)$ of symmetric tensors, which is canonically identified with $\C\left[ \gfrak^* \right]$, the algebra of polynomials on $\gfrak^*$. There are two interesting structures on $\C\left[ \gfrak^* \right]$, whose combination is referred to as the trivial Poisson-Lie structure on $\gfrak^*$. Basically, we are only saying that $\gfrak^*$ has to be seen as a flat Poisson manifold, once endowed with the canonical Kirillov-Kostant-Souriau (KKS) Poisson bracket. The formal definition is as follows, which will fix the notations for later use.
 
On the one hand, let $\Cc^\infty(\gfrak^*)$ be the algebra of smooth functions on $\gfrak^*$ and we have the inclusion of sub-algebras $\C\left[ \gfrak^* \right] \hookrightarrow \Cc^\infty(\gfrak^*)$. As a semi-classical limit, $\Cc^\infty(\gfrak^*)$ becomes a Poisson algebra once endowed with the KKS bracket $\{ \cdot, \cdot \}_0: \Cc^\infty(\gfrak^*) \times \Cc^\infty(\gfrak^*) \rightarrow \Cc^\infty(\gfrak^*)$. By definition, a Poisson bracket is a derivation in both variables. Therefore, because of the Leibniz rule, the Poisson bracket is entirely determined by its values on linear functions:
$$ \forall X \in \gfrak \approx (\gfrak^*)^*, \ f_X(\cdot) := \langle X, \cdot \rangle \ .$$
On linear forms, the KKS bracket is defined as:
\begin{align}
\label{def:kks}
\{ f_X, f_Y \}_{0} & := f_{[X,Y]} = \langle [X,Y], \cdot \rangle \ .
\end{align}

On the other hand, recall that if $\left( \Gc, *_\Gc \right)$ is a group, then the group law $*_\Gc$ can be encoded thanks to a coproduct on algebras of functions. By definition, the coproduct $\Delta$ associated to $\left( \Gc, *_\Gc \right)$ is the map:
\begin{align}
\label{def:coproduct}
\begin{array}{cccc}
 \Delta: & \Cc^\infty\left( \Gc \right) & \rightarrow & \Cc^\infty\left( \Gc \times \Gc \right) \\
         & f & \mapsto &
         \left( (g_1, g_2) \mapsto f( g_1 *_\Gc g_2 ) \right) 
\end{array}
\ .
\end{align}
Since $\Delta$ is a morphism of algebras, it needs to be specified on generators only. If $\Ac \subset \Cc^\infty(\Gc)$ is a dense sub-algebra such that for all $f \in \Ac$, $\Delta(f)$ is a separable function, which is written in Sweedler's notation:
$$ \Delta(f)(g_1, g_2) = \sum_{(f)} f_1(g_1) f_2(g_2) \ ,$$
then we can actually write $\Delta: \Ac \rightarrow \Ac \otimes \Ac$. This is the customary choice in order to work algebraically. Here, consider $\left( \gfrak^*, + \right)$ to be {\it an Abelian group}, which amounts to the trivial coproduct $\Delta_0$ defined on linear functions $X \in \gfrak \approx (\gfrak^*)^*$ via:
\begin{align}
\label{def:trivialCoproduct}
\begin{array}{cccc}
 \Delta_0: & \C[ \gfrak^* ] & \rightarrow & \C[ \gfrak^* ] \otimes \C[ \gfrak^* ] \\
           & X & \mapsto &
             X \otimes 1 + 1 \otimes X 
\end{array}
\ .
\end{align}

We give the following trivial, yet key, example. This is the setting of Euclidean Brownian motion on $\R^3 \approx \afrak \oplus \nfrak$.
\begin{example}
\label{example:coProductR3}
Consider the Abelian group $\left( \R^3, + \right) = \left( \afrak \oplus \nfrak, + \right)$. Its coordinate algebra is the polynomial algebra in three variables $\C[X, Y, Z]$, and $(X,Y,Z)$ are linear forms on $\R^3$.

We have for $f \in \{X, Y, Z\}$:
$$ \Delta_0(f)\left(
   \begin{pmatrix} x_1 \\ y_1 \\ z_1 \end{pmatrix}, \ 
   \begin{pmatrix} x_2 \\ y_2 \\ z_2 \end{pmatrix} \right)
   \overset{\text{Eq. } \eqref{def:coproduct}}{=}
   f\left(
   \begin{pmatrix} x_1 + x_2 \\ y_1 + y_2 \\ z_1 + z_2 \end{pmatrix}
   \right)
   =
   f\left(
   \begin{pmatrix} x_1 \\ y_1 \\ z_1 \end{pmatrix}
   \right)
   +
   f\left(
   \begin{pmatrix} x_2 \\ y_2 \\ z_2 \end{pmatrix}
   \right) \ ,
$$
which is more compactly written $\Delta_0(f) = f \otimes 1 + 1 \otimes f$, with the convention that tensors with index $i=1,2$ are functions of the $i$-th variable. Since $\Delta_0$ is a morphism of algebras, it is entirely determined by its values on generators $X,Y,Z$.
\end{example}

The simplest illustration of the orbit method is the Heisenberg Lie algebra, which is the Lie algebra with two generators $x$ and $p$, along with the commutation relation $[x, p] = \hbar \id$. The orbit method morally says that ``the quantum mechanics of one particle on the real line is the representation theory of the Heisenberg algebra". 

In the end, implementing the orbit method consists in drawing correspondences between the two worlds: Unitary representations should correspond to orbits, characters should correspond to orbital integrals, tensor products should correspond to convolutions of orbital measures, etc\dots\ For an extensive dictionary, we refer again to \cite{Kir99}. In that sense, some of our results will be implementations of the orbit method and the general philosophy will be our inspiration.

Nevertheless, for the purposes of this paper, we will only consider $SL_2(\C)$ and groups related to it.


\subsection{Biane's quantum walks}
\label{subsection:bianeQW}

In accordance with the orbit method, Biane \cite{B91} considers $\Uc(\slfrak_2)=\Uc^{\hbar=1}(\slfrak_2)$ as an algebra of observables for a non-commutative probability space. This space is the quantization of $\sufrak_2^* \approx \R^3$ endowed with the KKS structure, and as such, one should think of a three dimensional space where we cannot measure the directions $(X, Y, Z)$ independently. The measurement operators $(X_\gfrak, Y_\gfrak, Z_\gfrak)$ do not commute, with exactly the relations given by the Lie bracket of $\sufrak_2$. Although it is never stated explicitly, studying quantum systems in this space is equivalent to the representation theory of $\sufrak_2$.

In a sense, Biane's work gives a dynamical flavor to Kirillov's orbit method. More precisely, Biane considers the following quantum dynamical system or quantum random walk using the framework of non-commutative probability. For a comprehensive survey, we recommend \cite{B08}. For the purposes of this paper, we want to push for the idea that one needs to separate the geometry of the underlying space and measurement operators. Biane's construction is about quantum mechanics on the flat geometric space $\sufrak_2^* \approx \R^3$, or equivalently non-commutative probability on the Abelian Lie group $\sufrak_2^* \approx \R^3$. The matrix mechanics for quantum measurements are controlled by the commutation relations of the Lie algebra $\sufrak_2$.

\medskip

\paragraph{\bf Notions of non-commutative probability:} If classical probability theory uses the $\C$-algebra of random variables $L^{\infty-}( \Omega ) := \cap_{p \geq 1} L^p( \Omega )$ endowed with the linear form $\E$, non-commutative probability relies on a possibly non-commutative involutive unital $\C$-algebra $\Ac$ endowed with a state $\tau\colon \Ac \rightarrow \C$. In order to distinguish with the more fundamental duality between a vector space $V$ and its dual $V^*$, we write $\dagger$ for the $\C$-anti-linear involution on $\Ac$. A state is a normalized positive linear form i.e. $\tau(a a^\dagger) \geq 0$ for all $a \in \Ac$. It plays the role of expectation. The elements of $\Ac$ are naturally called non-commutative random variables. If $\left( a_i \right)_{i \in I}$ is a family of non-commutative random variables, then their joint distribution is defined as the collection of non-commutative moments:
\begin{align}
\label{def:daggerMoments}
  \left( \tau\left( a_{i_1}^{\varepsilon_1} a_{i_2}^{\varepsilon_2} \dots a_{i_k}^{\varepsilon_k} \right) \ ; \
   k \in \N \ \text{and} \ 
   \forall j=1, \dots, k, \
   (i_j,\varepsilon_j) \in I \times \{1, \dagger\}
   \right) \ .
\end{align}
Convergence in distribution is defined as the convergence in non-commutative moments.   

Now, as a guiding example, let us construct explicitly the probability space underlying a random walk in $\R^d$ - with say, $d=3$ as in Example \ref{example:coProductR3} - and independent identically distributed increments sharing a common distribution. The independent increments are defined on the classical probability space $\Omega = \left( \R^d \right)^{\N}$ endowed with an infinite product measure $\P$. In order to have a dual point of view, we need to write everything in terms of functions, which are referred to as observables. Since an infinite product space is in fact a projective limit of spaces, the dual notion will be an inductive limit of functions. Given the natural inclusion $\Fc\left( \R^d \right)^{\otimes n} \hookrightarrow \Fc\left( \R^d \right)^{\otimes (n+1)}$, one realizes that a convenient algebra of functions is
$$ \Fc( \Omega ) = \varinjlim_{n} \Fc\left( \R^d \right)^{\otimes n} \ .$$
This is the inductive limit of polynomial functions depending on a finite number of increments. It is nothing but the polynomial algebra in infinitely variables and, if increments are bounded, we have the natural inclusion $\Fc( \Omega ) \hookrightarrow L^\infty(\Omega, \P)$. 

In the non-commutative setting, the analogue of the algebra of observables depending on a single increment is $\Uc(\slfrak_2)$. From the previous discussion, it is natural to consider:
$$ \Ac := \varinjlim_{n} \Uc(\slfrak_2)^{\otimes n} $$
as the algebra of all observables. The $\dagger$ involution makes the elements in $\sufrak_2$ self-adjoint. The natural inclusion $\Uc(\slfrak_2)^{\otimes n} \hookrightarrow \Uc(\slfrak_2)^{\otimes (n+1)}$ is $x \mapsto x \otimes 1$. In particular, in $\Ac$, we identify $x_1 \otimes \dots \otimes x_k$ with $x_1 \otimes \dots \otimes x_k \otimes 1^\infty$. As a state $\tau$, we take a product state on pure tensors i.e. for all $x_1 \otimes \dots \otimes x_k \in \Uc(\slfrak_2)^{\otimes k}$:
$$ \tau(x_1 \otimes \dots \otimes x_k ) = \prod_{i=1}^k \tau(x_i) \ ,$$
and for every single elementary observable $x \in \Uc(\slfrak_2)$, we have:
$$ \tau(x) = \tr \rho_1(x) \ ,$$
where $\rho_1: \Uc(\slfrak_2) \rightarrow \End(\C^2)$ is the natural representation via \eqref{def:HEF}. Recall that $\rho_1$ is the representation with highest weight $\lambda = 1$, and could easily be replaced by another representation.

Also, for the sake of simpler exposition, we do not detail the matters of completion throughout the paper. Indeed, at this point, $\Ac$ is the non-commutative analogue of a polynomial algebra. In order to have functional calculus available and carry certain analytic arguments, $\Ac$ needs to be completed into a Von Neumann algebra. We refer to Appendix \ref{appendix:topologicalConsiderations} for tying up such loose ends.

\medskip

\paragraph{\bf Random walks:} Now, the crucial point is that $\Uc(\slfrak_2)$ is endowed with a coproduct $\Delta_0: \Uc(\slfrak_2) \rightarrow \Uc(\slfrak_2) \otimes \Uc(\slfrak_2)$ that is exactly the same on $\slfrak_2$ as the trivial coproduct \eqref{def:trivialCoproduct}. This allows to construct a random walk whose algebra of observables is not commutative or quantum random walk for short. We stress that the underlying space has to be seen as {\it an Abelian group} because of the choice of the coproduct, but it is the algebra of observables that is not commutative. In this construction, we consider measurement operators, which measure for every time $n \in \N$, the observable $x \in \Uc(\slfrak_2)$ applied to the quantum random walk. As such, one exhibits a morphism of algebras $M_n: \Uc(\slfrak_2) \rightarrow \Ac$ as follows:
\begin{align}
\label{eq:def_measurementBiane}
   &
   \left\{
   \begin{array}{ccl}
   M_1 &  = & 1 \ , \\
   M_n &  =  & \left( M_{n-1} \otimes 1 \right) \circ \Delta_0, \ \text{ for $n\geq2$} \ .
   \end{array}
   \right.
\end{align}

We extend the definition from $n \in \N$ to $t \in \R_+$ by defining:
\begin{align}
\label{eq:timeConvention}
\forall t \in \R_+, \ M_t := & M_{\lfloor t \rfloor} \ .
\end{align}

Because $M_n$ is a morphism of algebras, the operators $\left( M_n(x) \ ; \ x \in \Uc(\slfrak_2) \right)$ have the same commutation relations as the enveloping algebras, and hence are truly quantum observables, on the non-commutative space $\Uc(\slfrak_2)$ which is the quantization of $\sufrak_2^* \approx \R^3$. The discrepancy between $\slfrak_2 = \sufrak_2 \otimes \C$ and $\sufrak_2$ is due to an implicit complexification, which will be explained upon discussing real forms.

\begin{thm}[Biane \cite{B_suq}]
\label{thm:Biane}
Consider $C_\gfrak := \sqrt{\half + X_\gfrak^2 + Y_\gfrak^2 + Z_\gfrak^2}$ to be the Casimir operator associated to $\Uc(\slfrak_2)$. Then define for $n \in \N$:
\begin{align}
\label{eq:quantumFlatProcesses}
   &
   \left\{
   \begin{array}{ccc}
   \left( X_n, Y_n, Z_n \right)
   & := 
   & \left( M_n(X_\gfrak), M_n(Y_\gfrak), M_n(Z_\gfrak) \right) \ ,\\
   \Lambda_n & := & M_n(C_\gfrak) = \sqrt{\half + X_n^2 + Y_n^2 + Z_n^2} \ .
   \end{array}   
   \right.
\end{align}
The triple $\left( (X_n, Y_n, Z_n) \ ; \ n \in \N \right)$ is a non-commutative process, with each coordinate being a simple random walk. Furthermore, the pair $\left( X, \Lambda \right)$ is a quantum Markov chain on $\Uc(\slfrak_2)$, with a classical transition operator on the state space:
$$ \left\{ (\omega, \lambda) \in \Z \times \N \ | \ 
           \omega \in \{-\lambda, \lambda+2, \dots, \lambda-2, \lambda\}
   \right\} $$
and with transitions:
\begin{align}
\label{eq:XL_transitions}
 p\left( (\omega,\lambda), (\omega+1,\lambda+1) \right) = \frac{\lambda+\omega+2}{2(\lambda+1)},
 & \quad
 p\left( (\omega,\lambda), (\omega-1,\lambda+1) \right) = \frac{\lambda-\omega+2}{2(\lambda+1)}, 
 \\
 p\left( (\omega,\lambda), (\omega+1,\lambda-1) \right) = \frac{\lambda-\omega}{2(\lambda+1)}, 
 & \quad
 p\left( (\omega,\lambda), (\omega-1,\lambda-1) \right) = \frac{\lambda+\omega}{2(\lambda+1)} \ .
 \nonumber
\end{align}
In particular, $X$ is a simple random walk, while $\Lambda$ follows the same transitions as \eqref{eq:LR_transitions}.
\end{thm}
If the pair $\left( X, \Lambda \right)$ has coordinate-wise exactly the same dynamic as in Pitman's theorem, the joint dynamic is different: $X$ is a simple random walk while $\Lambda$ is the quantized analogue of a Euclidean norm - not the Pitman transform of $X$! This is even more apparent upon taking the following semi-classical limit.

\begin{thm}[Biane]
\label{thm:Biane_CV}
In the sense of non-commutative moments, we have the convergence in law:
$$ \left( \hbar M_{t/\hbar^2}(X_\gfrak), \hbar M_{t/\hbar^2}(Y_\gfrak), \hbar M_{t/\hbar^2}(Z_\gfrak) \ ; \ t \geq 0\right) 
   \overset{\hbar \rightarrow 0}{\longrightarrow}
   \left( \left( X_t, Y_t, Z_t \right) ; t \geq 0 \right) \ ,
$$
where $\left( X, Y, Z \right)$ is a Euclidean Brownian motion on $\R^3$. Moreover, jointly with the above convergence $\hbar M_{\cdot/\hbar^2}(C_\gfrak)$ converges to the Euclidean norm $\sqrt{X^2+Y^2+Z^2}$ which is a Bessel-3 process.
\end{thm}

In order to see the relation with the approach of Bougerol and Jeulin, we make the double identification $\sufrak_2^* \approx \R^3 \approx \nfrak \oplus \afrak$ and reformulate the above result as the convergence of quantum observables to classical observables applied to the standard Brownian motion on $\nfrak \oplus \afrak$:
\begin{align*}
   & \left( \hbar M_{t/\hbar^2}(\Fc) \ ; \ \Fc \in \Uc\left( \slfrak_2 \right)
     , \ t \geq 0 \right)\\
   \overset{\hbar \rightarrow 0}{\longrightarrow} &
   \left( 
   f(x_t^0) = f\left(
   \begin{pmatrix}
   \half X_t     & 0 \\
   Y_t + i Z_t   & -\half X_t
   \end{pmatrix} \right) \ ; \ f = \pi(\Fc), \ t \geq 0 \right) \ .
\end{align*}
Here $\pi: \Uc\left( \slfrak_2 \right) \approx \Uc^\hbar\left( \slfrak_2 \right) \rightarrow \C[\sufrak_2^*] \approx \C[\nfrak \oplus \afrak]$ is the quotient map $\modulo \hbar$, which consists in seeing any non-commutative monomial as a commutative one.

\subsection{Quantum groups and crystals}
The mismatch between Pitman's Theorem \ref{thm:PitmanContinuous} and the norm process appearing in the previous Theorem \ref{thm:Biane_CV} is fixed upon considering quantum groups. The classical presentation of a quantum group is \cite[Example 3.2.1]{Majid00}:
\begin{align}
\label{eq:Uq_pres}
   \Uc_q\left( \slfrak_2 \right)
&
:= \langle K^{\half}, K^{-\half}, E, F \rangle / \mathcal{R} \ ,
\end{align}
where $K=q^{H}$, $q=e^{h}$, and $\mathcal{R}$ is the two-sided ideal generated by the relations:
\begin{equation}
\label{eq:Uq_relations}
   K^\half E K^{-\half} = q E, \ 
   K^\half F K^{-\half} = q F, \ 
   EF - FE = \frac{K-K^{-1}}{q-q^{-1}} . \ 
\end{equation}
Here, $h$ should not be seen as the actual Planck constant. It is a deformation parameter such that formally ``$\Uc_q\left( \slfrak_2 \right) \rightarrow \Uc(\slfrak_2)$" as $h \rightarrow 0$. This can be seen from Taylor expanding the relations up to order $1$. For example, upon writing $q = 1 + h + o(h)$, the first relation in \eqref{eq:Uq_relations} becomes:
$$ E + \half h [H,E] + o(h) = E + h E + o(h) \ ,$$
and therefore one recovers the classical commutation relation $[H, E] = 2E$ in $\slfrak_2$. 

As $q \rightarrow 0$, the algebra structure breaks down but a combinatorial structure called  crystals remains at the level of the representation theory. In \cite{B_crystals}, Biane understood that it is the combinatorics of crystals that is lurking behind Pitman's theorem. The generalization which consists in tensoring by other representations and in general Lie type is developed in \cite{LLP} and \cite{LLP2}. In fact, they revisited the works of \cite{BBO} and \cite{BBO2} where continuous crystals were directly constructed. The following statement is extracted from \cite{B_crystals}. As the Main Theorem \ref{thm:main} will demonstrate, it should be seen as a quantized version of Bougerol and Jeulin's Theorem \ref{thm:BJ}:

\begin{thm}
\label{thm:Biane_Pitman}
There is a quantum Markov chain $(X, \Lambda)$ on $\Uc_q(\slfrak_2)$, with a classical transition operator given by:
\begin{align}
\label{eq:XL_qtransitions}
  p\left( (\omega,\lambda), (\omega+1,\lambda+1) \right)
= \frac{q^{\lambda-\omega}-q^{2(\lambda+1)}}{2 (1-q^{2(\lambda+1)}) },
  & \quad
  p\left( (\omega,\lambda), (\omega-1,\lambda+1) \right)
= \frac{1-q^{\lambda-\omega+2}}{2 (1-q^{2(\lambda+1)}) } \ , \\
  p\left( (\omega,\lambda), (\omega+1,\lambda-1) \right)
= \frac{1-q^{\lambda-\omega}}{2 (1-q^{2(\lambda+1)}) } \ ,
  & \quad
    p\left( (\omega,\lambda), (\omega-1,\lambda-1) \right)
= \frac{q^{\lambda-\omega+2}-q^{2(\lambda+1)}}{2 (1-q^{2(\lambda+1)}) } \ ,
 \nonumber
\end{align}
with the convention that $0^0 = 1$. In particular, the limit $q \rightarrow 1$ coincides with the result of the previous section, while $q \rightarrow 0$ coincides with Pitman's theorem. 
\end{thm}

Indeed, upon computing the transition probabilities as $q \rightarrow 0$, one realizes that if $\left( X_n \ ; \ n \geq 0 \right)$ is a standard random walk, then $\left( X, \Lambda \right)$ are coupled as follows. One checks that $\Lambda_n = \Pc(X)_n$ where the path transform $\Pc$ is defined on any path via:
$$ \Pc: \ X \mapsto \left( t \mapsto X_t - 2 \inf_{0 \leq s \leq t} X_s \right) \ .$$

We conclude this subsection by stating that Pitman's theorem, in its discrete version, has to do with quantum random walks on $\Uc_q\left( \slfrak_2 \right)$ and taking $q$ from $q=1$ to $q=0$, where crystals do appear. In fact, everything can be conveniently recast in terms of the Littelmann path model \cite{LittICM95, Litt95}, which is a combinatorial model for crystals. The random walks at hand are readily identified with crystal elements. For an overview, see the introduction of one of the author's PhD thesis \cite{C13}.

\medskip

We are ready to state the problem that is addressed in the paper:

\begin{question}
\label{question:main}
If the Pitman transform $\Pc$ is intimately related to crystals, appearing at the level of the representation theory of $\Uc_q( \slfrak_2 )$ at $q=0$, why does it also appear in the geometric context of Bougerol and Jeulin?

Why would there be crystal-like phenomenons by taking curvature to infinity ($r \rightarrow \infty$) in a symmetric space $\H^3 = SL_2(\C)/SU_2 \approx NA$?
\end{question}

It is certainly desirable to have single global picture, with an interplay between both the representation theory of $\Uc_q(\slfrak_2)$, as $q>0$ varies, and the geometry of the symmetric space $\H^3 = SL_2(\C) / SU_2$ with varying curvatures $r>0$. Such a unifying point of view should also extend to dynamics, by relating Biane's quantum random walks and the dynamic of Bougerol-Jeulin on $\H^3$.

\section{Statement of the main result}

At this point, let us summarize the landscape:
\begin{itemize}
 \item On the one hand, at $q=1$, there is Biane's construction of quantum random walks \cite{B91}. The diffusive limit is Brownian motion on the space $\sufrak_2^*$, which can be seen as a flat space with zero curvature ($r=0$).
 \item On the other hand, at $q=0$, using Kashiwara crystals, for example in the path model form, one recovers Pitman's theorem. The latter is also recovered upon taking a Brownian motion on the symmetric space $\H^3 = SL_2(\C)/SU_2$ and taking the curvature to infinity ($r \rightarrow \infty$).
\end{itemize}

Thus, we want to interpolate the two different regimes, and perhaps reinterpret the parameter $q$ in quantum groups as a curvature parameter. The most fruitful idea in trying to answer Question \ref{question:main} is to discard the idea that $q = e^h$ in the Drinfeld-Jimbo quantum group $\Uc_q\left( \slfrak_2 \right)$, with $h$ being a Planck constant. The following conversation will take us back to the genesis of quantum groups, which we feel is necessary in order to really distinguish what is quantum and what is not. We begin by introducing two important ingredients $\Uc_q^\hbar(\slfrak_2)$ and $\C\left[ (SU_2^*)_r \right]$. These are tailored so that the formal diagram in Figure \ref{fig:commutativeDiag} commutes.

\begin{figure}[ht!]
\begin{tikzpicture}[align=center,node distance=1.25cm and 1.25cm, auto]
\node at (-3.5,0) (A) {$\Uc_{q}^\hbar(\slfrak_2)$ };
\node (B) [right =of A] {$\C\left[ (SU_2^*)_r \right]$} ;
\node (C) [below =of A] {$\Uc^\hbar(\slfrak_2) $};
\node (D) [below =of B] {$\C\left[ \sufrak_2^* \right]$};
\draw[->] (A) [above] to node {\tiny $\hbar \rightarrow 0$}(B);
\draw[->] (A) [left] to node {\tiny $r \rightarrow 0$} (C);
\draw[->] (-2.7,-2) -- (-1.1,-2) ; \draw (-2,-2.5) node[above]  {\tiny $\hbar \rightarrow 0$} ; 
\draw[->] (B) to node {\tiny $r \rightarrow 0$} (D);
\draw[,dash pattern={on 5pt off 6pt}] (-2.0,1.0) to (-2.0,-3.3);
\draw[,dash pattern={on 5pt off 6pt}] (-7,-0.7) to (3.2,-0.7);
\node at (-5.0,1.5) {\footnotesize Quantum mechanics, \\
                     \footnotesize Representation theory
                    };
\node at (1.5,1.5)    {\footnotesize Semiclassical limit,\\
					 \footnotesize Poisson geometry
                    };
\node at (5,0.25)  {\footnotesize Curved setting};
\node at (5,-1.5)  {\footnotesize Flat setting};
\end{tikzpicture}
\caption{A formal commutative diagram}
\label{fig:commutativeDiag}
\end{figure}
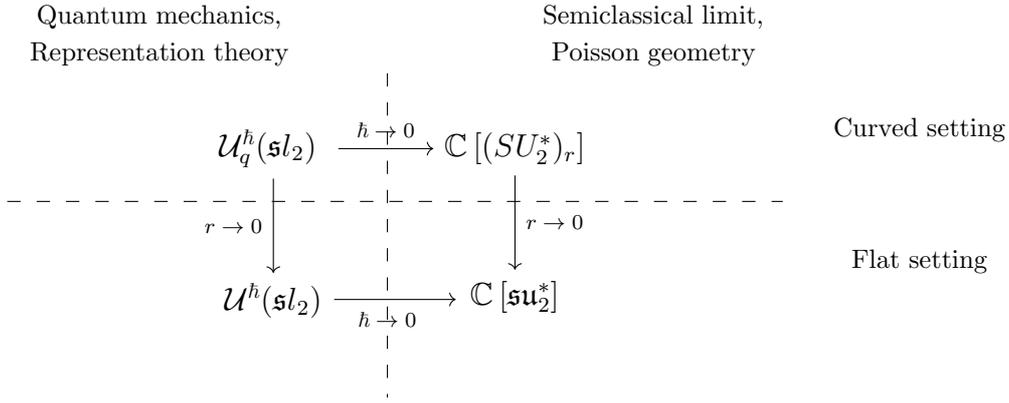

Quoting Kirillov \cite[p.305]{Survey05}, who attributes the statement to Drinfeld, the first approximation to quantum groups as classical objects are Poisson-Lie groups. This leads us to the first ingredient, that is a family of Poisson-Lie groups $(SU_2^*)_r$ with varying curvatures $r > 0$. $\C\left[ (SU_2^*)_r \right]$ will denote the coordinate algebra. In order for such an object to appear as a semi-classical limit, we have to revisit the presentation given in \eqref{eq:Uq_pres}. We require a different presentation $\Uc_q^\hbar(\slfrak_2)$ of the Jimbo-Drinfeld quantum group with two parameters $\hbar>0$ and $q=e^{-r}$.

Again, as mentioned just before Theorem \ref{thm:Biane}, one notices the discrepancy between $\slfrak_2 = \sufrak_2 \otimes \C$ in the quantum picture and $\sufrak_2$ in the semiclassical picture. Now, let us define $\Uc_q^\hbar(\slfrak_2)$ and $\C\left[ (SU_2^*)_r \right]$.

\subsection{Definitions}

\subsubsection{A different presentation of the Drinfeld-Jimbo quantum group}
We define $\Uc_q^\hbar\left( \slfrak_2 \right)$ with $q = e^{-r}$ as follows. As explained before, $r>0$ has to be understood as curvature and $\hbar>0$ is the actual Planck constant. We set
\begin{align}
\label{def:heteroQuantumGroup}
\Uc_q^\hbar\left( \slfrak_2 \right) := \langle K^\half, K^{-\half}, E, F \rangle / \Rc
\end{align}
where this time $K^\half = q^{\half H} = e^{-\half rH}$ and $\Rc$ is the two-sided ideal  generated by the relations:
\begin{equation}
\label{eq:Uqh_relations}
K^\half E K^{-\half} = q^{  \hbar} E \ ,  \quad
K^\half F K^{-\half} = q^{- \hbar} F \ ,  \quad
EF - FE = \hbar \frac{K^{-1}-K}{2r} = \hbar \frac{e^{rH}-e^{-rH}}{2r}\ .
\end{equation}

Furthermore, $\Uc_q^\hbar\left( \slfrak_2 \right)$ is a Hopf algebra once endowed with the co-product $\Delta_r\colon \Uc_q^\hbar\left( \slfrak_2 \right) \rightarrow \Uc_q^\hbar\left( \slfrak_2 \right) \otimes \Uc_q^\hbar\left( \slfrak_2 \right)$:
\begin{align}
\label{def:Delta_r}
   &
   \left\{
   \begin{array}{ccc}
   \Delta_r\left( e^{\half r H} \right) & = & e^{\half r H} \otimes e^{\half r H} \ ,\\
   \Delta_r\left( F \right) & = & F \otimes e^{ \half r H} + e^{-\half r H} \otimes F \ , \\
   \Delta_r\left( E \right) & = & E \otimes e^{ \half r H} + e^{-\half r H} \otimes E \ , 
   \end{array}
   \right.
\end{align}
while the antipode and counit maps $S^\hbar_r, \varepsilon_r: \Uc_q^\hbar\left( \slfrak_2 \right) \rightarrow \Uc_q^\hbar\left( \slfrak_2 \right)$ are given by:
\begin{align}
\label{eq:Uq_S}
& S^\hbar_r\left( e^{\pm \half r H} \right) = e^{\mp \half r H},  \ S^\hbar_r(E) = - q^{\hbar} E, \ S^\hbar_r(F) = - q^{-\hbar} F ,
\end{align}
\begin{align}
\label{eq:Uq_epsilon}
& \varepsilon_r\left( e^{\pm \half r H} \right) = 1,  \ \varepsilon_r(E) = \varepsilon_r(F) = 0 \ .
\end{align}
 It is easy to check that, over $\C$, there is a Hopf algebra isomorphism $\Phi\colon \Uc_{q^\hbar}\left( \slfrak_2 \right) \rightarrow \Uc_q^\hbar\left( \slfrak_2 \right)$, between the classical presentation \eqref{eq:Uq_pres} of Drinfeld-Jimbo and ours \eqref{def:heteroQuantumGroup}, such that:
\begin{align}
\label{def:isom}
\Phi(H) = \frac{H}{\hbar} \ , \quad 
\Phi(E) = E \sqrt{\frac{2r}{\hbar (q^{-\hbar} - q^{\hbar})}} \ , \quad
\Phi(F) = F \sqrt{\frac{2r}{\hbar (q^{-\hbar} - q^{\hbar})}} \ .
\end{align}

Strictly speaking, the first equation holds upon continuously extending $\Phi$ to a completion so that $K = q^{\hbar H} \in \Uc_{q^\hbar}\left( \slfrak_2 \right)$ maps to:
$$ \Phi(K) 
   = \Phi( e^{-r \hbar H} )
   = e^{-rH}
   = K   
   \ .
$$

As such, the usual Casimir element:
\begin{align}
\label{def:usualCasimir}
C^q := EF + \frac{q^{-1}K + q K^{-1}}{(q-q^{-1})^2} \in \Uc_{q}\left( \slfrak_2 \right)
\end{align}
is changed to
$$ C^{q^\hbar} := EF + \frac{q^{-\hbar}K + q^\hbar K^{-1}}{(q^\hbar-q^{-\hbar})^2} \in \Uc_{q^\hbar}\left( \slfrak_2 \right) \ ,$$
which maps via the isomorphism $\Phi$ and a rescaling to
\begin{align}
\label{def:Casimir}
C^{r, \hbar}  &:= 
\ r \hbar \left( q^{-\hbar} - q^{\hbar} \right) \Phi\left( C^{q^\hbar} \right)
= \half \left( 4r^2 EF + \left( q^{-\hbar}K + q^\hbar K^{-1} \right) \frac{2r\hbar}{(q^{-\hbar} - q^{\hbar})} \right) \\
\nonumber
& \, = \half \left( 4r^2 EF + \left( e^{r\hbar}K + e^{-r\hbar} K^{-1} \right) \frac{2r\hbar}{(e^{r\hbar} - e^{-r\hbar})} \right) \ .
\end{align}
Naturally, $C^{r, \hbar}$ generates the center of $\Uc_q^\hbar\left( \slfrak_2 \right)$ by \cite[Theorem VI.4.8]{K12}. We also define the element $\Lambda^{r, \hbar}$ belonging to a completion of $\Uc^\hbar_q \left(\slfrak_2 \right)$ as
\begin{equation}
\label{def:elementLambda}
\Lambda^{r, \hbar} := \frac1r \Argcosh \left( \frac{e^{r\hbar}-e^{-r\hbar}}{2r\hbar} \, C^{r, \hbar} \right) -\hbar \ .
\end{equation}
As we will see in the upcoming subsection \ref{subsection:representations}, the definition of $\Lambda^{r, \hbar}$ has been tailored so that $\Lambda^{r, \hbar}$ acts as the {\it appropriate} constant in any fixed irreducible representation of $\Uc^\hbar_q \left(\slfrak_2 \right)$.

Finally, the analogue of choosing a real form for a Lie algebra in the context of Hopf algebras is exactly the choice of an anti-involution $\dagger$. We recommend the discussion in \cite[Section 1.2.7]{KS12} regarding that matter. Here, the compact real form of $\Uc_q^\hbar\left( \slfrak_2 \right)$ is defined as the pair $\left( \Uc_q^\hbar\left( \slfrak_2 \right),  \dagger \right)$ where $\dagger$ is the algebra anti-involution given by \cite[p.59]{KS12}:
\begin{align}
\label{def:quantum_dagger}
K^\dagger = K, \ \quad E^\dagger = F, \ \quad F^\dagger = E.
\end{align}
This real form is compatible with the real form we shall choose for Poisson-Lie groups.

\subsubsection{Poisson-Lie groups with varying curvatures \texorpdfstring{$r > 0$}{r > 0}}

Consider $B \subset SL_2(\C)$ as the Borel subgroup:
$$ 
   B := \left\{ 
   \begin{pmatrix}
   a & 0 \\
   b & a^{-1}
   \end{pmatrix} \ | \ 
   a \in \C^*, \ 
   b \in \C
   \right\} \ ,
$$
while $B^+$ is the transpose. If $b \in B \cup B^+$, then $[b]_0$ denotes the projection onto the diagonal. The following complex group will play an important role:
\begin{align}
\label{def:poisson_dual_G} 
SL_2^* := & \left\{ (b, b^+) \in B \times B^+ \ | \ [b]_0 = [b^+]^{-1}_0 \right\} \ ,
\end{align}
which is called the Poisson-Lie group dual to $SL_2(\C)$, equipped with the standard structure (see \cite{CP95} or \cite{K97}). The group law is the pair-wise matrix multiplication. Its Lie algebra
\begin{align}
\label{def:poisson_dual_g} 
\slfrak_2^* := & T_e SL_2^* = \bfrak \oplus_{\hfrak} \bfrak^+ \ ,
\end{align}
is made of the two triangular subalgebras $\bfrak = T_e B$ and $\bfrak^+ = T_e B^+$, with the diagonal parts in $\hfrak$ being opposite. Here $\hfrak := \afrak + i \afrak \subset \slfrak_2$ is simply the Abelian subalgebra of complex diagonal matrices.

In order to have varying curvatures $r > 0$ and interpolate with the trivial Poisson-Lie group $\slfrak_2^*$, we define $\left( SL_2^* \right)_r$ as the Lie group with Lie algebra $\left( \slfrak_2^*, \ r [ \cdot, \cdot ]_{\slfrak_2^*} \right)$. This is nothing more than $SL_2^*$ as a space but with a different group law. We define $\C\left[ \left(SL_2^*\right)_r \right]$ as the polynomial algebra generated by the variables $e^{\half r H}$, $e^{-\half r H} = \left( e^{\half r H} \right)^{-1}$, $E$ and $F$:
\begin{align}
\label{def:coordinateAlgebra}
   \C\left[ \left(SL_2^*\right)_r \right]
   :=
   & \ 
   \C\left[ e^{\half r H}, e^{-\half r H}, 
            E, F \right] \ .
\end{align}
In turn, these variables are seen as coordinate functions by writing:
\begin{align}
\label{def:functionsHEF}
   \forall g \in \left(SL_2^*\right)_r, \
   & 
   g = \left(
   \begin{pmatrix}
   e^{\half r H(g)}      & 0 \\
   2r \ F(g)             & e^{-\half r H(g)}
   \end{pmatrix},
   \begin{pmatrix}
   e^{-\half r H(g)} & 2r \ E(g) \\
   0                 & e^{\half r H(g)}
   \end{pmatrix}
   \right) \ .
\end{align}
When convenient, we will drop the dependence in $g$ for $f(g) = f \in \{ H, E, F \}$, as in the definition of the coordinate algebra $\C\left[ \left(SL_2^*\right)_r \right]$.

Now define:
\begin{align}
\label{def:SU2star}
   (SU_2^*)_r := & 
                 \left\{ g\in (SL_2^*)_r \ | \ 
   				 H(g) \in \R, \ E(g)=\overline{F(g)}
                 \right\} \ .
\end{align}
This is clearly a subgroup of  $(SL_2^*)_r$ and we  will see in the next section that it is the Poisson-Lie group dual to $SU_2$, via an involution $\dagger$ which respects the duality at the level of Hopf algebras. Its curvature will also be shown to vary with $r>0$. In view of the definition of the elements of $(SU_2^*)_r$, all the information is contained in the lower Borel subgroup with positive diagonals, leading to a natural identification
$$
\left(SU_2^*\right)_r \approx NA \ .
$$
The corresponding coordinate algebra is naturally denoted $\C[(SU_2^*)_r]$.

There is also the following, more analytic, presentation of $\left(SU_2^*\right)_r$. Notice that the exponential map is a diffeomorphism $\exp: \sufrak_2^* \approx \nfrak \oplus \afrak \xrightarrow{\sim} NA \approx \left( SU_2^* \right)_r$. As such, we can identify $\sufrak_2^*$ and $\left( SU_2^* \right)_r$ as topological spaces. Then we define a group law with a parameter $r > 0$ via:
\begin{align}
\label{eq:groupLaw}
 \forall (X, Y) \in \sufrak_2^* \times \sufrak_2^*,
 \ X *_r Y & := \frac{1}{r} \log\left( e^{r X} e^{r Y} \right) \ .
\end{align}
The new group is denoted by $\left( \left(SU_2^*\right)_r, *_r \right)$ and its Lie bracket is naturally $r \left[ \cdot, \cdot \right]_{\slfrak_2^*}$, i.e. the rescaling of the original bracket by a factor $r > 0$. Clearly, as $r \rightarrow 0$, the group $\left( \left(SU_2^*\right)_r, *_r \right)$ becomes the Abelian group $\left( \sufrak_2^*, + \right)$. 
 
\subsection{Main result}

In order to construct the quantum walk on $\Uc_q^\hbar(\slfrak_2)$, this latter algebra is taken as the algebra of observables for one increment. Thanks to the framework detailed in subsection \ref{subsection:bianeQW}, our algebra of non-commutative random variables is the inductive limit
\begin{align}
\label{def:quantumOmega}
 \Ac^{r,\hbar}
 := & \varinjlim_{n} \left( \Uc_q^\hbar\left( \slfrak_2 \right)\right)^{\otimes n} \ .
\end{align}

Given that $\Uc_q^\hbar(\slfrak_2)$ is isomorphic to the usual Jimbo-Drinfeld quantum group, the representations are essentially the same. The specifics are not needed for now. The state $\tau$ is the product state using the standard representation $\C^2$. The pair $\left( \Ac^{r, \hbar}, \tau \right)$ will be our working non-commutative probability space.

Measurement operators are exactly the same as Eq. \eqref{eq:def_measurementBiane}, except that we have to use the coproduct $\Delta_r$. As such, the morphism of algebras $M_n: \Uc_q^\hbar(\slfrak_2) \rightarrow \Ac^{r, \hbar}$ defined for discrete times $n \in \N$ are as follows. $M_0=\varepsilon_r$ is given by the counit, and
\begin{align}
\label{eq:def_measurement}
   \begin{cases}
   M_1 & =  1 \ , \\
   M_n & =  \left( M_{n-1} \otimes 1 \right) \circ \Delta_r \ , \ \text{for $n\geq2$} .
   \end{cases}
\end{align}
Since we want the random walk to classically start from the identity, and the quantum version consists in expressing everything dually at the level of measurement operators, one sees that $M_0$ has to be taken as the counit. 

The convention \eqref{eq:timeConvention} is still in place. As a random walk on $\Uc_q^\hbar\left( \slfrak_2 \right)$, we define three non-commutative processes via:
$$ \forall t \in \R_+, \ S_t^{r,\hbar} := M_{t/\hbar^2}\left( S \right) $$
for each generator $S \in \{H, E, F\}$ of the quantum group $\Uc_q^\hbar\left( \slfrak_2 \right)$. This three-dimensional non-commutative process is neatly repackaged in matrices of $\left( SU_2^* \right) _r$ with non-commutative entries:
\begin{align}
\label{eq:def_x}
\forall t \in \R_+, \ g_t^{r, \hbar} := & 
\left(
   \begin{pmatrix}
   e^{\half r H_t^{r, \hbar}} & 0 \\
   2r \ F_t^{r, \hbar}        & e^{-\half r H_t^{r, \hbar}}
   \end{pmatrix},
   \begin{pmatrix}
   e^{-\half r H_t^{r, \hbar}} & 2r \ E_t^{r, \hbar} \\
   0                           & e^{\half r H_t^{r, \hbar}}
   \end{pmatrix}
   \right)
  \in \left( SU_2^* \right) _r \otimes \Ac^{r,\hbar}
   \ .
\end{align}
By stating that the above quantity is in $\left( SU_2^* \right) _r \otimes \Ac^{r,\hbar}$, we are implicitly saying that the lower and upper triangular parts are $\dagger$-conjugate. Therefore $g^{r, \hbar}_t$ can be seen as an element in $NA$ with operator-valued entries. The quantum dynamic $\left( \Lambda_t^{r, \hbar} \ ; \ t \geq 0 \right)$ is defined from the measurement of the Casimir element \eqref{def:Casimir} thanks to the explicit expression:
\begin{align}
\label{eq:def_Lambda}
\forall t \in \R_+, \
\frac{2r \hbar}{e^{r\hbar}-e^{-r\hbar}}
\cosh \left( r \hbar + r \Lambda_t^{r, \hbar} \right)
:= \ &
M_{t/\hbar^2}\left( C^{r, \hbar} \right)
\ .
\end{align}
This is equivalent to directly setting $\Lambda_t^{r, \hbar} := M_{t/\hbar^2}\left( \Lambda^{r, \hbar} \right)$ after continuously extending the measurement operators to the completion where $\Lambda^{r, \hbar}$  lives (see Eq. \eqref{def:elementLambda}).

\medskip

We are ready to state the main result of this paper, which unifies the results of Biane on the one hand and Bougerol-Jeulin on the other hand. Since the crystal regime $q = e^{-r} \rightarrow 0$ is tractable in both quantum and semi-classical settings, it explains why crystal-like phenomena appear upon taking infinite curvature limits. This recovers indeed Pitman's $2M-X$ Theorem in the discrete and continuous versions.

\begin{thm}[Main Theorem]
\label{thm:main}
In the sense of (possibly non-commutative) moments, we have the following convergences in law between processes indexed by $t \in \R_+$:
$$
\hspace{-1cm}
\begin{tikzcd}
\begin{array}{c}
\Lambda_t^{\infty,\hbar} = X^\hbar_t - 2\inf_{0\leq s \leq t} X^\hbar_t\\
\text{\small Pitman's Theorem \ref{thm:PitmanDiscrete}}\\
\text{\small (discrete case) }
\end{array} 
\arrow{r}{\hbar \rightarrow 0} 
&
\begin{array}{c}
\Lambda_t^{\infty}= X_t - 2 \inf_{0 \leq s \leq t} X_s \\
\text{\small Pitman's Theorem \ref{thm:PitmanContinuous}}\\
\text{\small (continuous case) }
\end{array} \\
\begin{array}{c}
g_t^{r,\hbar} \in   \left( SU_2^* \right) _r \otimes \Ac^{r,\hbar}    \\
\Lambda_t^{r,\hbar} \\
\text{\small Quantum random walks}\\
\text{\small on $\Uc_q^\hbar(\slfrak_2)$ as in Eq. \eqref{eq:def_x}}
\end{array}
\arrow{u}[swap]{r \rightarrow \infty}
\arrow{d}{r \rightarrow 0}
\arrow{r}{\hbar \rightarrow 0} &
\ 
\begin{array}{c}
g_t^r \in \left( SU_2^* \right) _r \otimes L^{\infty -}(\Omega) \\
\Lambda_t^r = \frac1r \Argcosh \circ \tr \left( g_t^r \left( g_t^r \right)^\dagger   \right) \\
\text{\small Bougerol-Jeulin's convolution dynamic}\\
\text{\small and its radial part as in Theorem \ref{thm:BJ}}
\end{array}
\arrow{u}[swap]{r \rightarrow \infty}
\arrow{d}{r \rightarrow 0}
\\
\ 
\begin{array}{c}
x_t^{0,\hbar}=\begin{pmatrix}
\half X^\hbar_t     & 0 \\
Y_t^\hbar + i Z_t^\hbar   & -\half X_t^\hbar
\end{pmatrix} \in \sufrak_2^* \otimes \Ac^{0,\hbar}\\
\Lambda_t^{0,\hbar} = \sqrt{ \half \hbar^2+ \left (X_t^\hbar\right )^2 + \left(Y_t^\hbar\right )^2 + \left(Z_t^\hbar\right)^2 }\\
\text{\small Biane's quantum random walks}\\
\text{\small  on $\Uc^\hbar(\slfrak_2)$ as in Theorem \ref{thm:Biane} }
\end{array}
\arrow{r}{\hbar \rightarrow 0} & 
\begin{array}{c}
x^0_t=\begin{pmatrix}
\half X_t     & 0 \\
Y_t + i Z_t   & -\half X_t
\end{pmatrix} \in \sufrak_2^* \otimes L^{\infty -}(\Omega)\\
\Lambda_t^0 = \sqrt{ X_t^2 + Y_t^2 + Z_t^2 } \\
 \text{\small Flat Brownian Motion on $\sufrak_2^* \approx \R^3$ } \\
 \text{\small and its radial part}
\end{array}
\end{tikzcd}
$$
Moreover, on both quantum and semi-classical pictures, i.e. for $\hbar>0$ and $\hbar=0$, the dynamic of $\Lambda^{r, \hbar}$ does not depend on $r$.
\end{thm}

\subsection{Structure of the paper}
Section \ref{section:commutativeDiagram} is aimed at giving a precise meaning to the formal commutative diagram \ref{fig:commutativeDiag}. This is given as Theorem \ref{thm:commutativeDiagram4Spaces} which is the geometric shadow of the commutative diagram in the Main Theorem. In essence, that theorem is implicit in Drinfeld's foundational ICM talk \cite{D88}. Nevertheless, the value of this section resides in giving a precise statement calibrated for this paper and in interpreting its ingredients. To that endeavor, we start by detailing classical coproducts and Poisson brackets with the goal of showing that the presentation \eqref{def:heteroQuantumGroup} can be surmised from classical objects. Then we give a more natural definition of the involution $\dagger$ for $\left( SL_2^* \right)_r$ and explain why $r>0$ is indeed a curvature parameter. Only then we will be able to formulate Theorem \ref{thm:commutativeDiagram4Spaces} as a result about deformation of algebras (via curvature $r$) and quantizations (via the Planck constant $\hbar$). For later use, we conclude the section by recording how changing the presentation of the Jimbo-Drinfeld quantum group, from \eqref{eq:Uq_pres} to \eqref{def:heteroQuantumGroup}, rescales the representations.

In section \ref{section:statics}, we prove in Theorem \ref{thm:staticLimits} that quantum observables in large representations become classical observables. This is the implementation of Kirillov's orbit method, which proves quantitatively that large representations behave like symplectic leaves (here dressing orbits), and that the tensor product of two representations behaves like convolution of two orbital measures. Furthermore, we demonstrate the appearance of the crystal tensor product rule in the $r \rightarrow \infty$ limit. We feel that this is instructive to prepare for the proof of the Main Theorem which deals with diffusive limits. In that sense, the Main Theorem is about infinitely many tensor products and infinitely many convolutions.

Finally, in section \ref{section:dynamics}, we tackle the proof of the Main Theorem \ref{thm:main}, which unifies the representation theoretic construction of Biane and the geometric construction of Bougerol-Jeulin. There, we isolate in separate subsections an independent proof of Theorem \ref{thm:BJ} and the argument that the dynamic of $\Lambda$ does not depend on $r>0$.

\section{Commutative and non-commutative geometry of Poisson-Lie groups}
\label{section:commutativeDiagram}

\subsection{Coproducts and Poisson brackets}
Now that we know that the Poisson-Lie group $\left( SL_2^* \right)_r$ will play an important role for the Main Theorem \ref{thm:main}, we will begin this section by detailing two structures on it: the coproduct and the Poisson bracket. Recalling the general definition \eqref{def:coproduct} of a coproduct for the coordinate algebra of a group, we have:
\begin{lemma}
\label{lemma:coproduct_r}
The coproduct on $\C\left[ \left(SL_2^*\right)_r \right]$ is given by:
$$ \left\{
   \begin{array}{ccc}
   \Delta_r\left( e^{\half r H} \right) & = & e^{\half r H} \otimes e^{\half r H} \ ,\\
   \Delta_r\left( F \right) & = & F \otimes e^{ \half r H} + e^{-\half r H} \otimes F \ , \\
   \Delta_r\left( E \right) & = & E \otimes e^{ \half r H} + e^{-\half r H} \otimes E \ . 
   \end{array}
   \right.
$$
\end{lemma}
\begin{proof}
Writing two elements $g_i \in \left( SL_2^* \right)_r, \ i = 1,2$ as in \eqref{def:functionsHEF}, we obtain six complex numbers $H_i = H(g_i)$, $E_i = E(g_i)$ and $F_i = F(g_i)$ for $i=1,2$. By computing the matrix product, we obtain that $g_1 g_2 = (b, b^+)$ with:
\begin{align*}
 b = & \begin{pmatrix}
   e^{\half r (H_1 + H_2)} & 0 \\
   2r \left( F_1 e^{\half r H_2} + e^{-\half r H_1} F_2 \right) & e^{-\half r (H_1 + H_2)}
   \end{pmatrix} \ ,
   \\
 b^+ = &
   \begin{pmatrix}
   e^{-\half r (H_1 + H_2)} & 2r \left( E_1 e^{\half r H_2} + e^{-\half r H_1} E_2 \right) \\
   0                        & e^{\half r (H_1 + H_2)}    
   \end{pmatrix}
 \ .
\end{align*}
Then using the definition of $\Delta_r$ in \eqref{def:coproduct}, we extract from the above expressions exactly the announced expressions for the coproduct.
\end{proof}

\begin{rmk}[The coproduct of $\Uc_q^\hbar(\slfrak_2)$ is not quantum]
\label{rmk:DeltaNotQuantum}
By comparing to Eq. \eqref{def:Delta_r}, the reader will recognize the coproduct $\Delta_r$ of the quantum groups $\Uc_q^\hbar(\slfrak_2)$ upon setting $K=e^{- r H}$. Already we see there is nothing quantum about that!
\end{rmk}

Now, by using the same trick as in the definition of the KKS structure, we identify $X_\gfrak \in \slfrak_2 = \sufrak_2 \otimes \C$ with a complex linear forms $f_X$ on $\sufrak_2^* \approx (SU_2^*)_r$. It is an easy computation to check the following.

\begin{lemma}
\label{lemma:PoissonBracket-r}
There exists a legitimate Poisson bracket on $\C\left[ (SU_2^*)_r \right]$ defined by:
\begin{align*}
\left\{ H, E \right\}_r & := 2 E \ ,  \\
 \left\{ H, F \right\}_r & := -2 F \ ,  \\
 \left\{ E, F \right\}_r & := \frac{e^{r H}-e^{-r H}}{2r} \ . 
\end{align*}
\end{lemma}
\begin{proof}
A Poisson bracket is determined by its value on linear forms because of the Leibniz rule. Thus we only have to check the Jacobi identity:
$$ 0 = \{A,\{B,C\}_r\}_r
     + \{B,\{C,A\}_r\}_r
     + \{C,\{A,B\}_r\}_r
$$
for all elements $A, B, C$ that are linear forms. The Jacobi identity can be checked on a basis. Also, because $\{\cdot, \cdot\}_r$ is anti-symmetric, there is no need to check all the possibilities, only $\left( A, B, C \right) = \left( H, E, F \right)$ i.e.
\begin{align*}
0 & \overset{?}{=} \left\{H, \frac{e^{rH}-e^{-rH}}{2r} \right\}_r
     + \{E, 2 F\}_r
     + \{F, 2 E \}_r\\
  & = \left\{ H, \frac{e^{r H}-e^{-rH}}{2r} \right\}_r \ .
\end{align*}
This last term is indeed zero, as the Poisson bracket is a derivation in each variable.
\end{proof}

Notice that the space $\sufrak_2^*$ has no reason of having a Lie bracket, hence the assignment of a trivial Lie bracket on $\sufrak_2^*$. It is natural to interpret a trivial Lie bracket as zero curvature, since that for Lie groups equipped with an invariant metric, the curvature tensor is basically equivalent to the bracket. We will explicit this idea later in this section. Adding to that the KKS Poisson bracket $\{ \cdot, \cdot \}_{0}$, which is nothing but the Lie bracket of $\sufrak_2$, one says that (\cite{CP95})
$$ \left( \sufrak_2^*, 0 = [ \cdot, \cdot ]_{\sufrak_2^*}, \{ \cdot, \cdot \}_{0} \right) $$
is a Lie bialgebra. 

Here, we have described a curved version of this statement, i.e. that we have a bialgebra structure:
$$ \left( \sufrak_2^*, r [ \cdot, \cdot ], \{ \cdot, \cdot \}_{r} \right) $$
where $r$ plays the role of curvature, dilating the Lie bracket on $\slfrak_2^*$, and $\{ \cdot, \cdot \}_{r}$ has to be a compatible Poisson bracket.

%
%
%

\subsection{The real form \texorpdfstring{$\left( SU_2^* \right)_r$}{} and Poisson-Hopf duality}
\label{subsec:SUdagger}

We write 
$$ \slfrak_2^* = \Span_\C\left( H^*, E^*, F^* \right)$$
by making the choice of basis:
$$ H^* =  \left( \begin{pmatrix} -\frac18 & 0 \\  0 & \frac18 \end{pmatrix} , \begin{pmatrix} \frac18 & 0 \\  0 & -\frac18 \end{pmatrix} \right)  ; \
   E^* =  \left( \begin{pmatrix} 0 & 0 \\  0 &  0 \end{pmatrix} , \begin{pmatrix} 0 & 1 \\  0 &  0 \end{pmatrix} \right)  ; \
   F^* =  \left( \begin{pmatrix} 0 & 0 \\  1 &  0 \end{pmatrix} , \begin{pmatrix} 0 & 0 \\  0 &  0 \end{pmatrix} \right)  .
$$
This choice of basis yields the same presentation of $\slfrak_2^*$ as in \cite{K97}.

On $\slfrak_2$, we write $\dagger$ for the (standard) conjugate transpose. It is the involutive anti-morphism which determines the compact form $\sufrak_2 \subset \slfrak_2$ at the level of Lie algebras, and the compact form $SU_2 \subset SL_2(\C)$ at the level of Lie groups:
$$ SU_2 := \left\{ x \in SL_2(\C) \ | \ x x^\dagger = \id \right\} \ .$$
Because of the standard duality between $\slfrak_2$ and $\slfrak_2^*$ (see \cite{CP95} and \cite{K97}), so that:
$$ \langle H^*, H \rangle = \langle E^*, E \rangle = \langle F^*, F \rangle = 1 $$
while all other duality brackets vanish, we can transport $\dagger$ to the Lie algebra $\slfrak_2^* = \Span_\C\left( H^*, E^*, F^* \right)$. This is done via the following relation which is compatible with Hopf algebra duality (see \cite[p.117]{CP95}):
$$ \forall f \in \slfrak_2^*, \ 
   \forall X \in \slfrak_2, \ 
   \langle f^\dagger, X \rangle := \overline{ \langle f, S(X)^\dagger \rangle } \ .$$
Here $S: \slfrak_2 \rightarrow \slfrak_2$ is the antipode map obtained by degenerating the definition \eqref{eq:Uq_S} with $\hbar \rightarrow 0$. All in all, we obtain an anti-linear involution on $\slfrak_2^*$:
$$ 
   \left( H^* \right)^\dagger = - H^*, \ 
   \left( E^* \right)^\dagger = - F^*, \ 
   \left( F^* \right)^\dagger = - E^*.
$$

Upon exponentiating, we have an antimorphism $\dagger$ acting on points of $\left(SL_2^*\right)_r$ as follows. If $H$, $F$, $E$ are complex scalars such that:
$$ x = \left(
   \begin{pmatrix}
   e^{\half r H} & 0 \\
   2r F          & e^{-\half r H}
   \end{pmatrix},
   \begin{pmatrix}
   e^{-\half r H} & 2r E \\
   0              & e^{\half r H}    
   \end{pmatrix}
   \right) \ ,$$
exactly as in Eq. \eqref{def:functionsHEF}, then we have:
$$ x^\dagger = \left(
   \begin{pmatrix}
   e^{-\half r \overline{H} }   & 0 \\
   -2r \overline{E}             & e^{\half r \overline{H}}
   \end{pmatrix},
   \begin{pmatrix}
   e^{\half r \overline{H} }  & -2r \overline{F} \\
   0                          & e^{-\half r \overline{H} }    
   \end{pmatrix}
   \right) \ .$$
As such, we obtain a presentation of the subgroup $\left( SU_2^* \right)_r$ which is more natural than \eqref{def:SU2star} via:
$$
   \left( SU_2^* \right)_r = \left\{ x \in \left( SL_2^* \right)_r \ | \ x x^\dagger = \id \right\} \approx NA \ .
$$
Indeed, $x x^\dagger = \id$ yields that $H-\overline{H}=0$ i.e. $H$ is real and $E=\overline{F}$. Therefore, all the information is in the lower Borel subgroup with positive diagonals. Going back to the Poisson-Lie dual, $\sufrak_2^*$ is thus naturally identified with lower triangular matrices and $\left( SU_2^* \right)_r \approx NA$. The subscript $r$ indicates that we renormalize the group law thanks to the parameter $r$.

\subsection{The constant \texorpdfstring{$r$}{r} is a curvature parameter}
\label{subsection:rIsCurvature} 
Following \cite[$\mathsection 3.17$]{GHL90}, if $G$ is a Lie group with invariant metric, then the $(1,3)$-curvature tensor is
$$ R(X,Y,Z) = \frac{1}{4} [X, [Y, Z]] \ ,$$
where $X$, $Y$ and $Z$ represent invariant vector fields. In our case, the parameter $r > 0$ controls the curvature of $\H^3 \approx \left(SU_2^*\right)_r$ via:
$$ R(X,Y,Z) = \frac{r^2}{4} [X, [Y, Z]_{\slfrak_2^*} ]_{\slfrak_2^*} \ .$$

The sectional curvature along any plane $P = \text{Span}_\R(X, Y) \subset \sufrak_2^* \approx \afrak \oplus \nfrak$, spanned by orthonormal vectors $X$ and $Y$, is by definition $ K(P) := \langle R(X,Y,X), Y \rangle \ ,$ where the natural scalar product is given by the Killing form. By invariance of the Killing form,
$ K(P) = -\frac{r^2}{4} \| [Y, X]_{\slfrak_2^*} \|^2$. This expression is again invariant under conjugation by $SU_2$ or more accurately via the coadjoint action of $SU_2$ on $\sufrak_2^* \approx \afrak \oplus \nfrak$. As a consequence, there is no loss of generality in assuming $Y = \frac{H}{\|H\|} = \frac{H}{\sqrt{2}} \in \afrak \oplus \nfrak$. Any orthogonal vector must be of the form $X = e^{i\theta} \frac{F}{\|F\|} = e^{i\theta} F$, which yields $\| [Y, X]_{\slfrak_2^*} \| = \| \frac{1}{\sqrt{2}} [H, F]_{\slfrak_2^*} \| = \sqrt{2}$.

In the end, the sectional curvature of $\left( SU_2^* \right)_r \approx SL_2(\C)/SU_2 = \H^3$ is indeed constant and equal to $K(P) = -\half r^2$, which goes to $-\infty$ as $r \rightarrow \infty$.

\subsection{Deformation and quantization}

Let  $\C[[z]]$ denote the ring  of formal power series  in the indeterminate $z$. 
Recall from \cite[Definition~6.1.1 and Definition~6.2.4]{CP95}:

\begin{definition}
\label{def:deformation}
A Hopf deformation over $\C[[z]]$ of a Hopf algebra $\Ac$  with product $\mu$ and coproduct $\Delta$,  is a Hopf algebra $\Ac_z$ with product $\mu_z$ and coproduct $\Delta_z$ such that
\begin{enumerate}
\item[(i) ] $\Ac_z \approx \Ac[[z]]$  as $\C[[z]]$-module,
\item[(ii)] $\mu_z \equiv \mu \ \modulo z$,  and  $ \Delta_z \equiv \Delta \ \modulo z$ .
\end{enumerate}
The notation $\Ac[[z]]$ means the formal power series in $z$ with coefficient in $\Ac$. 

Moreover, if $\Ac$ has the additional structure of a Poisson algebra, one says  that $\Ac_z$ is a quantization of $\Ac$ if $\Ac_z$ is a Hopf deformation of $\Ac$ and
$$
\frac{\mu_z(a,b)-\mu_z(b,a)}{z} \equiv \{\pi_z(a),\pi_z(b)\} \ \modulo z \ ,
$$
where $\{\cdot,\cdot\}$ is the Poisson bracket on $\Ac$ and $\pi_z: \Ac_z \rightarrow \Ac_z /z \Ac_z$ is the quotient map $\modulo z$.
\end{definition}

Notice that as a consequence of $(i)$, we have a linear isomorphism $\Ac \approx \Ac_z /z \Ac_z$. We can now turn the formal diagram of Figure~\ref{fig:commutativeDiag} into a commutative diagram which expresses a deformation (via curvature $r$) and a quantization (via the Planck constant $\hbar$).

\begin{thm}[A commutative diagram for spaces]
\label{thm:commutativeDiagram4Spaces}
Let us denote by $\mu^\hbar_q$ and $\mu^\hbar$  the products on $\Uc^\hbar_q (\slfrak_2)$ and $\Uc^\hbar (\slfrak_2)$ respectively. The following diagram between  Hopf algebras commutes:
$$
\begin{tikzcd}
\left( \,  \left( \Uc_{q}^\hbar(\slfrak_2), \dagger \right), \ \mu^\hbar_q, \ \Delta_r \, \right)
                         \arrow{d}{\modulo r}
                         \arrow{r}{\modulo \hbar} &
\left( \C\left[ (SU_2^*)_r \right] ,  \ \{ \cdot, \cdot \}_r, \ \Delta_r \right)
    \arrow{d}{\modulo r}\\
\left( \,  \left( \Uc^\hbar(\slfrak_2), \dagger \right), \ \mu^\hbar, \ \Delta_0 \, \right)
    \arrow{r}{\modulo \hbar} & 
\left( \C\left[ \sufrak_2^* \right], \ \{ \cdot, \cdot \}_0, \ \Delta_0 \right) 
\end{tikzcd}
$$
Vertical arrows are curvature deformations, while horizontal arrows are quantizations of Poisson algebra.
\end{thm}
As mentioned before, this theorem can historically be attributed to Drinfeld since it is implicitly present in his ICM foundational paper \cite{D88}. It appears more explicitly and in a very general form in the paper of Kassel-Turaev \cite{KT00}. A more detailed discussion is given in the next subsection. For the sake of completeness, let us give a condensed proof while interpreting its ingredients.

\begin{proof}
For the purposes of this proof, all algebras are completed with respect to the $\hbar$-adic and $r$-adic topology. If an algebra $\Ac_z$ uses the parameter $z \in \{r, \hbar\}$, then $\Ac$ is seen  as a $\C[[z]]$-module. For example, $\Uc_q^\hbar\left( \slfrak_2 \right)$ uses both parameters ($q = e^{-r}$).

We start by checking (i) in Definition \ref{def:deformation}. The argument boils down to Poincar\'e-Birkhoff-Witt theorems (PBW) \cite{SW15} thanks to which all four algebras are free modules with essentially the same basis. Indeed, by the classical PBW theorem, the ordered monomials
$$
H^aE^bF^c, \quad \text{ for } a,b,c \geq 0,
$$
form a (topological) basis of $\Uc^\hbar (\slfrak_2)$. As such, any element of $\Uc^\hbar (\slfrak_2)$ can be written as a sum of monomials $H^aE^bF^c$ with coefficient in $\C[[\hbar]]$, giving the isomorphism between $\C[[\hbar]]$-modules:
$$
\Uc^\hbar (\slfrak_2) \approx \C [H,E,F] [[\hbar]] = \C[ \sufrak_2^* ][[\hbar]] \ . 
$$
Note that the discrepancy between $\slfrak_2$ and $\sufrak_2$ is due to the fact that we consider the above linear isomorphism as a linear isomorphism between involutive algebras: it is the compact form $\left( \Uc^\hbar (\slfrak_2) , \dagger \right)$ which deforms the Poisson algebra $\C[\sufrak_2^*]$, viewed as the coordinate  algebra $\C[\slfrak_2^*]\approx \C[H,E,F]$ equipped with the $\dagger$-involution.
In the case of $\Uc_q^\hbar (\slfrak_2)$, one invokes the quantum PBW theorem, which states that the ordered monomials
$$
K^a E^b F^c, \text{ for $a\in\Z$, $b,c\in \N$}
$$
form a (topological) basis of $\Uc^\hbar_q (\slfrak_2)$, giving the isomorphism between $\C[[\hbar]]$-modules:
$$
\Uc^\hbar_q (\slfrak_2) \approx \C[K,E,F] [[\hbar]] = \C[ \left(SU_2^*\right)_r ][[\hbar]] \ .
$$
As for the flat case $r=0$, the above equation shows a discrepancy between $\slfrak_2$ and $\left(SU_2^*\right)_r$. The compact form $\left( \Uc^\hbar_q (\slfrak_2), \dagger \right)$ will be a Hopf deformation of the Poisson algebra $\C[(SU_2^*)_r]$, viewed as the real form of the Poisson algebra $\C[(SL_2^*)_r]$ equipped with the $\dagger$-involution. Recall that subsection~\ref{subsec:SUdagger} explains how $\dagger$ is transported from the algebra of functions to points and thus defines $(SU_2^*)_r \subset (SL_2^*)_r$.

Now notice that, since $K^{\pm 1} = e^{\pm rH}$, one realizes that by expanding the exponential into formal power series, the two previous bases are related by a triangular linear transformation with coefficients in $\C[r]$. As such:
\begin{align*}
   \Uc_q^\hbar\left( \slfrak_2 \right) &
   \approx
   \C[H, E, F][[r]][[\hbar]]
   \approx
   \Uc^\hbar\left( \slfrak_2 \right)[[r]]  \\ 
\intertext{and}   
   \C[ \left(SU_2^*\right)_r ] & 
   \approx
   \C[H, E, F][[r]]
   \approx
   \C[\sufrak_2^*][[r]] \ .
\end{align*}
This proves $(i)$ for all four arrows.

\medskip

In order to prove $(ii)$, coproducts are swiftly dealt with as follows. Remark \ref{rmk:DeltaNotQuantum} reminds us that coproducts are unchanged with the deformation in $\hbar$. For the deformation in $r$, $\Delta_r \equiv  \Delta_0 \ \modulo r$ is seen from the expression of the coproduct at the level of generators in Lemma \ref{lemma:coproduct_r}: $H$ remains primitive for all $r > 0$ and $e^{\half r H} \equiv 1 \ \modulo r$.

For products, we need to analyze the relations in the algebra $\Uc_q^\hbar\left( \slfrak_2 \right)$ given after Eq. \eqref{def:heteroQuantumGroup}. Again, viewing the element $K^{\pm 1}=e^{\pm rH}$ as the formal power series $\sum_{k\geq0} \frac{(\pm 1)^k r^k}{k!}H^k$, the Hopf algebra $\Uc^\hbar_q (\slfrak_2)$ can be defined as the algebra generated by the elements $H,E,F$ and relations
\begin{align}
\label{eq:Uqh_relations2}
&
[H,E]=2\hbar E, \quad [H,F]=-2\hbar F, \quad EF-FE=\hbar \frac{e^{rH}-e^{-rH}}{2r} \ .
\end{align}
Indeed, the first two relations implies easily the two relations $K^\half E K^{-\half} = e^{-r\hbar}E$  and $K^\half F K^{-\half} = e^{r\hbar} F$ of (\ref{eq:Uqh_relations}). Since the third relation yields:
$$ EF-FE=\hbar H \ \modulo r \ ,$$
we recover the relations of $\Uc^\hbar\left( \slfrak_2 \right)$ modulo $r$. The congruence $\mu^\hbar_q \equiv \mu^\hbar \ \modulo r$ follows. The commutative product of functions is obviously the same in $\C[\sufrak_2^*]$ and $\C[(SU_2^*)_r]$. For products modulo $\hbar$, Eq.~\eqref{eq:Uqh_relations2} yields that $\mu_q^\hbar$ (resp. $\mu^\hbar$) is congruent $\modulo \hbar$ to a commutative product.

We have thus proved (ii) for all four arrows, which are therefore deformations in the sense of Definition \ref{def:deformation}.

\medskip

Finally, let us show that the two horizontal arrows are also quantizations, by proving that commutators $\modulo \hbar^2$ encode Poisson brackets. It is well known that the Hopf algebra  $\Uc^\hbar (\slfrak_2)$ is a quantization of the Poisson algebra $\C[\sufrak_2^*]$ (see for instance the more general example \cite[Example~6.2.5]{CP95}). We only prove that  $\Uc^\hbar_q (\slfrak_2)$ is a quantization of $\C[(SU_2^*)_r]$ with an argument that is uniform in $r > 0$. 

Using the quotient map $\pi = \modulo \hbar$ and the Poisson bracket $\{\cdot,\cdot\}_r$ on $\C[(SU_2^*)_r]$, it remains to prove that, for any $a,b\in \Uc^\hbar_q (\slfrak_2) $,
$$
\frac{\mu^\hbar_q(a,b)-\mu^\hbar_q(b,a)}{\hbar} =  \{\pi(a),\pi(b)\}_r \ \modulo \hbar \ .
$$
Using commutators on the left-hand side, and because of the Leibniz rule for the Poisson bracket on the right-hand side, it suffices to prove the above on the generators of $\Uc^\hbar_q (\slfrak_2) $ only. The claim then follows from putting together Eq.~\eqref{eq:Uqh_relations2} and Lemma \ref{lemma:PoissonBracket-r}. Notice that the same argument carries upon taking $r \rightarrow 0$.
\end{proof}

\subsection{Further comments on two-parameter deformations}

Since $\slfrak_2$ is (semi-)simple, Drinfeld's rigidity theorem \cite[Theorem 6.1.8]{CP95} states that the deformation theory of $\Uc\left( \slfrak_2 \right)$ is trivial, in the sense that every deformation of $\Uc(\slfrak_2)$ using $\hbar$ is isomorphic to $\Uc(\slfrak_2)[[\hbar]]$ as an algebra. One could take a pessimistic stance on that theorem and decide that there is no point in trying to play with algebras such as the one-parameter deformation $\Uc_q(\slfrak_2)$. Nevertheless no explicit equivalence is known and the success of quantum groups has proven this intuition wrong. 

Even going beyond one-parameter deformations, there are quite a few papers which develop two-parameter deformations. We are only aware of the papers \cite{KT00} and \cite{BCO09}, which develop exactly the same two-parameter deformation but with a different point of view.

\subsubsection{Kassel and Turaev's biquantization}

In \cite{KT00}, Kassel and Turaev present a commutative diagram that is intriguingly similar to the formal diagram in Figure \ref{fig:commutativeDiag}, as it uses two parameters denoted $u$ and $v$. Furthermore, their result, termed as biquantization of Lie bialgebras, holds for all Lie algebras. 

The deformed law \eqref{eq:groupLaw} is considered in \cite[Eq. (2.6)]{KT00}, and defined formally from the Baker-Campbell-Hausdorff formula. The Baker-Campbell-Hausdorff formula converges in our case, since we are dealing with a solvable group. Dually, one obtains a coproduct $\Delta_r$ on the coordinate algebra of $\left( SL_2^* \right)_r$. Also, it is made into a Poisson algebra thanks to a Poisson bivector, which comes from the bracket of $\slfrak_2$ (see \cite[Eq. (2.9)]{KT00}).

Their deformations in the parameter $u$ are indeed quantization of Poisson bialgebras while the deformations in $v$ are co-quantizations of co-Poisson algebras. Their result also holds dually, replacing $\mathfrak g$ by its dual $\mathfrak g^*$, giving a  symmetric roles to the parameters $u$ and $v$, hence the term of biquantization. In the context of Question \ref{question:main}, we never invoke the co-quantization using the parameter $r > 0$. In our case the parameter $\hbar \geq 0$ controls the non-commutativity in the algebra of observables $\Ac^{r, \hbar}$, in accordance with the principles of quantum mechanics, while $r > 0$ controls the non-commutativity of invariant vector fields, in the underlying space $\left( SU_2^* \right)_r$ and thus controls the curvature.

\subsubsection{Ballesteros, Celeghini and Del Olmo's point of view} In a very insightful comment, the authors of \cite{BCO09} argue that quantum groups and quantum algebras need to be viewed as abstract Hopf algebras since that, in many physical cases, the parameter $q$ is completely independent from the truly quantum $\hbar$ constant. An example they give, among others, is that $q$ clearly plays the role of anisotropy parameter in the context of the Heisenberg XXZ spin chain. Commutative diagrams similar to Fig.~\ref{fig:commutativeDiag} do appear in their paper.

\medskip

Let us conclude the section by arguing that the common name of QUE algebras which stands for ``Quantum Universal Enveloping algebras" is misleading. It is perhaps more appropriate to refer to the Drinfeld-Jimbo quantum group either as a ``quantized Poisson-Lie function algebra", which is Drinfeld's original point of view, or as a ``curved universal enveloping algebra".

\subsection{Representations}
\label{subsection:representations}
Since Theorem \ref{thm:commutativeDiagram4Spaces} deals with another presentation of the Drinfeld-Jimbo quantum group, the representation theory is essentially unchanged. Of course, the constants $\hbar$ and $r$ appear at various places and the goal of this subsection is to record these unessential changes.

Recall from \cite[Page 61]{KS12} that $\Uc_q(\slfrak_2)$, with the usual presentation (\ref{eq:Uq_pres}), that for every $\lambda \in \N$, there exists a representation $V^q(\lambda)$ of dimension $\lambda+1$ and with basis labelled 
$$ \left( e_k \ ; \ k=-\lambda, -\lambda+2, \dots, \lambda-2, \lambda \right) \ .$$
The action on generators is given in terms of $q$-integers $[n]_q := \frac{q^{n}-q^{-n}}{q-q^{-1}}$ by:
\begin{align*}
 \left\{\begin{array}{ccc}
 K \cdot e_k = & q^k e_k \ , & \\
 E \cdot e_k = & \sqrt{\left[\half\lambda-\half k \right]_q
                       \left[\half \lambda+\half k+1 \right]_q
                      } e_{k+2}
             = & \frac{\sqrt{
                      q^{\lambda+1}+q^{-(\lambda+1)}
                      -q^{k+1}-q^{-(k+1)}
                      }}
                      {\left| q-q^{-1} \right|} e_{k+2} \ ,\\
 F \cdot e_k = & \sqrt{\left[\half \lambda+ \half k \right]_q
                       \left[\half \lambda-\half k+1 \right]_q
                      } e_{k-2} 
             = & \frac{\sqrt{
                      q^{\lambda+1}+q^{-(\lambda+1)}
                      -q^{k-1}-q^{-(k-1)}
                      }}
                      {\left| q-q^{-1} \right|} e_{k-2} \ .
 \end{array}
 \right.
\end{align*}

By definition, the weight lattice is the union of possible spectra for the operator $H$, in all possible representations. Upon rescaling thanks to the isomorphism \eqref{def:isom}, the weight lattice is rescaled from $\Z$ to $\hbar \Z$. Highest weights are therefore of the form $\Lambda^\hbar = \hbar \lfloor \Lambda / \hbar \rfloor \in \hbar \Z$ for $\Lambda \geq 0$, and highest weight representations are indexed by such elements. As such, denoting by $V^q(\Lambda^\hbar)$ the irreducible representation  of dimension $\lfloor \Lambda/\hbar \rfloor + 1$, it has a linear basis:
\begin{equation}
\left( e_{\hbar k } \ ; \ k=-\lfloor \Lambda/\hbar\rfloor, \lfloor \Lambda/\hbar\rfloor+2,\ldots, \lfloor \Lambda/\hbar\rfloor-2, \lfloor \Lambda/\hbar\rfloor \right)
\subset
V^q(\Lambda^\hbar)
\ . 
\end{equation}
 The action on generators of $\Uc_q^\hbar(\slfrak_2)$ becomes:
\begin{align}
\label{eq:repAction}
  \left\{
  \begin{array}{cc}
  K \cdot e_{\hbar k} = & e^{-r \hbar k} e_{\hbar k} \ , \\
  E \cdot e_{\hbar k} = & \alpha_r^\hbar \sqrt{e^{r(\Lambda^\hbar+\hbar)}+e^{-r(\Lambda^\hbar+\hbar)}
                       -e^{r\hbar(k+1)}-e^{-r\hbar(k+1)}
                       } e_{\hbar(k+2)} \ ,\\
  F \cdot e_{\hbar k} = & \alpha_r^\hbar \sqrt{e^{r(\Lambda^\hbar+\hbar)}+e^{-r(\Lambda^\hbar+\hbar)}
                       -e^{r\hbar(k-1)}-e^{-r\hbar(k-1)}
                       } e_{\hbar(k-2)} \ ,
  \end{array}
  \right.
\end{align}
where $\alpha_r^\hbar = \frac{1}{2r} \sqrt{ \frac{2r \hbar}{e^{r \hbar}-e^{-r \hbar}} }$. The vector $e_{\hbar k}$ is therefore the vector of weight $\hbar k$. In particular, $e_{\Lambda^\hbar}$ is the highest weight vector with weight $\Lambda^\hbar$. 

Likewise, instead of the usual Casimir \eqref{def:usualCasimir} of $\Uc_q(\slfrak_2)$, which acts on $V^q(\lambda)$ as the constant:
\begin{align}
\label{eq:usualCasimirConstant}
       \left( C^q \right)_{| V^q(\lambda) }
 & :=  \frac{q^{ \lambda+1} + q^{-( \lambda+1)}}
           {\left( q - q^{-1} \right)^2}\ ,
\end{align}
we have to consider our Casimir \eqref{def:Casimir} which acts on $V^q(\Lambda^\hbar)$ like the constant:
\begin{align}
\label{eq:casimirConstant}
        C^{r, \hbar} (\Lambda^\hbar)
 & := \half \frac{2r\hbar}{ e^{r\hbar} - e^{-r\hbar} }
 \left(
             e^{  r (\Lambda^\hbar + \hbar)}
            +e^{- r (\Lambda^\hbar + \hbar)}
            \right)  \ .
\end{align}
This is equivalent to saying that the operator $\Lambda^{r, \hbar}$ defined in \eqref{def:elementLambda} acts on $V^q({\Lambda^\hbar})$ as the constant $\Lambda^\hbar$. Notice that this constant is independent of the curvature $r>0$.

\section{Static semi-classical limits and crystallization}
\label{section:statics}

The main theorem of this section deals with an effective implementation of the orbit method i.e. quantum observables converging to classical observables, while tracking the effect of the curvature parameter $r > 0$. This allows to show the appearance of crystals in the infinite curvature limit, both on the quantum and semi-classical side. 

\medskip

\paragraph{\bf Dressing orbits:} In order to give an effective implementation of the orbit method, we first require the definition of the dressing action of $SU_2$ on $\left( SU_2^* \right)_r \approx NA$, which is exactly the curved version of the coadjoint action (see \cite{K97}). It is defined as:
$$
\begin{array}{ccc}
SU_2 \times NA  & \rightarrow & NA \\
(k,b)           & \mapsto     & k \cdot b \ ,
\end{array}
$$
where $k \cdot b = b'$ is the lower triangular matrix obtained via the Gram-Schmidt decomposition $ kb = b'k' \ ,$ with $k' \in SU_2$. 
We identify $\Lambda > 0$ with the diagonal matrix $\begin{pmatrix} \Lambda & 0 \\ 0 & -\Lambda \end{pmatrix}$. By definition, the dressing orbit of $SU_2$ passing through $\Lambda \in \R_+$ in $\left( SU_2^* \right)_r$ is:
\begin{align}
\Oc_r(\Lambda) := & SU_2 \cdot \Lambda
                = \left\{ k \cdot \exp\left( r \begin{pmatrix} \Lambda & 0 \\ 0 & -\Lambda \end{pmatrix} \right)
                            \in  \left( SU_2^* \right)_r \ | \ k \in SU_2
                    \right\} \ .
\end{align}
The orbit $\Oc_r(\Lambda)$ is endowed with a natural invariant measure, induced by the Haar measure on $SU_2$. A uniform element on $\Oc_r(\Lambda)$ is a random variable which is distributed according to this invariant measure.

\medskip

\paragraph{\bf Crystals:} For any fixed quantum group, Kashiwara associates to any representation $V^q$ a set $\Bc$ called the crystal basis or crystal for short. Morally, a crystal is the combinatorial object which remains upon taking the $q \rightarrow 0$ limit of the so-called Kashiwara's global basis which is equivalent to Lusztig's canonical basis. For reference, we recommend the comprehensive \cite{KC02}. A crystal is a set endowed with a graph structure so that connected components are in bijection with highest weight irreducible modules. Furthermore, Kashiwara defined a combinatorial tensor product rule for two crystals $\Bc$ and $\Bc'$, by giving a crystal graph structure to:
$$ \Bc \otimes \Bc' := \left\{ b_1 \otimes b_2 \ | \ (b_1 , b_2) \in \Bc \times \Bc' \right\} \ .$$
As a set, it is nothing but $\Bc \times \Bc'$. Nevertheless the notation $\Bc \otimes \Bc'$ hints to the fact that it has a richer combinatorial structure thanks to the crystal graph.

In the case of $\Uc_q^\hbar(\slfrak_2)$, let us illustrate the previous paragraph, while changing the weight lattice from $\Z$ to $\hbar \Z$ to accommodate our setting. We write that any crystal $\Bc$ is a disjoint union of highest weight crystals $\Bc(\Lambda)$:
$$
   \Bc = \sqcup_{\Lambda \in \Cc} \Bc(\Lambda) \ ,
$$
where $\Cc$ is the multi-set of connected components - that is to say that repetition of $\Lambda \in \hbar \N$ is allowed. Any highest weight crystal $\Bc(\Lambda)$ for a highest weight $\Lambda\in \hbar \N$ is naturally identified with the weight basis:
\begin{align}
\label{eq:BcIdentification}
\Bc(\Lambda) \approx & \ \left\{ -\Lambda, -\Lambda + 2\hbar, \dots, \Lambda - 4\hbar, \Lambda - 2\hbar, \Lambda \right\} \ ,
\end{align}
and the crystal graph structure is depicted in Figure \ref{fig:crystalExample}. It is obtained by connecting integer multiples of $\hbar$ which are neighbors on the real line.

\begin{center}
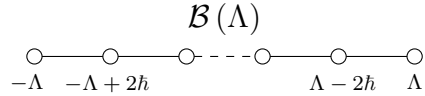
\begin{figure}[ht!]
\begin{tikzpicture}[scale=1]
\draw (4.5,4.5) node {$\Bc\left( \Lambda  \right)$} ; 
\foreach \x in {2,...,7} { \draw (\x,4) circle (0.1) ; }
\foreach \x in {2,3,5,6} { \draw (\x.1,4) -- (\x.9,4)   ; } \draw[dashed] (4.1,4) -- (4.9,4) ;
\draw (1.9,3.9) node[below] {\tiny $-\Lambda$} ;
\draw (2.95,3.9) node[below] {\tiny $-\Lambda+2\hbar$} ;
\draw (6.05,3.9) node[below] {\tiny $\Lambda-2\hbar$} ;
\draw (7,3.9) node[below] {\tiny $\Lambda$} ;

\end{tikzpicture}
\caption{The crystal graph for $\Bc(\Lambda)$ for $\Lambda \in \hbar \N$.}
\label{fig:crystalExample}
\end{figure}
\end{center}

For any crystal $\Bc$, there are two well-defined maps 
$$ \wt: \Bc \rightarrow \hbar \Z \ ; \quad \hw: \Bc \rightarrow \hbar \N \ ,$$
which respectively give the weight of a crystal basis element and the highest weight corresponding to its connected component. Notice that for the very special case of $\slfrak_2$, the identification \eqref{eq:BcIdentification} corresponds to fixing $\Lambda \in \hbar \N$ and identifying $b \in \Bc(\Lambda)$ with $\wt(b) \in \hbar \Z$. The Kashiwara tensor product rule for $b_1 \in \Bc(\Lambda_1)$ and $b_2 \in \Bc(\Lambda_2)$ states that (see \cite[D\'ef.~2.3.2]{KC02}):
$$
   \wt( b_1 \otimes b_2 ) = \wt(b_1) + \wt(b_2) \ ; \quad
   \hw( b_1 \otimes b_2 ) = \max\left( \Lambda_1 + \wt(b_2),  \ -\wt(b_1) + \Lambda_2 \right) \ .
$$
In our opinion, this rule and its pictorial description (Fig. \ref{fig:crystalTensorProduct}) are landmarks of crystal combinatorics. The left-hand side in Figure \ref{fig:crystalTensorProduct} is classical (see e.g. \cite[Chapter 2, Fig. 2.1]{KC02}) and is reproduced for the reader's convenience.

\begin{center}
\begin{figure}[ht!]
\begin{tikzpicture}[scale=0.8]
\draw (0.0,1.5) node {\small $\Bc\left( \Lambda_2^\hbar \right)$} ; 
\draw (4.5,4.5) node {\small $\Bc\left( \Lambda_1^\hbar \right)$} ; 

\foreach \x in {0,...,3} { \draw (1,\x) circle (0.1) ; }
\draw (1,0.1) -- (1,0.9) ; \draw[dashed] (1,1.16) -- (1,1.93) ;  \draw (1,2.1) -- (1,2.9) ; 

\foreach \x in {2,...,7} { \draw (\x,4) circle (0.1) ; }
\foreach \x in {2,3,5,6} { \draw (\x.1,4) -- (\x.9,4)   ; } \draw[dashed] (4.16,4) -- (4.93,4) ;

\foreach \x in {0,...,3} { \draw (2,\x) circle (0.1) ; }
\foreach \x in {3,...,7} { \draw (\x,3) circle (0.1) ; }
\foreach \x in {0,2} {\draw (2,\x.1) -- (2,\x.9) ; }
\draw[dashed] (2,1.16) -- (2,1.93) ;
\foreach \x in {2,3,5,6} { \draw (\x.1,3) -- (\x.9,3)  ; }
\draw[dashed] (4.16,3) -- (4.93,3) ;

\draw (3,0) circle (0.1) ;  \draw (3,1) circle (0.1) ; \draw (3,2) circle (0.1) ; \draw (4,2) circle (0.1) ; \draw (5,2) circle (0.1) ; \draw (6,2) circle (0.1) ; \draw (7,2) circle (0.1) ;
\draw (3,0.1) -- (3,0.9) ; \draw (3,1.1) -- (3,1.9) ;  \draw (3.1,2) -- (3.9,2) ;   \draw (4.1,2) -- (4.9,2) ;  \draw (5.1,2) -- (5.9,2) ;  \draw (6.1,2) -- (6.9,2) ;
\draw (4,0) circle (0.1) ; \draw (4,1) circle (0.1) ; \draw (5,1) circle (0.1) ; \draw (6,1) circle (0.1) ;  \draw (7,1) circle (0.1) ; 
\draw (4,0.1) -- (4,0.9) ;  \draw (4.1,1) -- (4.9,1) ;  \draw (5.1,1) -- (5.9,1) ;  \draw (6.1,1) -- (6.9,1) ;
\draw (5,0) circle (0.1) ; \draw (6,0) circle (0.1) ;  \draw (7,0) circle (0.1) ;
\draw (5.1,0) -- (5.9,0) ; \draw (6.1,0) -- (6.9,0)  ;


\draw (8.5,2) node {\large$\overset{\hbar\to0}{\longrightarrow}$ } ;

\draw (13,4.5) node {\small $\Lambda_1>0$} ; 
\draw (17,1.5) node {\small $\Lambda_2>0$} ; 

\draw (10,0) -- (16,0) ; \draw (10,0) -- (10,4) ; \draw (16,0) -- (16,4) ; \draw (10,4) -- (16,4) ;
\draw (11,0) -- (11,3) ; \draw (11,3) -- (16,3) ; 
\draw (12,0) -- (12,2) -- (16,2) ;
\draw (13,0) -- (13,1) -- (16,1) ;
\end{tikzpicture}
\caption{Tensor product of crystals and scaling limit as $\hbar \rightarrow 0$. Connected components correspond to level sets of the map $\hw$.}
\label{fig:crystalTensorProduct}
\end{figure}
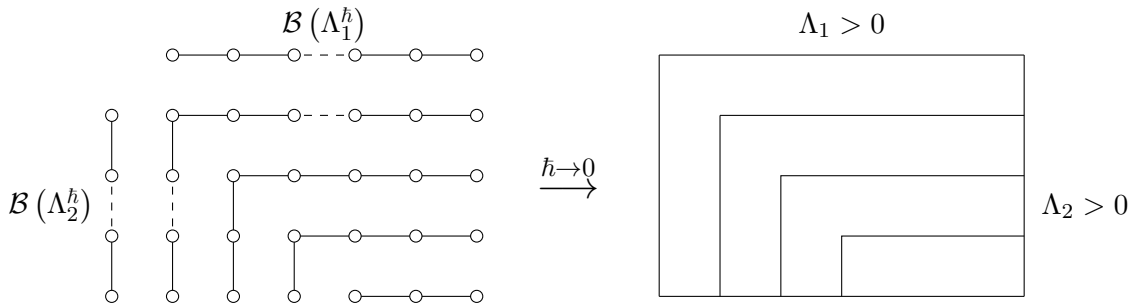
\end{center}

The appearance of rational expressions in the max-plus algebra is intimately related to the following tropicalization trick:
\begin{align}
\label{eq:tropTrick}
\forall (a,b) \in \R \times \R, \ 
   \lim_{r \rightarrow \infty}   
   \frac{1}{r}
   \log \left( e^{r a} + e^{r b} \right)
   = 
   \max(a,b) \ ,
\end{align}
which is simple yet ubiquitous in crystal combinatorics (see \cite[Chapter 7]{C13}). As we shall see, it will play an important role in this section and in the next.

We now state the main theorem of this section:
\begin{thm}
\label{thm:staticLimits}
Let $\Lambda_1,\Lambda_2>0$ and let $g_1$ and $g_2$ be random variables uniformly distributed on $\Oc_r(\Lambda_1)$ and $\Oc_r(\Lambda_2)$ respectively. Also for all observables $\Fc \in \Uc^\hbar_q\left( \slfrak_2 \right)$, write 
$$ f := \Fc \modulo \hbar \in \C\left[ (SU_2^*)_r \right] \ , $$
where $\modulo \hbar$ is the quotient map appearing in Theorem \ref{thm:commutativeDiagram4Spaces}.
 
Then the following semi-classical limits hold:
\begin{align}
   \label{eq:oneOrbit}
   \frac{ \Tr_{V^q(\Lambda_1^\hbar)}( q^{-H} \Fc) }
        { \Tr_{V^q(\Lambda_1^\hbar)}( q^{-H} ) }
    & \ \overset{\hbar \rightarrow 0}{\longrightarrow} \ 
      \E\left(   f(g_1)\right), \\
   \label{eq:twoOrbits}
   \frac{ \Tr_{V^q(\Lambda_1^\hbar) \otimes V^q(\Lambda_2^\hbar)}( q^{-H} \Fc) }
        { \Tr_{V^q(\Lambda_1^\hbar) \otimes V^q(\Lambda_2^\hbar)}( q^{-H} ) }
  & \ \overset{\hbar \rightarrow 0}{\longrightarrow} \ 
    \E \left(  f(g_1g_2) \right) .
\end{align}

In particular, the crystal regime appears by letting $r \rightarrow \infty$ and by taking an observable depending only on $H$ and on the operator $\Lambda^{r,\hbar}$ defined in (\ref{def:elementLambda}). More specifically, upon writing $\Fc = \varphi(\Lambda^{r,\hbar}) \psi(H)$ for two polynomials $\varphi$ and $\psi$, the following commutative diagram holds:
\begin{center}
\begin{equation} \label{eq:CDcrystals}
\hspace*{-1cm}
\begin{tikzpicture}
\node at (0,0) (A) {
$ 
	    \tr_{V^q(\Lambda_1^\hbar) \otimes V^q(\Lambda_2^\hbar)} [ \varphi(\Lambda^{r,\hbar}) \psi(H) ] 
$
} ;
\node at (9,0) (B)  {
$ \displaystyle
\frac{\E \left( e^{-rH(g_1g_2)} \varphi \circ \Lambda(g_1g_2)\,  \psi\circ H(g_1g_2)  \right)}
{\E \left(  e^{-rH(g_1g_2)}   \right)}
$
};
\node at (-0.5,2.1) (C)  {
$
 \displaystyle \sum_{ \substack{b_1 \otimes b_2  \in \\     \Bc(\Lambda_1^\hbar ) \otimes \Bc( \Lambda_2^\hbar ) }} 
\!\!\!\!\!\!\!  \frac{\varphi  \left(  \hw( b_1 \! \otimes\! b_2 )\right) \,  \psi \left( \wt( b_1\! \otimes \! b_2 ) \right) }  { \scriptstyle {( \Lambda_1^\hbar + 1 ) ( \Lambda_2^\hbar + 1 ) } }
$
};
\node at (9.4,2.4) (D)  {
$ 
\frac{1}{4 \Lambda_1 \Lambda_2}
{ \displaystyle \int_{-\Lambda_1}^{\Lambda_1} \int_{-\Lambda_2}^{\Lambda_2} }
 \varphi \!\left( \hw_{\Lambda_1,\Lambda_2}(\mu_1,\mu_2) \right)   \psi( \mu_1 \! + \! \mu_2 )  d \mu_1 d \mu_2 
$
} ;
\draw[->] (3,0) -- (5,0) ; \draw (4,0) node[above]  {\tiny $\hbar \rightarrow 0$} ; 
\draw[->] (3.2,2.38) -- (4.5,2.38) ; \draw (3.8,2.38) node[above]  {\tiny $\hbar \rightarrow 0$} ; 
\draw[->] (0.3,0.6) -- (0.3,1.5) ; \draw (0.3,1) node[right]  {\tiny $r \to \infty$} ; 
\draw[->] (8.5,0.7) -- (8.5,1.8) ; \draw (8.5,1.25) node[right]  {\tiny $r \to \infty$} ; 
\end{tikzpicture}
\end{equation}
\end{center}
where $\hw_{\Lambda_1,\Lambda_2}(\mu_1,\mu_2)= \max (\Lambda_1+\mu_2,-\mu_1+\Lambda_2)$ and $\Lambda(g) = \Lambda \in \R$ is such that $g \in \Oc_r(\Lambda)$. In order to simplify notation, we identify in (\ref{eq:CDcrystals}) a function and its image by the quotient map  $\modulo \hbar$. 
\end{thm}

Limits (\ref{eq:oneOrbit}) and (\ref{eq:twoOrbits}) are considered folklore in the representation theory community and often not made explicit. In the literature, the authors managed to only identify \cite[Theorem 3.5]{Yi16} where Eq.~\eqref{eq:oneOrbit} is explicitly present with morally $r=1$. Sun's approach works for $\Uc_q(\glfrak_n)$ and his proof uses a neat induction on the degree of PBW basis elements and arguments from symplectic geometry. Our statement has the following distinctive features:
\begin{itemize}
\item We keep track of the dependence in $r > 0$.
\item Contrary to Sun's paper \cite{Yi16}, we only treat the very specific case of $\slfrak_2$ in Eq. \eqref{eq:oneOrbit}, where the structure of representations can be made entirely explicit. Although less general, the proof is interesting in its own right as it transparently shows that matrix coefficients of representations are indeed deformations of functions on dressing orbits. We have not found this proof in the literature.
\item We complete the picture with the convolution of two orbits in Eq. \eqref{eq:twoOrbits}.
\item We show the appearance of the crystal tensor product rule in the infinite curvature limit thanks to the above commutative diagram \eqref{eq:CDcrystals}. 
\end{itemize}

The proof of Theorem \ref{thm:staticLimits} is carried in the next four subsections: we describe in the first subsection the uniform measure on dressing orbits  and then prove in the remaining three subsections the  convergences stated in Theorem~\ref{thm:staticLimits}. But already, let us give a remark on how to understand the quantum Casimir \eqref{def:Casimir} as the quantization of a dressing orbit invariant.

\begin{rmk}
\label{rmk:CasimirIsDressingInvariant}
We have seen in Theorem \ref{thm:commutativeDiagram4Spaces} that as $\hbar \rightarrow 0$, $\left( \Uc_q^\hbar(\slfrak_2), \dagger \right)$ degenerates to the coordinate algebra $\C[ \left( SU_2^* \right)_r ]$. Moreover, in that limit, the Casimir \eqref{def:Casimir} becomes:
$$ \half \left( 4r^2 EF + \left( e^{rH} + e^{-rH} \right) \right) = \tr(x x^\dagger) \ ,$$
for $x \in \left( SU_2^* \right)_r$, which is essentially the only function of $x$ invariant under the dressing action. Indeed, for $x \in \Oc_r( \Lambda )$, we have that $\tr(x x^\dagger) = \cosh\left( r \Lambda \right)$ ; and every function in $\C[ \left( SU_2^* \right)_r ]$ invariant under the dressing action will be a function of $\Lambda$.
\end{rmk}

\subsection{Uniform measure on a dressing orbit and Archimedes' Theorem}

The following proposition describes, in coordinates, the uniform measure on a dressing orbit $\Oc_r(\Lambda) = SU_2 *_r \Lambda$. By uniform measure, we mean the measure induced by the Haar measure.

\begin{proposition}
\label{proposition:uniformDressing}
A uniform element in $\Oc_r(\Lambda) = SU_2 *_r \Lambda$ is uniquely written:
\begin{align}
\label{eq:uniformDressing}
 g = & 
   \begin{pmatrix}
   e^{\half r \mu}  & 0\\
   e^{i\Theta} \sqrt{e^{r\Lambda}+e^{-r\Lambda} - e^{r \mu} - e^{-r \mu}}
   & e^{-\half r \mu}
   \end{pmatrix} \ ,
\end{align}
with $\mu$ and $\Theta$ independent random variables. The phase $\Theta$ is uniform on $[0, 2 \pi]$ and $\mu$ follows the distribution:
$$ \P\left( \mu \in dx \right)
 = \frac{r e^{rx}dx}{e^{r\Lambda}-e^{-r\Lambda}}
   \mathds{1}_{[-\Lambda, \Lambda]} \ .
$$
\end{proposition}
In fact, the law of $\mu$ is exactly encoded by the Harish-Chandra spherical function, but we will prefer an elementary derivation based on the following result, which dates back to Archimedes' work ``On the sphere and cylinder". We provide the proof for the sake of completeness.

\begin{thm}[Archimedes' Theorem]
\label{thm:archimedes}
Consider the unit sphere $S^2 \subset \R^3$. Then:
\begin{itemize}
\item (Global version): If the sphere is inscribed inside a cylinder, the sphere has exactly the same surface area as the lateral side of the cylinder.
\item (Local version): Considering the uniform measure on the sphere, the image measure through the projection along any axis is uniform on $[-1, 1]$.
\end{itemize}
\end{thm}
\begin{proof}
The global version is naturally deduced from the local version by integration along the axis. One can also invoke the nowadays known area formulas. This is almost Archimedes' original statement as he included the bases of the cylinder, leading him to the celebrated formulation that ``the surface of the cylinder is half as large again as the surface of the sphere" \cite{H02}.

For the local version, one can perform a Jacobian computation but we prefer the following probabilistic derivation. It generalizes easily to any dimension, showing that projections along an axis are Beta distributions. Now, because of the rotation invariance of the Gaussian distribution, a uniform random variable on the unit sphere is obtained by normalizing a standard Gaussian vector $\left( \Nc_1, \Nc_2, \Nc_3 \right)$. Therefore, the measure of interest is the law of 
$$ X := \frac{\Nc_1}{\sqrt{\Nc_1^2 + \Nc_2^2 + \Nc_3^2}} \ .$$
We need to show that:
$$ \P\left( X \in dx \right) = \mathds{1}_{[-1, 1]}(x) \frac{dx}{2} \ .$$
We have that $\varepsilon = \sgn(X)$ is a Bernoulli random variable and $|X|$ is independent of $\varepsilon$, by symmetry of the Gaussian distribution. Therefore, we write:
$$ X^2 = \frac{\Nc_1^2}{\Nc_1^2 + \Nc_2^2 + \Nc_3^2} = \frac{\gamma_\half}{\gamma_\half + \gamma_1} \ ,$$
where the $\gamma_k$'s are independent Gamma variables with parameter $k$. Because of the Beta-Gamma algebra identities, $X^2$ has a Beta distribution with parameters $(\half, 1)$. As such, $\P( X^2 \in dx ) = \mathds{1}_{[0, 1]} \frac{dx}{2 x^\half}$, which implies by a change of variables that $\P( |X| \in dx ) = \mathds{1}_{[0, 1]}(x) dx$. The result follows.
\end{proof}

\begin{proof}[Proof of Proposition \ref{proposition:uniformDressing}]
Using notation \eqref{def:functionsHEF}, we have:
$$ e^{r\Lambda} + e^{-r\Lambda}
   =
   \Tr( g g^\dagger )
   =
   e^{r H} + e^{-r H} + 4 r^2 |F|^2 \ .
$$
The expression \eqref{eq:uniformDressing} then follows, by setting $H=\mu$ and $2r F = e^{i\Theta} |2r F|$. Uniqueness is obvious. Let us now compute the joint distribution of $\Theta$ and $\mu$.

We have that for any $\theta \in \R$, 
$ t = \begin{pmatrix}
   e^{i \theta}  & 0\\
   0 & e^{-i \theta}
   \end{pmatrix} \in SU_2$, 
$$ t g t^{-1}= 
   \begin{pmatrix}
   e^{\half r \mu}  & 0\\
   e^{i(\Theta-2\theta)} \sqrt{e^{r\Lambda}+e^{-r\Lambda} - e^{r \mu} - e^{-r \mu}}
   & e^{-\half r \mu}
   \end{pmatrix} 
$$
still belongs to the orbit. By definition of the uniform measure on the dressing orbit $\Oc_r(\Lambda)$, $t g t^{-1}$ is also distributed according to the uniform measure on $\Oc_r(\Lambda)$. As such for all $\theta \in \R$, we have the equality in law
 $$ \left( \mu, \Theta \right) \eqlaw
    \left( \mu, \Theta - 2 \theta \ \modulo \ 2\pi \right) \ .$$
Necessarily $\Theta$ is uniform on $[0, 2\pi]$ and is independent of $\mu$. In order to track the distribution of $\mu$, let us write:
$$ g = k_1
   \begin{pmatrix}
   e^{\half r \Lambda}  & 0\\
   0 & e^{-\half r \Lambda}
   \end{pmatrix} 
   k_2
$$
with $k_1$ Haar distributed on $SU_2$, and $k_2 \in SU_2$ is chosen in such a way that $g \in NA$. Thus:
\begin{align*}
g g^\dagger
 = & k_1
   \begin{pmatrix}
   e^{r \Lambda}  & 0\\
   0 & e^{-r \Lambda}
   \end{pmatrix} 
   k_1^\dagger\\
 = &
   \begin{pmatrix}
   e^{r \mu}
   & e^{-i\Theta} e^{\half r \mu}\sqrt{e^{r\Lambda}+e^{-r\Lambda} - e^{r \mu} - e^{-r \mu}}\\
     e^{ i\Theta} e^{\half r \mu}\sqrt{e^{r\Lambda}+e^{-r\Lambda} - e^{r \mu} - e^{-r \mu}}
   & \left( e^{r\Lambda}+e^{-r\Lambda} - e^{r \mu} - e^{-r \mu} \right) + e^{-r \mu}
   \end{pmatrix} \ ,
\end{align*}
and therefore $e^{r\mu}$ is the top-left coefficient of a Hermitian matrix, whose spectrum is $\{ e^{-r \Lambda}, e^{r \Lambda} \}$, and whose eigenvectors are Haar distributed. The space of such matrices is a sphere with principal diameter $[e^{-r \Lambda}, e^{r \Lambda}]$, and the induced measure is uniform. Hence, by Archimedes' Theorem \ref{thm:archimedes}, $e^{r\mu}$ is uniform on $[e^{-r \Lambda}, e^{r \Lambda}]$.

We conclude by the following computation. For all bounded measurable function $\varphi: \R \rightarrow \R$, the last paragraph yields:
\begin{align*}
    \E\left( \varphi(\mu) \right)
= & \int_{e^{-r \Lambda}}^{e^{r \Lambda}} \frac{du}{e^{r \Lambda}- e^{-r \Lambda}}
    \varphi\left( \frac{1}{r} \log u \right) \\
= & \int_{-\Lambda}^{\Lambda} \frac{r e^{rx} dx}{e^{r \Lambda}- e^{-r \Lambda}}
    \varphi\left( x \right) \ .
\end{align*}
\end{proof}

\subsection{Proof of Eq. (\ref{eq:oneOrbit}): One orbit}
Let $\Lambda \in \R_+^*$. $V^q(\Lambda^\hbar)$ has basis
$$ \left( e_{\hbar k} \ ; \ k = -\lfloor\Lambda/\hbar\rfloor, -\lfloor\Lambda/\hbar\rfloor+2, \cdots, \lfloor\Lambda/\hbar\rfloor-2, \lfloor\Lambda/\hbar\rfloor \right) \ .$$
Because of the PBW basis theorem, it suffices to prove the result when $\Fc = E^a F^b H^c$ for natural numbers $a$, $b$ and $c$. Thanks to the structure of the representation given in \eqref{eq:repAction}, we have that:
$$ \Fc e_{\hbar k} \in \C e_{\hbar k+2\hbar(a-b)} \ .$$
As such, we have that $\langle \Fc e_{\hbar k}, e_{\hbar k}^* \rangle = 0$ if $a \neq b$. Upon computing the coefficients, we have as $\hbar \rightarrow 0$:
$$ \langle q^{-H} \Fc e_{\hbar k}, e_{\hbar k}^* \rangle
 = \left( 1 + o(1) \right) \delta_{a,b} \ 
   e^{r(\hbar k)} \
   (\hbar k)^c
   \frac{1}{(2r)^{2a}}
   \left| e^{\kbar\Lambda}+e^{-\kbar\Lambda}
         -e^{\kbar\hbar k}-e^{-\kbar\hbar k} \right|^a\ ,$$
excepted for a bounded number of indices $k$, and each of these terms is bounded. As such:
\begin{align*}
    \Tr_{V^q(\Lambda^\hbar)}\left( q^{-H} \Fc \right)
= & \sum_{ \substack{k=-\lfloor\Lambda/\hbar\rfloor \\ \lfloor\Lambda/\hbar\rfloor - k \text{ even} } }^{\lfloor\Lambda/\hbar\rfloor}
    \langle q^{-H} \Fc e_{\hbar k}, e_{\hbar k}^* \rangle\\
= & \ \Oc(1) +
    \sum_{ \substack{k=-\lfloor\Lambda/\hbar\rfloor \\ \lfloor\Lambda/\hbar\rfloor - k \text{ even} } }^{\lfloor\Lambda/\hbar\rfloor}
	\delta_{a,b}    
    \left( 1 + o(1) \right)
    \ e^{r(\hbar k)} \
    (\hbar k)^c \left| \frac{e^{\kbar\Lambda}+e^{-\kbar\Lambda}-e^{\kbar\hbar k}-e^{-\kbar\hbar k}}{(2r)^2} \right|^a \\
= & \ \Oc(1) +
    \delta_{a,b} \frac{1+o(1)}{\hbar} \int_{-\Lambda}^{\Lambda} e^{rx} \ x^c \ 
    \left| \frac{e^{\kbar\Lambda}+e^{-\kbar \Lambda}-e^{\kbar x}-e^{-\kbar x} }{(2r)^2} \right|^a dx \ ,
\end{align*}
where we recognized a Riemann sum at the last step. Therefore, if we divide by 
$$ \Tr_{V^q(\Lambda^\hbar)}\left( q^{-H} \right)
   =
   \Oc(1) +
   \frac{1+o(1)}{\hbar}
   \frac{ e^{r \Lambda} - e^{-r \Lambda} }{r} \ ,
$$
we obtain:
\begin{align*}
    \frac{ \Tr_{V^q(\Lambda^\hbar)}\left( q^{-H} \Fc \right) }
         { \Tr_{V^q(\Lambda^\hbar)}\left( q^{-H} \right) }
= & \ \Oc(\hbar) +
    \delta_{a,b} \left( 1+o(1) \right) \int_{-\Lambda}^{\Lambda} 
    \frac{ r e^{rx} }{ e^{r \Lambda} - e^{-r \Lambda} } \ x^c \ 
    \left| \frac{ e^{\kbar\Lambda}+e^{-\kbar \Lambda}-e^{\kbar x}-e^{-\kbar x} }{(2r)^2}\right|^a dx \ ,
\end{align*}

Now, thanks to Proposition \ref{proposition:uniformDressing}, any element $g$ which is uniform in $\Oc_\kbar( \Lambda ) = SU_2 \cdot \Lambda$ can be written:
$$ g = \left( 
       \begin{pmatrix}
                e^{\half \kbar H} & 0 \\
        2 \kbar F                 & e^{-\half \kbar H}
       \end{pmatrix}
       ,
       \begin{pmatrix}
                e^{-\half \kbar H} & 2 \kbar E \\
                0                  & e^{\half \kbar H}
       \end{pmatrix}
       \right)
$$
where $\overline{ E(g) } = F(g) = \frac{1}{2r} e^{i \Theta} \sqrt{e^{\kbar\Lambda}+e^{-\kbar \Lambda}-e^{\kbar H(g)}-e^{-\kbar H(g)}}$, $\Theta$ is uniform on $[0, 2\pi]$ and $H(g)$ follows the prescribed distribution. As such:
\begin{align*} 
\frac{ \Tr_{V^q(\Lambda^\hbar)}\left( q^{-H} \Fc \right) }
        { \Tr_{V^q(\Lambda^\hbar)}\left( q^{-H} \right) }
 & = \Oc(\hbar) + \left( 1 + o(1) \right) \E\left( E(g)^a F(g)^b H(g)^c \right) \\
 & = o(1) + \E\left( f(g)  \right) \ .
\end{align*}
 This concludes the proof of Eq. \eqref{eq:oneOrbit} for one orbit. Notice that the proof holds for all $r > 0$, uniformly on compact intervals.

\subsection{Proof of Eq. (\ref{eq:twoOrbits}): Convolution of two orbits}
Sweedler's notation for the coproduct $\Delta_r$ on $\Uc_q^\hbar(\slfrak_2)$ is written:
$$ \Delta_r \Fc = \sum_{(\Fc)} \Fc_1 \otimes \Fc_2 \ . $$
Hence, since $q^{-H}$ is group-like and $\Delta_r$ is a morphism:
$$ \Delta_r \left( q^{-H} \Fc \right)
 = \sum_{(\Fc)} \left( q^{-H} \Fc_1 \right) \otimes \left( q^{-H} \Fc_2 \right) \ . $$
Upon considering $\left( e^i_{\hbar k} \right)_k$ as the basis of $V^q(\Lambda_i^\hbar)$, for $i=1,2$, this gives:
\begin{align*}
    \Tr_{V^q(\Lambda_1^\hbar) \otimes V^q(\Lambda_2^\hbar) }\left( q^{-H} \Fc \right)
= & \sum_{i,j} \langle \Delta_r q^{-H} \Fc \cdot e^1_{\hbar i} \otimes e^2_{\hbar j}, \left( e^1_{\hbar i} \otimes e^2_{\hbar j} \right)^* \rangle\\
= & \sum_{(\Fc)} \sum_{i,j} \langle q^{-H} \Fc_1 e^1_{\hbar i}, \left( e^1_{\hbar i} \right)^* \rangle \langle q^{-H} \Fc_2 \cdot e^2_{\hbar j}, \left( e^2_{\hbar j} \right)^* \rangle\\
= & \sum_{(\Fc)} \Tr_{V^q(\Lambda_1^\hbar)}(q^{-H} \Fc_1)
             \ \Tr_{V^q(\Lambda_2^\hbar)}(q^{-H} \Fc_2) \ .
\end{align*}
Again, because $q^{-H}$ is group-like, we have:
$$ \Tr_{V^q(\Lambda_1^\hbar) \otimes V^q(\Lambda_2^\hbar)}( q^{-H} )
   =
   \Tr_{V^q(\Lambda_1^\hbar)}( q^{-H} ) \Tr_{V^q(\Lambda_2^\hbar)}( q^{-H} ) \ ,
$$
and upon dividing by that quantity, we obtain the equality:
$$ 
  \frac{
  \Tr_{V^q(\Lambda_1^\hbar) \otimes V^q(\Lambda_2^\hbar) }\left( q^{-H} \Fc \right)
  }{
  \Tr_{V^q(\Lambda_1^\hbar) \otimes V^q(\Lambda_2^\hbar)}( q^{-H} )
  }
= \sum_{(\Fc)}
  \frac{
  \Tr_{V^q(\Lambda_1^\hbar)}\left( q^{-H} \Fc_1 \right)
  }{
  \Tr_{V^q(\Lambda_1^\hbar)}( q^{-H} )
  }
  \frac{
  \Tr_{V^q(\Lambda_2^\hbar)}\left( q^{-H} \Fc_2 \right)
  }{
  \Tr_{V^q(\Lambda_2^\hbar)}( q^{-H} )
  } \ .
$$
We are ready to invoke the semi-classical limit Eq. \eqref{eq:oneOrbit} for one orbit. Let $f_j := \Fc_j \modulo \hbar$ for $j=1,2$. Also consider a random variable  $x$ uniformly distributed on $\Oc_r(\Lambda_1)$ and an independent random variable $y$ uniformly distributed on $\Oc_r(\Lambda_2)$. We have, as $\hbar \rightarrow 0$:
\begin{align*}
   \frac{
  \Tr_{V^q(\Lambda_1^\hbar) \otimes V^q(\Lambda_2^\hbar) }\left( q^{-H} \Fc \right)
  }{
  \Tr_{V^q(\Lambda_1^\hbar) \otimes V^q(\Lambda_2^\hbar)}( q^{-H} )
  }
& =  o(1) + \sum_{(\Fc)} \E \left( f_1(x) \right) \E \left( f_2(y) \right) \\
& =  o(1) + \E \left( \sum_{(\Fc)} f_1( x ) f_2( y ) \right) \ .
\end{align*}
Now, recall from Remark \ref{rmk:DeltaNotQuantum} that the coproduct $\Delta_r$ on $\Uc_q^\hbar(\slfrak_2)$ has nothing inherently quantum! It is the coproduct associated to the dual Poisson-Lie group $(SL_2^*)_r$ and as such, by definition:
$$ \sum_{(\Fc)} f_1( x ) f_2( y )
  = f\left( x *_r y  \right) \ ,
$$
where $f = \Fc \modulo \hbar$. The result \eqref{eq:twoOrbits} follows.

\subsection{Proof of Eq. (\ref{eq:CDcrystals}): Crystals in the infinite curvature limit}
The bottom arrow is a specialization of the semi-classical limit for two orbits given in Eq. \eqref{eq:twoOrbits}. 

The top arrow is the convergence of a sum indexed by $\Bc\left(\Lambda_1^\hbar \right) \times \Bc\left( \Lambda_2^\hbar \right)$. Thanks to the identification \eqref{eq:BcIdentification}, this is a sum on a mesh which discretizes the rectangle $[-\Lambda_1, \Lambda_1] \times [-\Lambda_2, \Lambda_2]$. The statement is thus the convergence of a Riemann sum.

The left arrow is essentially a restatement of what crystal bases are. 
Indeed, as $q\to0$, the only remaining basis vectors of the representation $V^q\left(\Lambda_1^\hbar\right)\otimes V^q\left(\Lambda^\hbar_2\right)$ are monomial and thus identify with crystal elements of the form $b_1\otimes b_2\in  \Bc\left(\Lambda_1^\hbar \right ) \otimes \Bc \left( \Lambda_2^\hbar \right) $. The convergence follows, since as $r\to\infty$, $\Lambda^{r, \hbar}$ acts as the highest weight $\hw$ while $H$ acts as the weight $\wt$ in any representation.

Only the right arrow remains to be detailed. To that endeavor, we start by invoking Proposition \ref{proposition:uniformDressing} giving the explicit description of uniform elements $g_j \in \Oc_r(\Lambda_j)$, $j=1,2$, on two dressing orbits:
$$  g_j =
   \begin{pmatrix}
   e^{\half r \mu_i}  & 0\\
   e^{i\Theta_j} \sqrt{e^{r\Lambda_j}+e^{-r\Lambda_j} - e^{r \mu_j} - e^{-r \mu_j}}
   & e^{-\half r \mu_j}
   \end{pmatrix} \ .$$
The product is:
$$  g_1 g_2 =
   \begin{pmatrix}
   e^{\half r (\mu_1+\mu_2)}  & 0\\
   2 r F(g_1 g_2)
   & e^{-\half r (\mu_1+\mu_2)}
   \end{pmatrix} \ ,$$
with
\begin{align*}
   2r F(g_1 g_2) = & \ e^{i\Theta_1} \sqrt{e^{r\Lambda_1}+e^{-r\Lambda_1} - e^{r \mu_1} - e^{-r \mu_1}} e^{\half r \mu_2} \\
                   & \quad + e^{i\Theta_2} e^{-\half r \mu_1} \sqrt{e^{r\Lambda_2}+e^{-r\Lambda_2} - e^{r \mu_2} - e^{-r \mu_2}} \ .
\end{align*}
When $r \rightarrow \infty$, as soon as $|\mu_1|<\Lambda_1$, $|\mu_2|<\Lambda_2$, and $e^{i\Theta_1}\not=-e^{i\Theta_2}$, the tropicalization trick \eqref{eq:tropTrick} yields the estimate:
\begin{align}
\label{eq:tropEstimate}
    4 r^2 \left| F(g_1 g_2) \right|^2 
= & \left| e^{i\Theta_1} e^{\half r (\Lambda_1 + \mu_2  + o(1))}
         + e^{i\Theta_2} e^{\half r (-\mu_1 + \Lambda_2 + o(1))} \right|^2 
\nonumber
\\
= & \exp\left[ r \max\left( \Lambda_1 + \mu_2, \ -\mu_1 + \Lambda_2 \right) + o(r) \right] \ .
\end{align}

Now, we compute the scalar $\Lambda = \Lambda(g_1 g_2) \in \R$ such that $g_1 g_2$ belongs to the orbit $\Oc_r\left( \Lambda \right)$. We have:
\begin{align*}
   \Lambda (g_1 g_2)
& =  \frac{1}{r} \Argcosh \circ \tr\left[ (g_1 g_2) (g_1 g_2)^\dagger \right]\\
& =  \frac{1}{r} \Argcosh \left( \cosh\left( r(\mu_1 + \mu_2) \right)
                + 2 r^2 \left| F(g_1 g_2) \right|^2 \right) \\
& \overset{r \rightarrow \infty}{\longrightarrow}  \max\left( \Lambda_1 + \mu_2, \ -\mu_1 + \Lambda_2 \right) \ ,
\end{align*}
thanks to the estimate \eqref{eq:tropEstimate}, the asymptotic $\Argcosh(x) = \log 2x+o(1)$ as $x\to\infty$ and the tropicalization trick (\ref{eq:tropTrick}) again.
This  limit holds for (Lebesgue) almost every $\mu_1 \in [-\Lambda_1, \Lambda_1]$, $\mu_2 \in [-\Lambda_2, \Lambda_2]$ and  $\left( \Theta_1, \Theta_2 \right) \in [0, 2\pi]^2$.

We thus conclude, that for independent random variables $g_1$ and $g_2$,  uniformly distributed on $\Oc_r(\Lambda_1)$ and $\Oc_r(\Lambda_2)$ respectively, we have:
\begin{align*}
   & \frac{\E \left(  e^{-rH(g_1g_2)} \varphi\circ\Lambda(g_1g_2) \, \psi\circ H (g_1g_2) \right)}{\E \left(  e^{-rH(g_1g_2)} \right)  }   \\
 = & \frac{1}{ 4 \Lambda_1 \Lambda_2 }
          \int_{-\Lambda_1}^{\Lambda_1} \int_{-\Lambda_2}^{\Lambda_2} d\mu_1 d\mu_2
         \int_{0}^{2\pi} \int_{0}^{2\pi}\frac{d \Theta_1}{2\pi} \frac{d \Theta_2}{2\pi} \varphi \circ \Lambda (g_1 g_2) \ \psi\left( \mu_1 + \mu_2 \right) \\
 \overset{r \rightarrow \infty}{\longrightarrow} & \frac{1}{ 4 \Lambda_1 \Lambda_2 }
          \int_{-\Lambda_1}^{\Lambda_1} \int_{-\Lambda_2}^{\Lambda_2} d\mu_1 d\mu_2
            \ \varphi \circ \max\left( \Lambda_1 + \mu_2, \ -\mu_1 + \Lambda_2 \right) \ \psi\left( \mu_1 + \mu_2 \right) \ .
\end{align*}
In the above computation, the equality is obtained thanks to Proposition~\ref{proposition:uniformDressing} and the limit follows from the dominated convergence theorem.

\section{Dynamical semi-classical limits: Proof of Main Theorem \ref{thm:main}}
\label{section:dynamics}

This section is devoted to the proof of Theorem~\ref{thm:main}. Let us start by recalling what is already known among the convergences $\hbar \rightarrow 0$ or $r \rightarrow 0$ or $\infty$.

Theorem~\ref{thm:Biane_CV} of Biane gives the convergence in the semi-classical limit $\hbar\to0$ of the non-commutative process $\left( x^{0,\hbar}_t \ ; \ t\geq 0 \right)$ to the flat Brownian motion $\left( x^0_t \ ; \ t\geq 0 \right)$, as well as the convergence of the Casimir process $\left( \Lambda_t^{0,\hbar} \ ; \ t\geq0 \right)$ to the Bessel-$3$ process. The convergence of the dynamic $\left( g_t^{r,\hbar} \ ; \ t\geq0 \right)$ on $\Uc_q^\hbar(\slfrak_2)$ to Pitman's Theorem~\ref{thm:PitmanDiscrete} is also a rephrasing of Biane's Theorem~\ref{thm:Biane_Pitman} as the $q\to0$ limit corresponds to our curvature $r$ going to infinity. The $r\to0$ regime of  $\left( g_t^{r,\hbar} \ ; \ t\geq0 \right)$ is obviously given by the degeneration ``$q\to1$" of $\Uc_q^\hbar(\slfrak_2)$ to $\Uc^\hbar(\slfrak_2)$. In the semi-classical world, the convergence in both regimes $r\to0$ and $r\to\infty$ is Theorem~\ref{thm:BJ} by Bougerol and Jeulin. 

As already mentioned, since extracting their result is not trivial, we provide a complete proof of the convergences in subsection~\ref{subsection:BJ}. In subsection~\ref{subsection:qLR}, we compute the law of the dynamic $\left( \Lambda^{r,\hbar}_t \ ; \ t\geq0 \right)$ induced by the Casimir operator for generic $r>0$. Subsection \ref{subsection:qPitman} is the technical part where we provide the proof of the convergence of $\left( \left( g_t^{r,\hbar}, \Lambda^{r,\hbar}_t \right) \ ; \ t\geq0 \right)$ to $\left( \left( g_t^r, \Lambda^r_t \right) \ ; \ t\geq0 \right)$ as $\hbar \rightarrow 0$. Finally, we conclude in subsection \ref{subsection:rIndependence} with a conceptual explanation of why the law of $\Lambda^{r, \hbar}$ is independent from $r>0$.

\subsection{The proof of Bougerol and Jeulin's Theorem~\ref{thm:BJ}}
\label{subsection:BJ}

The fact that the distribution of $\Lambda^r$ does not depend on $r>0$ is rather subtle. As announced earlier, we leave the matter for subsection \ref{subsection:rIndependence}, where we will benefit from a more global vision of the problem.

For the first expression given in Eq. \eqref{eq:BJdynamic}, consider Eq. \eqref{eq:BJ_process} and write:
$$ \cosh\left(r \Lambda^r_t \right) = \tr g_t g_t^\dagger = \cosh(r X_t) + \half r^2 \left|e^{\half r X_t} \int_0^t e^{-r X_s} (dY_s + i dZ_s) \right|^{2} \ .$$
Taking the inverse of $\cosh$ yields the result.

Now, let us prove the existence of the limits $r \rightarrow 0$ and $r \rightarrow \infty$ for $\Lambda^r$, which are given in Eq. \eqref{eq:BJlimits}. We shall use the notation $o(1)$ to denote a sequence of random variables going to zero almost surely, and $o_\P(1)$ if this convergence is in probability.

The computation of $\Lambda^{r=0}$ goes via the following limit argument. For fixed $t>0$, we start with the estimate:
$$ 
   \int_0^t e^{-r X_s} (dY_s + i dZ_s) = Y_t + i Z_t + o_\P(1) \ ,
$$
which is an easy consequence of computing the $L^2(\Omega, \P)$ norm thanks to the Itô isometry property. Then, via Taylor expansions:
\begin{align*}
    \cosh \left( r \Lambda_t^{r} \right)
& = \cosh(r X_t) + \half r^2 \left|e^{\half r X_t} \int_0^t e^{-r X_s} (dY_s + i dZ_s) \right|^{2} \\
& = 1 + \half r^2 X_t^2 + r^2 o(1) + \half r^2 \left| \left( 1 + o(1) \right) \left( Y_t + i Z_t + o_\P(1) \right) \right|^{2} \\
& = 1 + \half r^2 \left( X_t^2 + Y_t^2 + Z_t^2 + o_\P(1) \right) .
\end{align*}
Since $\Argcosh(1+\half x) = \sqrt{x} + \Oc(x)$ for $x \rightarrow 0$, we have as $r \rightarrow 0$:
\begin{align*}
    \Lambda_t^{r} 
& = \frac{1}{r} \Argcosh\left[ 1 + \half r^2 \left( X_t^2 + Y_t^2 + Z_t^2 + o_\P(1) \right) \right] \\
& = \sqrt{X_t^2 + Y_t^2 + Z_t^2 + o_\P(1)} + o_\P(1) \\
& = \sqrt{X_t^2 + Y_t^2 + Z_t^2} + o_\P(1) \ .
\end{align*}
Therefore, we have the limit in probability for all $t>0$:
$$ \P - \lim_{r \rightarrow 0} \Lambda_t^{r} = \sqrt{X_t^2 + Y_t^2 + Z_t^2} \ ,$$
which yields the result for $\Lambda^{r=0}$.

For the computation of $\Lambda^{r=\infty}$, we fix $t>0$ and work conditionally to the trajectory $\left( X_s \ ; \ s \in \R_+ \right)$. By virtue of classical Wiener integration with respect to Brownian motion, there exists a complex standard Gaussian $\Nc^\C$ such that
$$ \int_0^t e^{-r X_s} (dY_s + i dZ_s) 
   \eqlaw
   \Nc^\C \ \sqrt{ \int_0^t e^{-2 r X_s} ds } \ .
$$
Indeed, $(Y, Z)$ is independent from $X$ and conditioning with respect to $\left( X_s \ ; \ 0 \leq s \leq t \right)$ allows to treat the integrand as a deterministic function. As such:
\begin{align*}
    \cosh \left( r \Lambda_t^{r} \right)
& = \cosh(r X_t) + \half r^2 \left| e^{\half r X_t} \int_0^t e^{-r X_s} (dY_s + i dZ_s) \right|^{2} \\
& = \half e^{r X_t} + \half e^{-r X_t} + \half r^2 |\Nc^\C|^2 e^{r X_t} \int_0^t e^{-2 r X_s} ds \ .
\end{align*}
Notice that the above quantity is a sum of positive terms and diverges to $\infty$ as $r \rightarrow \infty$, since either $X_t>0$ or $-X_t>0$ almost surely. Moreover as $x \rightarrow \infty$, we have $\Argcosh(x) = o(1) + \log 2x$, hence:
\begin{align*}
    \Lambda_t^{r} 
& = o\left( 1 \right) 
    +
    \frac{1}{r}
    \log \left( e^{r X_t} + e^{-r X_t} + r^2 |\Nc^\C|^2 e^{r X_t} \int_0^t e^{-2 r X_s} ds
         \right) \ .
\end{align*}

Now, the crux of the argument consists in invoking the Laplace method which can be seen as the continuous analogue of the tropicalization trick \eqref{eq:tropTrick}. It states that for all continuous functions $f$:
$$ \lim_{r \rightarrow \infty} 
   \frac{1}{r}\log \int_0^t e^{-r f(s)} ds
   =
   - \inf_{0 \leq s \leq t} f(s) \ .$$
In particular, for $f = X$, we have: 
\begin{align*}
    \Lambda_t^{r}
& = o\left( 1 \right) 
    +
    \frac{1}{r}
    \log \left( e^{r X_t} + e^{-r X_t} + r^2 |\Nc^\C|^2 e^{r \left( o(1) + X_t - 2 \inf_{0 \leq s \leq t } X_s \right) }
         \right) \ 
    ,
\end{align*}
Invoking the tropicalization trick \eqref{eq:tropTrick} again, we obtain:
\begin{align*}
    \Lambda_t^{r}
& = o(1) + \max\left( |X_t|, o(1) + X_t - 2 \inf_{0 \leq s \leq t} X_s \right) \ .
\end{align*}
Upon noticing the almost sure inequality $X_t - 2 \inf_{0 \leq s \leq t} X_s > |X_t|$, we have shown that there is an almost sure limit:
$$ \lim_{r \rightarrow \infty} \Lambda_t^r = X_t - 2 \inf_{0 \leq s \leq t} X_s \ ,$$
which implies the limit in probability given in the theorem.

\subsection{The Littlewood-Richardson dynamic}
\label{subsection:qLR}

Here we compute the distribution of the commutative process $\left( M_n(C^{r, \hbar}) \ ; \ n \geq 1 \right)$. From the facts that the $M_n$'s are algebra homomorphisms and that $C^{r, \hbar}$ lies in the center of $\Uc_q^\hbar\left(\slfrak_2\right)$, it is clear that $M_n(C^{r, \hbar})$ commutes with any $M_k(C^{r, \hbar})$ for $k\leq n$. Moreover, from its definition \eqref{def:Casimir}, $C^{r, \hbar}$ is Hermitian with respect to the involution $\dagger$. As such, one can simultaneously diagonalize the operators $\left( M_1(C^{r, \hbar}),\ldots,M_n(C^{r, \hbar}) \right)$. Recall that $C^{r, \hbar}$ acts on the irreducible representation $V^q\left( \Lambda^\hbar \right)$ as the constant $C^{r, \hbar}(\Lambda^\hbar)$ defined in Eq. \eqref{eq:casimirConstant}. Because of the rigidity of quantum groups, the Littlewood-Richardson rule is unchanged which will imply  that $\left( M_n(C^{r, \hbar}) \ ; \ n \geq 0 \right)$ follows the same dynamic regardless of the parameter $r > 0$. The following lemma is an analogue of Lemma~3.3 in \cite{B08}.

\begin{lemma}
Let $J=\{\lambda_1,\ldots,\lambda_n   \}\subset  \N$ with $\lambda_1=1$ and $| \lambda_{k+1}-\lambda_k | =1$.  There exists a subspace $V^q_{J}\subset {(\C^2)}^{\otimes n}$, isomorphic to the irreducible representation $V^q(\hbar\lambda_n)$ of dimension $\lambda_n+1$,  which is a common eigenspace of $M_1(C^{r, \hbar}),\ldots,M_n(C^{r, \hbar})$ with corresponding eigenvalues $C^{r, \hbar}(\hbar\lambda_1),\ldots, C^{r, \hbar}(\hbar\lambda_n)$.
\end{lemma}
\begin{proof}
The proof is done by induction on $n$. For $n=1$, $V^q_J=V^q(\hbar)\approx \C^2$ is indeed an eigenspace of $M_1(C^{r, \hbar})=C^{r, \hbar}$ and the base case is proven.

Now, suppose the assertion is true for fixed $n\geq1$. Let $\lambda_1,\ldots, \lambda_{n+1}$ with $| \lambda_{k+1}-\lambda_k | =1$. Let $J'=\{\lambda_1,\ldots,\lambda_n   \}$. Because of the induction hypothesis, there exists a subspace $V^q_{J'} \subset V^q(\hbar)^{\otimes n}$ which is a common eigenspace of  $M_1(C^{r, \hbar}),\ldots,M_n(C^{r, \hbar})$ with eigenvalues $C^{r, \hbar}(\hbar \lambda_1),\ldots, C^{r, \hbar}(\hbar \lambda_n)$. Moreover $V^q_{J'}\simeq V^q(\hbar\lambda_n)$. Upon tensoring by $V^q(\hbar)$, we obtain $V^q_{J'} \otimes V^q(\hbar) \subset V^q(\hbar)^{\otimes (n+1)}$ which is still a common eigenspace for the previous operators and with exactly the same eigenvalues.

Here, we invoke the Littlewood-Richardson rule which governs the decomposition of tensor products into irreducibles. We have
\begin{align}
\label{eq:LRrule}
V^q_{J'} \otimes V^q(\hbar) \approx \,  & V^q(\hbar\lambda_n)\otimes V^q(\hbar) \approx V^q\left(\hbar(\lambda_{n}+1)\right) \oplus V^q\left(\hbar(\lambda_{n}-1)\right) ,
\end{align}
irrespective of the parameter $q = e^{-r}$, a property which is referred to as the {\it rigidity of quantum groups}.

On the irreducible representation $V^q\left(\hbar(\lambda_{n}+1)\right)$ (resp. $V^q(\hbar(\lambda_{n}-1))$), $M_{n+1}(C^{r, \hbar})$ acts as the constant $C^{r, \hbar}(\hbar(\lambda_n+ 1))$ (resp. $C^{r, \hbar}(\hbar(\lambda_n- 1))$). Moreover, one can construct $V^q(\hbar(\lambda_n + 1))$ and $V^q(\hbar(\lambda_n - 1))$ as follows (see \cite[p.~35]{klimyk-vilenkin}). Let 
$$ \left( e^{l}_{\hbar k} \ ; \ k=-l,-l+2,\ldots,l-2,l \right)$$ 
be the weight basis of $V^q(\hbar l)$. Then $V^q(\hbar(\lambda_n + 1))$ is spanned by the vectors $e^{\lambda_n+1}_{\hbar m}$ defined by
\begin{align*}
e^{\lambda_n+1}_{\hbar m} & :=  \operatorname{CGC}^\hbar_q(\lambda_n,1,\lambda_n+1;m-1,1,m) \ e^{\lambda_n}_{\hbar(m-1)} \otimes e^1_{\hbar} \\
& \quad \quad + \operatorname{CGC}^\hbar_q(\lambda_n,1,\lambda_n+1;m+1,-1,m) \ e^{\lambda_n}_{\hbar(m+1)} \otimes e^1_{-\hbar} \ ,
\end{align*}
for $m=-(\lambda_n+1),-\lambda_n+1,\ldots,\lambda_n-1,\lambda_n+1$. On the other hand, $V^q(\hbar(\lambda_n - 1))$ is spanned by the vectors $e^{\lambda_n-1}_{\hbar m}$ defined by
\begin{align*}
e^{\lambda_n-1}_{\hbar m} & := \operatorname{CGC}^\hbar_q(\lambda_n,1,\lambda_n-1;m-1,1,m) \ e^{\lambda_n}_{\hbar(m-1)} \otimes e^1_\hbar \\
&\quad\quad+ \operatorname{CGC}^\hbar_q(\lambda_n,1,\lambda_n-1;m+1,-1,m) \ e^{\lambda_n}_{\hbar(m+1)} \otimes e^1_{-\hbar} \ ,
\end{align*}
for $m=-(\lambda_n-1),-\lambda_n+3,\ldots,\lambda_n-3,\lambda_n-1$. Here the explicit coefficients
$\operatorname{CGC}^\hbar_q(\cdot \ ; \ \cdot)$ are the so-called Clebsch-Gordan coefficients. 

Since for $k\leq n$, the operator $M_k(C^{r, \hbar})$ acts only on the first $k$-th legs of the tensor product, $V^q(\hbar(\lambda_n + 1))$ and $V^q(\hbar(\lambda_n - 1))$ are still eigenspaces of $M_1(C^{r, \hbar}),\ldots,M_n(C^{r, \hbar})$, so the lemma is proved letting $V^q_J=V^q(\hbar(\lambda_n + 1))$ for $\lambda_{n+1}=\lambda_n+1$ and $V^q_J=V^q(\hbar(\lambda_n - 1))$ for $\lambda_{n+1}=\lambda_n-1$.
\end{proof}

Using the above lemma, for any polynomial $f$ in $n$ variables, the spectral theorem gives
$$
   \tau \left[ f\left(M_1(C^{r, \hbar}),\ldots,M_n(C^{r, \hbar})  \right)   \right]
 = \sum_{J=\{\lambda_1,\ldots,\lambda_n\}}
   f\left(C^{r, \hbar}(\hbar\lambda_1),\ldots,C^{r, \hbar}(\hbar\lambda_n)\right) \frac{1}{2^n} \Tr (\Pi_J) \ ,
$$
where $\Pi_J$ is the projector onto $V^q_J$ which has dimension $\lambda_n+1$. Thus $\Tr (\Pi_J)=\lambda_n+1$ and the probability of $(M_1(C^{r, \hbar}),\ldots,M_n(C^{r, \hbar}))$ having  the trajectory $(C^{r, \hbar}(\hbar\lambda_1),\ldots,C^{r, \hbar}(\hbar\lambda_n))$ is equal to 
$$
\frac{\lambda_n+1}{2^n} = \frac{\lambda_1+1}{2} \frac{\lambda_2+1}{2(\lambda_1+1)} \cdots \frac{\lambda_n+1}{2(\lambda_{n-1}+1)} \ .
$$
We deduce that the process $\left( M_n(C^{r, \hbar}) \ ; \ n \geq 1 \right)$ is a Markov chain on the state space $\left\{ C^{r, \hbar}(\hbar k) \ ; \ k\in\N \right\}$ with transitions given by
$$
p\left(C^{r, \hbar}(\hbar\lambda),C^{r, \hbar}(\hbar(\lambda+1))\right)=\frac{\lambda+2}{2(\lambda+1)}, \quad p\left(C^{r, \hbar}(\hbar\lambda),C^{r, \hbar}(\hbar(\lambda-1))\right)=\frac{\lambda}{2(\lambda+1)} \ .
$$
This shows that the dynamic induced by the Casimir operator has transitions which do not depend on $r$ and are given as in Pitman's Theorem~\ref{thm:PitmanDiscrete}. Consequently, using (\ref{eq:def_Lambda}), we have that the dynamic of the process $\left( \Lambda_n^{r,\hbar} \ ; \ n \geq 0\right)$ is independent of the curvature $r$ and is a Markov chain on highest weights $\hbar\N$ with transitions~\eqref{eq:LR_transitions}.

\medskip

\subsection{Semi-classical limit of quantum walks on \texorpdfstring{$\Uc_q^\hbar( \slfrak_2)$}{the quantum group} and Brownian motion on \texorpdfstring{$\left(SU_2^*\right)_\kbar$}{the dual Poisson Lie group}}
\label{subsection:qPitman}

Recall that $H$ is defined in a completion of $\Uc_q^\hbar( \slfrak_2)$ through $K=e^{-rH} \in \Uc_q^\hbar( \slfrak_2)$, and define $\beta = \Re \beta + i \Im \beta \in \Uc_q^\hbar( \slfrak_2)$ where
$$
   \Re \beta = 2\Re F =  F+F^\dagger=F+E\ , \quad
   \Im \beta = 2\Im F =  i(F^\dagger-F)= i(E-F) \ .
$$

Using the action on the irreducible representation $V^q(\hbar)\approx \C^2$ given by (\ref{eq:repAction}), the non-commutative self-adjoint random variables $H$, $\Re \beta$ and $\Im \beta$ are represented thanks to $(V^q(\hbar), \rho_\hbar)$ by the following matrices:
$$
\rho_\hbar(H) = \hbar \begin{pmatrix} 1 & 0 \\ 0 & -1 \end{pmatrix}, \quad
\rho_\hbar(\Re \beta) = \beta_r^\hbar\begin{pmatrix}  0 & 1 \\ 1 & 0 \end{pmatrix}, \quad
\rho_\hbar(\Im \beta) = \beta_r^\hbar\begin{pmatrix}  0 & i \\ -i & 0 \end{pmatrix} ,
$$
with $\beta_r^\hbar= \alpha_r^\hbar (e^{r\hbar}-e^{-r\hbar})\simeq \hbar$ as $\hbar\to0$. Accordingly, one sees that $H$, $\Re \beta$ and $\Im \beta$ are renormalized Pauli matrices. Since the state $\tau$ is the normalized trace on $\End(V^q(\hbar))$, their marginal distributions $\left( \Lc(H), \Lc(\Re F), \Lc(\Im F) \right)$ are given by their spectral measures:
$$
       \Lc(H) = \half \delta_{\hbar} + \half  \delta_{-\hbar} \ ,
 \quad \Lc(\Re F) = \half \delta_{\beta^\hbar_r} +\half \delta_{-\beta^\hbar_r} \ ,
 \quad \Lc(\Im F) = \half \delta_{\beta^\hbar_r} +\half \delta_{-\beta^\hbar_r} \ .
$$
As such, $H$ is distributed as $\hbar$ times a Bernoulli random variable and $\Re \beta$ and $\Im \beta$ are distributed as $\beta_r^\hbar$ times a Bernoulli random variable.

Now, let us adopt the following convenient but abusive notations:
\begin{align*}
M_n( \beta ) & := M_n( \Re \beta ) + i M_n( \Im \beta ) \ , \\
M_n( \Re \beta ) & :=  \sum_{k=1}^n 1^{\otimes k-1} \otimes \Re \beta \ , \\
M_n( \Im \beta ) & :=  \sum_{k=1}^n 1^{\otimes k-1} \otimes \Im \beta \ .
\end{align*}
This is an abuse of notation because the right-hand side has no reason to belong to the image of $M_n$. Moreover, it is different from the result obtained by applying the measurement operator $M_n$ to $\beta$, as $\beta$ is not primitive. We will nevertheless pretend that it is the case, since it will be convenient.  Furthermore, since $H$ is primitive, one has
$$
M_n(H)=\sum_{k=1}^n 1^{\otimes k-1} \otimes H.
$$
Since the random variables $\left( 1^{\otimes k-1} \otimes A \ ; \ k\geq1 \right)$, for each fixed $A\in \{H,\Re \beta,\Im \beta  \}$, are commuting and thus independent, each of the walks  $M_n(H)$, $M_n(\Re \beta)$, $M_n(\Im \beta)$ is a sum of independent Bernoulli random variables, and this triplet corresponds to the quantum Bernoulli walk of Biane appearing in Theorem~\ref{thm:Biane} but with a different normalization. As such, Donsker's invariance principle gives that each of these walks converges towards a Brownian motion. Moreover, using the commutation relations of \eqref{def:heteroQuantumGroup} defining the algebra $\Uc_q^\hbar\left(\slfrak_2\right)$, one sees that all the brackets $[M_n(A),M_n(B)]$ for given $A,B \in \{H,\Re \beta,\Im \beta\}$ are of order $\hbar$ and thus converge to zero in the semi-classical limit $\hbar\to0$. This is the crucial argument in Biane's Theorem, giving that the quantum walk $\left( M_{t/\hbar^2}(H), M_{t/\hbar^2}(\Re \beta), M_{t/\hbar^2}(\Im \beta) \ ; \ t \geq 0 \right)$ converges, in the semi-classical limit, towards a classical Euclidean $3$-dimensional Brownian motion $(X,Y,Z)$. Notice that compared to us, the Planck constant is not explicit in Biane's argument, as it is hidden in the diffusive rescaling.

In any case, recalling that $M_n(\beta)=M_n(\Re \beta)+iM_n(\Im \beta)$, we get the convergence in distribution
\begin{align}
\label{eq:walks_cv}
\left( M_{t/\hbar^2}( H ), M_{t/\hbar^2}(\beta) \ ; \ t \geq 0 \right)
& 
\overset{\hbar \rightarrow 0}{\longrightarrow}
\left( X_t, \beta_t^\C \ ; \  t \geq 0\right) \ ,
\end{align}
where $\beta^\C_t=Y_t+iZ_t$ is a standard complex Brownian motion and $X_t$ a standard real Brownian motion. We made such a choice of notation so that one can recognize the driving processes associated to the Bougerol-Jeulin dynamic:
       $$ g_t^\kbar = 
          \begin{pmatrix}
          e^{\half \kbar X_t} & 0 \\
          r e^{\half \kbar X_t} \int_0^t e^{- \kbar X_s} d\beta_s^\C & e^{-\half \kbar X_t}
          \end{pmatrix}
          \in \left(SU_2^*\right)_\kbar \approx NA
          \ .
       $$
Note that in view of the convergence of $\left( M_{t/\hbar^2}(H) \ ; \ t\geq0 \right)$, one has the convergence in distribution of the commutative process $\left( K^{r,\hbar}_t := M_{t/\hbar^2}(K) \ ; \ t\geq0 \right)$ to $\left( e^{-rX_t} \ ; \ t\geq0 \right)$. 

In order to justify the convergence of the full non-commutative process $\left( g^{r,\hbar}_t \ ; \ t \geq 0 \right)$ to $\left( g^r_t \ ; \ t \geq 0 \right)$, it remains to prove that, jointly with  $\left( K^{r,\hbar}_t \ ; \ t\geq0 \right)$,
$$ 
 \left( 2r F_t^{r,\hbar} :=  2r M_{t/\hbar^2}( F ) \ ; \ t \geq 0 \right)
 \overset{\hbar \rightarrow 0}{\longrightarrow}
 \left( r e^{\half \kbar X_t} \int_0^t e^{- \kbar X_s} d\beta_s^\C \ ; \  t \geq 0\right) \ .
$$
Note that contrary to $\left( K^{r,\hbar}_t \ ; \ t\geq0 \right)$, the process $\left( F_t^{r,\hbar} \ ; \ t\geq0 \right)$ is not classical, since $M_n(F)$ does not commute with $M_k(F)$ for $k<n$ and it is not self-adjoint either.  Nevertheless, using the expression of coproducts for $s \in \N$:
\begin{align}
\label{eq:coproductF}
   M_{s+1}(F)
 = & 
   (M_s \otimes 1)( \Delta F)
 = M_s(F) \otimes K^{\half} + M_s(K^{-\half}) \otimes F \ ,
\end{align}
one sees that for all $n \in \N$:
\begin{align*}
  & M_n(F) \\
= & M_n(K^{\half}) M_n(K^{-\half}) M_n(F)\\
= & M_n(K^{\half}) \sum_{s=0}^{n-1} \left( M_{s+1}(K^{-\half}) M_{s+1}(F) - M_{s}(K^{-\half})M_{s}(F) \right) \\
\overset{\eqref{eq:coproductF}}{=}
  & M_n(K^{\half}) \sum_{s=0}^{n-1} \left[ M_{s+1}(K^{-\half}) \cdot \left( M_{s}(F) \otimes K^{\half}  
 + M_{s}(K^{-\half}) \otimes F  \right) 
  - M_{s}(K^{-\half})M_{s}(F) \right] \\
= & M_n(K^{\half}) \sum_{s=0}^{n-1}  \left[      M_{s}(K^{-\half}) M_{s}(F)
                             + M_{s}(K^{-1}) \otimes K^{-\half} F
                             - M_{s}(K^{-\half}) M_{s}(F) \right ]\\
= & M_n(K^{\half}) \sum_{s=0}^{n-1} M_{s}(K^{-1}) \otimes K^{-\half} F \ .
\end{align*}
This expression has been tailored to look similar to the computation of the coefficient $F(x)$ if $x \in \left( SU_2^* \right)_r$ is a product of $n$ matrices. After all, according to the orbit method, an $n$-fold tensor product should correspond to the quantization of an $n$-fold convolution. A crucial point is that the above expression is a (non-commutative) martingale transform or a discrete stochastic integral. We are now ready to prove:
\begin{proposition}
The following convergence in distribution holds:
$$
 \left( \left( K^{r,\hbar}_t, \ 2rF_t^{r,\hbar} \right) \ ; \  t \geq 0 \right)
 \overset{\hbar \rightarrow 0}{\longrightarrow}
 \left( \left( e^{-rX_t}, \ r e^{\half \kbar X_t} \int_0^t e^{- \kbar X_s} d\beta_s^\C \right)
        \ ; \  t \geq 0 \right) \ .
$$
\end{proposition}
\begin{proof}
Let $D$ be a subdivision of the segment $[0,t]$. The maximal meshsize is denoted $|D|$. Associated to $D$ and $t$, one defines the sums:
\begin{align*}
     F_t^{\hbar, D}
:= &  M_{t/\hbar^2}(K^{\half}) \sum_{[u,v] \in D} \sum_{s=\lfloor u/\hbar^2\rfloor+1}^{\lfloor v/\hbar^2\rfloor}  M_{u/\hbar^2}(K^{-1}) \otimes 1^{\otimes s-1-\lfloor u/\hbar^2\rfloor} \otimes K^{-\half} F\\
 = & M_{t/\hbar^2}(K^\half) \sum_{[u,v] \in D} M_{u/\hbar^2}(K^{-1}) \sum_{s=\lfloor u/\hbar^2\rfloor+1}^{\lfloor v/\hbar^2\rfloor}  1^{\otimes (s-1)} \otimes K^{-\half} F \ .
\end{align*}
Again,  the above sum can be understood as the definition of the non-commutative stochastic integral of the  adapted process $M_t(K^{-1})$ with respect to the process defined by the independent increments $1^{\otimes (s-1)} \otimes K^{-\half} F$. 

{\bf Step 1:} Using the weak convergence of Eq. \eqref{eq:walks_cv}, one has the weak convergence jointly with the other walks $( M_{\cdot}( H ), M_{\cdot}(\beta) )$:
$$ 
 \left( 2 F_t^{\hbar, D} ; t \geq 0 \right) 
 \overset{\hbar \rightarrow 0}{\longrightarrow}
 \left(
 e^{\half \kbar X_t} \sum_{[u,v] \in D}
 e^{- \kbar X_u} \left( \beta_v^\C - \beta_u^\C \right) ;  t \geq 0\right) \ .
$$

Indeed even when considering different times, $F_t^{\hbar, D}$ is  a polynomial function of known random walks. Since $M_{t/\hbar^2}(K)$ converges to $e^{-rX_t}$ and $K^{-\half}\Re F$ (resp. $K^{-\half}\Im F$) has the same spectral measure in $V^q(\hbar)$ than $\Re F$ (resp. $\Im F$), 
one has
$$
 \sum_{s=\lfloor u/\hbar^2\rfloor+1}^{\lfloor v/\hbar^2\rfloor} 1^{\otimes (s-1)} \otimes 2 K^{-\half} F \to \beta_v^\C - \beta_u^\C 
$$
as $\hbar\to0$, so 
we get the above convergence of $F_t^{\hbar, D}$  by 
simply applying a polynomial version of  the mapping theorem in the non-commutative setting.

{\bf Step 2:} Let us prove that for every fixed $\hbar_0$, $T>0$ and $p>1$
$$ \limsup_{|D| \rightarrow \infty} \sup_{\hbar \in [0, \hbar_0]} \sup_{t \in [0,T]} \Big\Vert F_t^{r,\hbar} - F_t^{\hbar, D} \Big\Vert_{p} = 0 \ ,$$
where $\Vert\cdot\Vert_p$ denotes the norm in $L^p(\Ac^{r, \hbar},\tau)$. 
We start by noticing that $M_{t/\hbar^2}(K)$ is uniformly bounded in all $L^p$ as (renormalized) Bernoulli walks have exponential moments of all order. Then using the Cauchy-Schwarz inequality:
\begin{align*}
     & \Big\Vert F_t^{r,\hbar} - F_t^{\hbar, D} \Big\Vert_{p} \\
\leq & \Big\Vert M_{t/\hbar^2}(K^\half) \Big\Vert_{2p}
      \Big\Vert  \sum_{[u,v] \in D} \sum_{s=\lfloor u/\hbar^2\rfloor+1}^{\lfloor v/\hbar^2\rfloor}  \left( M_{s-1}(K^{-1}) - M_{u/\hbar^2}(K^{-1}) \otimes 1^{\otimes (s-1-\lfloor u/\hbar^2\rfloor)} \right) \otimes K^{-\half} F \Big\Vert_{2p} \\
\ll  & \Big\Vert  \sum_{[u,v] \in D}  \sum_{s=\lfloor u/\hbar^2\rfloor+1}^{\lfloor v/\hbar^2\rfloor} \left( M_{s-1}(K^{-1}) - M_{u/\hbar^2}(K^{-1}) \otimes 1^{\otimes (s-1-\lfloor u/\hbar^2\rfloor)} \right) \otimes K^{-\half} F \Big\Vert_{2p} \ .
\end{align*}
The above sum is thus a martingale transform of the adapted process $\xi_s =   M_{s-1}(K^{-1}) - M_{u/\hbar^2}(K^{-1}) \otimes 1^{\otimes (s-1-\lfloor u/\hbar^2\rfloor)} $ with respect to the martingale difference process given by the increments $\left( 1^{\otimes s-1}\otimes K^{-1} F \ ; \ s \in \N \right)$. Note that the increment $1^{\otimes s-1}\otimes K^{-1} F$ is independent from $\xi_s$ since $\xi_s$ is identity on the $s$-th leg of the tensor product.

Now, thanks to the non-commutative Burkholder-Davis-Gundy inequality \cite[Theorem 2.1]{PX97} established by Pisier-Xu, its $L^{2p}$ norm  is controlled by the $L^p$ norms of non-commutative brackets:
\begin{align*}
&\Big\Vert F_t^{r,\hbar} - F_t^{\hbar, D} \Big\Vert_{2p}^2 \\
& \ll \max \left\{ 
\Big\Vert\sum_{[u,v] \in D}  \sum_{s=\lfloor u/\hbar^2\rfloor+1}^{\lfloor v/\hbar^2\rfloor} \left( M_{s-1}(K^{-1}) - M_{u/\hbar^2}(K^{-1}) \otimes 1^{\otimes (s-1-\lfloor u/\hbar^2\rfloor)} \right)^2 \otimes E K^{-1} F \Big\Vert_p^{\half} \right. \ , \\
&\left. 
\Big\Vert \sum_{[u,v] \in D}  \sum_{s=\lfloor u/\hbar^2\rfloor+1}^{\lfloor v/\hbar^2\rfloor} \left( M_{s-1}(K^{-1}) - M_{u/\hbar^2}(K^{-1}) \otimes 1^{\otimes (s-1-\lfloor u/\hbar^2\rfloor)} \right)^2 \otimes K^{-\half} FEK^{-\half} 
\Big\Vert_p^{\half} \right\}
\end{align*}
where we have used $\Vert A^{\half}\Vert_{2p}= \Vert  A\Vert_p^{\half}$. Using H\"older's inequality and the fact that $\Vert EK^{-1}F\Vert_p$ and $\Vert K^{-\half} FEK^{-\half}\Vert_p$ are of order $\hbar^2$, we get
\begin{align*}
\Big\Vert F_t^{r,\hbar} - F_t^{\hbar, D} \Big\Vert_{2p}^4 
 \ll  \sum_{[u,v] \in D}  \sum_{s=\lfloor u/\hbar^2\rfloor+1}^{\lfloor v/\hbar^2\rfloor}  \hbar^2
\Big\Vert \left( M_{s-1}(K^{-1}) - M_{u/\hbar^2}(K^{-1})\right) ^2 \Big\Vert_{2p} \ .
\end{align*}

Now, since $\left( M_s(K^{-1}) \ ; \ s \in \N \right)$ is a commutative process, one can use the classical Burkholder-Davis-Gundy inequality to obtain  the estimate
$$ \Vert ( M_{s-1}(K^{-1}) - M_{u/\hbar^2}(K^{-1})) ^2 \Vert_{2p}
   \ll (2r\hbar)^2 (s-1-\lfloor u/\hbar^2\rfloor) \ .$$
As such, we get
\begin{align*}
\Big\Vert F_t^{r,\hbar} - F_t^{\hbar, D} \Big\Vert_{2p}^4 
&\ll \sum_{[u,v] \in D}  \sum_{s=\lfloor u/\hbar^2\rfloor+1}^{\lfloor v/\hbar^2\rfloor}  \hbar^4  \left( s-1 -  \lfloor u/\hbar^2 \rfloor \right) \\
&=   \sum_{[u,v] \in D} \hbar^4 \sum_{s=1}^{\lfloor v/\hbar^2\rfloor-\lfloor u/\hbar^2\rfloor} s \\
&\ll \sum_{[u,v] \in D} \hbar^4 \left( \lfloor v/\hbar^2\rfloor-\lfloor u/\hbar^2\rfloor \right)^2 \\
&\ll |D| \,
\end{align*}
uniformly in $t\in[0,T]$ and $\hbar$.


{\bf Step 3:} By the same computation, in the easier classical setting, any fixed moments of the family
$\left(
 e^{\half \kbar X_t} \sum_{[u,v] \in D}
 e^{- \kbar X_u} \left( \beta_v^\C - \beta_u^\C \right) ;  t \geq 0\right)$
are close to the corresponding moments of the family
$\left(
 e^{\half \kbar X_t} \int_0^t e^{- \kbar X_s} d\beta_s^\C ;  t \geq 0\right)$,
the error depending only on $|D|$, this error vanishing as $|D| \rightarrow 0$. Also one can invoke the convergence in probability and boundedness in every $L^p$.

{\bf Step 4:} To conclude, any non-commutative moments of $\left(K_t^{r,\hbar},2rF_t^{r,\hbar} ; t\geq0\right)$ are 
\begin{itemize}
 \item close to the moments of $\left(K_t^{r,\hbar} , 2rF_t^{\hbar, D} ; t\geq0\right)$ via Step 2.
 \item which are close to the moments of $\left(e^{-r X_t}, re^{\half \kbar X_t} \sum_{[u,v] \in D}
 e^{- \kbar X_u} \left( \beta_v^\C - \beta_u^\C \right) ; t\geq0 \right) $ via Step 1.
 \item which are close to the moments of $\left(e^{-r X_t}, re^{\half \kbar X_t} \int_0^t e^{- \kbar X_s} d\beta_s^\C ; t\geq0\right)$ via Step 3. \qedhere
\end{itemize}
\end{proof}

Now let us deal with the semi-classical limit of the dynamic induced by the quantum Casimir. From Eq. \eqref{eq:def_Lambda}, we have:
\begin{align*}
\frac{2r \hbar}{e^{r\hbar}-e^{-r\hbar}}
\cosh \left( r \hbar + r \Lambda_t^{r, \hbar} \right)
& =
\half \left(  |2rF^{r,\hbar}_t|^2 + \left( e^{r\hbar}K^{r,\hbar}_t + e^{-r\hbar} (K^{r,\hbar}_t)^{-1} \right) \frac{2r\hbar}{(e^{r\hbar} - e^{-r\hbar})} \right) \ .
\end{align*}
As $\hbar\to0$, exactly as in Remark \ref{rmk:CasimirIsDressingInvariant}, we obtain the dressing orbit invariant in the limit
\begin{align*}
  & \lim_{h \rightarrow 0} \frac{2r \hbar}{e^{r\hbar}-e^{-r\hbar}} \cosh \left( r \hbar + r \Lambda_t^{r, \hbar} \right) \\
= & \half \left| r  e^{\half r X_t} \int_0^t e^{- \kbar X_s} d\beta_s^\C \right|^2 + \half \left(e^{-rX_t}+e^{rX_t}\right) \\
= & \half \Tr g^r_t {g^r_t}^\dagger \ .
\end{align*}
By inverting the $\cosh$, we obtain as announced:
$$
   \lim_{\hbar \rightarrow 0} \Lambda_t^{r, \hbar}
   = \frac{1}{r} \Argcosh\left[ \half \left| r  e^{\half r X_t} \int_0^t e^{- \kbar X_s} d\beta_s^\C \right|^2 + \cosh(rX_t) \right]
   = \Lambda_t^r  \ .
$$

\subsection{The law of \texorpdfstring{$\Lambda^r$}{the radial process} is independent from \texorpdfstring{the curvature $r>0$}{curvature} }
\label{subsection:rIndependence}

Now that we have developed all the machinery, the simplest proof of the independence of $\Lc\left( \Lambda^r \right)$ consists in putting together the two following facts:
\begin{itemize}
\item via subsection \ref{subsection:qPitman}, $\Lambda^r$ is the semi-classical limit of the Littlewood-Richardson dynamic.
\item as  described in subsection \ref{subsection:qLR}, the Littlewood-Richardson dynamic is the one measured by the Casimir. It depends on the relative dimensions of representations, upon tensoring by $V^q(\hbar)$ and breaking into irreducibles. As the tensor product rule \eqref{eq:LRrule} is independent of $q = e^{-r}$, so is the Littlewood-Richardson dynamic.
\end{itemize}

As long as we are concerned with the required proofs in the paper, we are done. We finish by revisiting the subtle argument of Bougerol and Jeulin \cite[Section 3]{BJ02} through the lens of spherical harmonic analysis and by explaining where the curvature parameter is hidden. The key ingredient for this reinterpretation is the ($r$-deformed) Harish-Chandra spherical function $\phi^{(r)}_z(\Lambda)$, which is defined as follows. Given an element $g^{r, \Lambda}$ uniformly distributed on the dressing orbit $\Oc_r(\Lambda)$, define:
$$ \phi^{(r)}_z(\Lambda)
   := \E\left( e^{r(z-1)H(g^{r, \Lambda})} \right)
   = \frac{e^{rz\Lambda}-e^{-rz\Lambda}}
          {z(e^{r\Lambda} -e^{-r\Lambda})}\ ,
$$
where the last equality is computed via Proposition \ref{proposition:uniformDressing}. Notice the trivial dependence in $r$ via a rescaling:
\begin{align}
\label{eq:trivialDep}
\phi^{(r)}_z(\Lambda) & = \phi^{(1)}_z(r \Lambda) \ .
\end{align}
From Theorem \ref{thm:staticLimits}, this is the semi-classical manifestation of (renormalized) characters 
$$ \frac{\Tr_{V^q(\Lambda^\hbar)}(q^{-z H})}
        {\Tr_{V^q(\Lambda^\hbar)}(q^{-H})}
$$
depending trivially on $q=e^{-r}$, itself a side-product of the rigidity of quantum groups.

\medskip

Now, let us sketch the argument before discussing it. They start by considering the canonical (bi-invariant) Brownian motion $B = \left( B_t \ ; \ t \geq 0 \right)$ on $\H^3 = SL_2(\C) / SU_2 \approx \left( SU_2^* \right)_{r=1}$. Its radial part has generator
$$ \half \partial_\Lambda^2 + \partial_\Lambda \log \sinh(\Lambda) \ \partial_\Lambda \ ,$$
whose heat kernel is diagonalized using the spherical functions $\left( \phi^{(r=1)}_z(\Lambda) \ ; \ z \in i \R \right)$.

Bougerol and Jeulin perform a first change of measure using $\phi_0^{(r=1)}(\Lambda) = \frac{\Lambda}{\sinh(\Lambda)}$ as a density. Under the new measure, $B$ becomes the process $B^0$. By Girsanov, the generator of its radial part transforms as:
$$ \half \partial_\Lambda^2 + \partial_\Lambda \log \sinh(\Lambda) \ \partial_\Lambda + \partial_\Lambda \log \phi_0^{(r=1)}(\Lambda) \ \partial_\Lambda
 = \half \partial_\Lambda^2 + \partial_\Lambda \log \Lambda \ \partial_\Lambda
 \ .$$
This is where we recognize the Bessel 3 process, which incidentally is the law of the radial part of the flat Brownian motion $x^0$ (denoted $Z$ in their paper).

Another change of measure ends up having a similar effect. Consider the map $\tau := e^{-H}$ on $NA \approx \H^3$. This time, the change of measure uses $\tau$ as density and aims at creating a special drift in the diagonal part of the symmetric space. Under the new measure, $B$ becomes the process $B^\tau$ and it is directly related to our curved Brownian motion via $B^\tau \eqlaw g^{r=1}$.

\cite[Theorem 3.1]{BJ02} invokes the invariance of the Brownian motion  $B$ in order to average $\tau$ over $SU_2$ and obtain $\phi_0$. This entails that the radial parts of $B^\tau$ and $B^0$ have the same distribution. All in all, this proves that $\Lambda^{r=1}$, the radial part of $g^{r=1}$, is a Bessel 3 process. Brownian scaling extends the result to all $r>0$ in \cite[Proposition 3.2]{BJ02}. Such a scaling argument however does not make apparent why $r>0$ is a curvature parameter.

\medskip

Let us now revisit the argument by introducing the parameter $r>0$ since the beginning. Here, the hyperbolic space $\H^3$ with curvature parameter $r$ is identified with $\left( SU_2^* \right)_r$. 

As detailed in subsection \ref{subsection:rIsCurvature}, introducing the parameter $r$ amounts to multiplying the sectional curvature by $r^2$. In turn, this is equivalent to rescaling the metric tensor by $r^{-2}$ via \cite[$\mathsection 3.20$, (iii)]{GHL90}. Indeed, recall that thanks to the Killing-Hopf theorem \cite[$\mathsection 3.28$]{GHL90}, in the context of constant sectional curvature, there is a one to one correspondence between curvature and the Riemannian metric tensor. Here the metric on $\H^3 = \left( SU_2^* \right)_{r=1}$ corresponds in spherical coordinates to \cite[p.152]{Helgason84}:
$$ ds^2 = d\Lambda^2 + \sinh(\Lambda)^2 d\sigma^2 \ ,$$
where $d\sigma^2$ is the Riemannian structure on the unit sphere, in the tangent space at the identity. By rescaling the curvature and then changing $\Lambda$ to $r\Lambda$, we obtain on $\H^3 = \left( SU_2^* \right)_r$:
$$ ds^2 = d\Lambda^2 + \left( \frac{\sinh(r\Lambda)}{r} \right)^2 d\sigma^2 \ .$$
Computing the Laplace-Beltrami operator (see \cite[Chapter II, $\mathsection 2.4$, Eq. (13)]{Helgason84}), we obtain in spherical coordinates:
$$ \half \partial_\Lambda^2
 + \partial_\Lambda \log \left( \frac{\sinh(r \Lambda)}{r} \right) \ \partial_\Lambda \ .
$$

From this point on, the argument of Bougerol and Jeulin carries verbatim and we obtain that the generator of $\Lambda^r$, the radial part of $g^{r}$, is:
$$ \half \partial_\Lambda^2
 + \partial_\Lambda \log \left( \frac{\sinh(r\Lambda)}{r} \right) \ \partial_\Lambda
 + \partial_\Lambda \log \phi_0^{(r)}(\Lambda) \ \partial_\Lambda
 = \half \partial_\Lambda^2 + \partial_\Lambda \log \Lambda \ \partial_\Lambda
 \ ,$$
which is the generator of a Bessel-3 process, independently of $r>0$. The Brownian rescaling argument is not needed anymore and we have indeed used curvature.

\medskip

In the end, we realize the key role of the simple dependence in $r>0$ of the spherical harmonic analysis in Eq. \eqref{eq:trivialDep}, which is nothing but the rigidity of quantum groups.
 
\appendix

\section{Non-commutative topological considerations}
\label{appendix:topologicalConsiderations}

Throughout the paper, we avoided the matters of completions of algebras. Since this is definitely not the main focus of the paper, we chose to postpone these topological considerations to this appendix.

Already, let us explain why, in most of our proofs, completing into a Von Neuman algebra would have been an overkill. In the entire paper, one can perform functional calculus at the level of the matrix algebras obtained after representation. For example, the first instance where we invoked elements belonging to a completion was Theorem \ref{thm:Biane} where the Casimir element $C_\gfrak$ is defined as a square-root. The simplest way of defining the object is the following. One has to remember that, in the computation of non-commutative moments, $C_\gfrak^2$ is represented as a Hermitian matrix before taking the trace. At that level, functional calculus is available for Hermitian matrices and the square-root is perfectly well-defined. 

Nevertheless, since the machinery exists (see \cite{PX97}, \cite{PX03} and references therein), it is possible to have a more intrinsic point of view and complete any of the algebras $\Ac$ considered in the paper into a $C^*$-algebra or a Von Neumann algebra. The general technique relies on the state $\tau$. In order to form a Banach algebra, we define the norm $d$ given by:
$$ \forall X \in \End(V), \ d(X) = \lim_{p \rightarrow \infty} \tau\left( (X X^*)^p \right)^{\frac1p} \ ,$$
then we complete $\Ac$ thanks to $d$. In order to form a Von Neumann algebra, there is the Gelfand-Naimark-Segal construction (GNS for short). This is done in several steps which we detail for the quantum group $\Uc_q^\hbar(\slfrak_2) = \Ac$. The other algebras considered in the paper are tensor products or degenerations.

\medskip

{\bf Step 1: Linear structure.} Consider a representation 
$$\rho: \Ac = \Uc_q^\hbar(\slfrak_2) \longrightarrow \bigoplus_{n=0}^\infty \End\left( V^q(n\hbar) \right) \ ,$$
which can be taken to be the faithful Peter-Weyl isomorphism. For shorter notations, let $\rho_{n\hbar} = \rho_{| \End(V^q(n\hbar))}$ be the restriction to the $n$-th component. Then define a trace via
$$ \tau(a) = \sum_{n=0}^\infty d_n \frac{\Tr( \rho_{n\hbar}(a) )}{\dim V^q(n\hbar)} \ ,$$
where 
$$ \sum_{n=0}^\infty d_n = 1 \ .$$
From this normalized trace $\tau$, one forms the scalar product $\langle a, b \rangle := \tau( a^{\dagger} b )$ and considers the completion into a Hilbert space $H$. Thus we embed the algebra into a linear space, but the multiplicative structure is missing.

\medskip

{\bf Step 2: Multiplicative structure.} The algebra $\Ac$ acts on $H$ via multiplication. Moreover, seeing $A \in \Ac$ as an operator on $H$, we write for $b \in H$:
$$ A(b) := A \times b \ ,$$
and from the Cauchy-Schwarz inequality:
$$ \| A (b) \|_{H}
 = \sqrt{ \tau\left[ \rho(A^\dagger A) b^\dagger b \right] }
 \leq \sqrt{ \tau( \rho(A^\dagger A) ) } \| b \|_{H} \ .
$$
Hence the $A$ acts necessarily as a bounded operator. As such the algebra $\Ac$ is identified to a subalgebra of $B(H)$, the algebra of bounded operators on $H$. 

\medskip

{\bf Step 3: Completions. } Upon completion for the operator norm, one obtains a $C^*$ algebra. Upon completion with respect to the weak-* topology, one obtains a Von Neumann  algebra. Going even further, one obtains the unbounded operators on $H$ affiliated to the algebra $\Ac$.

\bigskip

\noindent {\bf Acknowledgments.} During a conference in Reims, R.C. had the opportunity to meet Kirillov and present some of the ideas behind this paper, while still in their infancy. Kirillov said with a smile that these ideas were not completely absurd and that now, ``you have to work". Needless to say, one could not hope for better words of encouragement.


F.C. and R.C. acknowledge the support of the grant PEPS JC 2017 ``Quantum walks on quantum groups" and the grant ANR-18-CE40-0006 MESA funded by the French National Research Agency (ANR).

Finally, the authors express their gratitude to the anonymous referees for their useful comments.

\bibliographystyle{halpha}
\bibliography{qPitman}

\newcommand{\etalchar}[1]{$^{#1}$}
\begin{thebibliography}{BCDO09}

\bibitem[AAS19]{AAS19}
Anton Alekseev, Elizaveta Arzhakova, and Daria Smirnova.
\newblock {Stochastic differential equations for Lie group valued moment maps},
  2019, 1904.06758.

\bibitem[BBO05]{BBO}
Philippe Biane, Philippe Bougerol, and Neil O'Connell.
\newblock Littelmann paths and {B}rownian paths.
\newblock {\em Duke Math. J.}, 130(1):127--167, 2005.

\bibitem[BBO09]{BBO2}
Philippe Biane, Philippe Bougerol, and Neil O'Connell.
\newblock Continuous crystal and {D}uistermaat-{H}eckman measure for {C}oxeter
  groups.
\newblock {\em Adv. Math.}, 221(5):1522--1583, 2009.

\bibitem[BCDO09]{BCO09}
A~Ballesteros, E~Celeghini, and MA~Del~Olmo.
\newblock {Poisson--Hopf limit of quantum algebras}.
\newblock {\em Journal of Physics A: Mathematical and Theoretical},
  42(27):275202, 2009.

\bibitem[Bia91]{B91}
Philippe Biane.
\newblock Quantum random walk on the dual of {${\rm SU}(n)$}.
\newblock {\em Probab. Theory Related Fields}, 89(1):117--129, 1991.

\bibitem[Bia06]{B_suq}
Philippe Biane.
\newblock Le th\'eor\`eme de {P}itman, le groupe quantique {${\rm SU}_q(2)$},
  et une question de {P}. {A}. {M}eyer.
\newblock In {\em In memoriam {P}aul-{A}ndr\'e {M}eyer: {S}\'eminaire de
  {P}robabilit\'es {XXXIX}}, volume 1874 of {\em Lecture Notes in Math.}, pages
  61--75. Springer, Berlin, 2006.

\bibitem[Bia08]{B08}
Philippe Biane.
\newblock Introduction to random walks on noncommutative spaces.
\newblock In {\em Quantum potential theory}, pages 61--116. Springer, 2008.

\bibitem[Bia09]{B_crystals}
Philippe Biane.
\newblock From {P}itman's theorem to crystals.
\newblock In {\em Noncommutativity and singularities}, volume~55 of {\em Adv.
  Stud. Pure Math.}, pages 1--13. Math. Soc. Japan, Tokyo, 2009.

\bibitem[BJ02]{BJ02}
Philippe Bougerol and Thierry Jeulin.
\newblock Paths in {W}eyl chambers and random matrices.
\newblock {\em Probability Theory and Related Fields}, 124(4):517--543, 2002.

\bibitem[Chh13]{C13}
Reda Chhaibi.
\newblock Littelmann path model for geometric crystals, {W}hittaker functions
  on {L}ie groups and {B}rownian motion.
\newblock {\em PhD thesis. Université Paris VI. arXiv preprint
  arXiv:1302.0902}, 2013.

\bibitem[CP95]{CP95}
Vyjayanthi Chari and Andrew~N Pressley.
\newblock {\em A guide to quantum groups}.
\newblock Cambridge university press, 1995.

\bibitem[DM09]{DM09}
Persi Diaconis and Laurent Miclo.
\newblock On times to quasi-stationarity for birth and death processes.
\newblock {\em Journal of Theoretical Probability}, 22(3):558--586, 2009.

\bibitem[Dri88]{D88}
V.~G. Drinfel'd.
\newblock Quantum groups.
\newblock {\em Journal of Soviet Mathematics}, 41(2):898--915, Apr 1988.

\bibitem[GHL90]{GHL90}
Sylvestre Gallot, Dominique Hulin, and Jacques Lafontaine.
\newblock {\em Riemannian geometry}, volume~3.
\newblock Springer, 1990.

\bibitem[H{\etalchar{+}}02]{H02}
Thomas~Little Heath et~al.
\newblock {\em The works of Archimedes}.
\newblock Courier Corporation, 2002.

\bibitem[Hel84]{Helgason84}
Sigurdur Helgason.
\newblock {\em Groups \& geometric analysis: Radon transforms, invariant
  differential operators and spherical functions}, volume~1.
\newblock Academic press, 1984.

\bibitem[Kas12]{K12}
Christian Kassel.
\newblock {\em Quantum groups}, volume 155.
\newblock Springer Science \& Business Media, 2012.

\bibitem[KC02]{KC02}
Masaki Kashiwara and Charles Cochet.
\newblock {\em Bases cristallines des groupes quantiques}, volume~9.
\newblock Soci{\'e}t{\'e} math{\'e}matique de France, 2002.

\bibitem[Kir99]{Kir99}
A.~A. Kirillov.
\newblock Merits and demerits of the orbit method.
\newblock {\em Bull. Amer. Math. Soc. (N.S.)}, 36(4):433--488, 1999.

\bibitem[KS97]{K97}
Yvette Kosmann-Schwarzbach.
\newblock Lie bialgebras, {P}oisson {L}ie groups and dressing transformations.
\newblock In {\em Integrability of nonlinear systems}, pages 104--170.
  Springer, 1997.

\bibitem[KS12]{KS12}
Anatoli Klimyk and Konrad Schm{\"u}dgen.
\newblock {\em Quantum groups and their representations}.
\newblock Springer Science \& Business Media, 2012.

\bibitem[KT00]{KT00}
Christian Kassel and Vladimir Turaev.
\newblock Biquantization of {L}ie bialgebras.
\newblock {\em Pacific J. Math}, 195(2):297--369, 2000.

\bibitem[Lit95a]{LittICM95}
Peter Littelmann.
\newblock The path model for representations of symmetrizable kac-moody
  algebras.
\newblock In {\em Proceedings of the International Congress of Mathematicians},
  pages 298--308. Springer, 1995.

\bibitem[Lit95b]{Litt95}
Peter Littelmann.
\newblock Paths and root operators in representation theory.
\newblock {\em Annals of Mathematics}, pages 499--525, 1995.

\bibitem[LLP12]{LLP}
C{\'e}dric Lecouvey, Emmanuel Lesigne, and Marc Peign{\'e}.
\newblock Random walks in {W}eyl chambers and crystals.
\newblock {\em Proc. Lond. Math. Soc. (3)}, 104(2):323--358, 2012.

\bibitem[LLP13]{LLP2}
C{\'e}dric Lecouvey, Emmanuel Lesigne, and Marc Peign{\'e}.
\newblock {Conditioned random walks from Kac-Moody root systems}.
\newblock {\em Transactions of the AMS (accepted)}, pages 1--30, 2013,
  arXiv:1306.3082.

\bibitem[Maj00]{Majid00}
Shahn Majid.
\newblock {\em Foundations of quantum group theory}.
\newblock Cambridge university press, 2000.

\bibitem[MO04]{MO04}
Hiroyuki Matsumoto and Yukio Ogura.
\newblock Markov or non-{M}arkov property of {$cM-X$} processes.
\newblock {\em J. Math. Soc. Japan}, 56(2):519--540, 04 2004.

\bibitem[Nik06]{N06}
Ashkan Nikeghbali.
\newblock An essay on the general theory of stochastic processes.
\newblock {\em Probab. Surveys}, 3:345--412, 2006.

\bibitem[Pit75]{P75}
James~W Pitman.
\newblock One-dimensional brownian motion and the three-dimensional {B}essel
  process.
\newblock {\em Advances in Applied Probability}, 7(3):511--526, 1975.

\bibitem[Pra05]{Survey05}
Victor Prasolov.
\newblock {\em Surveys in modern mathematics}, volume 321.
\newblock Cambridge University Press, 2005.

\bibitem[PX97]{PX97}
Gilles Pisier and Quanhua Xu.
\newblock Non-commutative martingale inequalities.
\newblock {\em Communications in mathematical physics}, 189(3):667--698, 1997.

\bibitem[PX03]{PX03}
Gilles Pisier and Quanhua Xu.
\newblock Non-commutative {$L^p$}-spaces.
\newblock {\em Handbook of the geometry of Banach spaces}, 2:1459--1517, 2003.

\bibitem[RP81]{RP81}
L~Chris~G Rogers and JW~Pitman.
\newblock Markov functions.
\newblock {\em The Annals of Probability}, pages 573--582, 1981.

\bibitem[RY13]{RY13}
Daniel Revuz and Marc Yor.
\newblock {\em Continuous martingales and Brownian motion}, volume 293.
\newblock Springer Science \& Business Media, 2013.

\bibitem[Sun16]{Yi16}
Yi~Sun.
\newblock {A new integral formula for Heckman--Opdam hypergeometric functions}.
\newblock {\em Advances in Mathematics}, 289:1157--1204, 2016.

\bibitem[SW15]{SW15}
Anne~V Shepler and Sarah Witherspoon.
\newblock {Poincar{\'e}-Birkhoff-Witt theorems}.
\newblock In {\em Commutative algebra and noncommutative algebraic geometry.
  Vol. I}, pages 259--290, 2015.

\bibitem[VK]{klimyk-vilenkin}
N~Ja Vilenkin and Anatoli{\u\i} Klimyk.
\newblock {\em Representation of Lie groups and special functions: Volume 3:
  Classical and quantum groups and special functions}.

\end{thebibliography}

\end{document}